\documentclass[11pt]{amsart}
\usepackage{amsmath,amscd,amssymb}
\usepackage[mathscr]{eucal}
\usepackage{amsmath}
\usepackage{enumerate}
\usepackage{color}

 \newtheorem{thm}{Theorem}[section]
 \newtheorem{cor}[thm]{Corollary}
 \newtheorem{lemma}[thm]{Lemma}
 \newtheorem{prop}[thm]{Proposition}

 \theoremstyle{definition}

  \newtheorem{exas}[thm]{Examples}
 \theoremstyle{remark}
 \newtheorem{rem}[thm]{Remark}
 \newtheorem{rems}[thm]{Remarks}

\newtheorem{remark}[thm]{Remark}

\numberwithin{equation}{section}

\def\R{\mathbb{R}}
\def\C{\mathbb{C}}
\def\N{\mathbb{N}}
\def\Z{\mathbb{Z}}
\def\a{\alpha}
\def\b{\beta}
\def\ep{\varepsilon}
\def\l{\lambda}
\def\B{L}
\def\Re#1{\operatorname{Re}#1}
\def\Sect#1{\operatorname{Sect}#1}

\def\Bes{\mathcal{B}}
\def\Bov{\mathcal{E}}
\def\Bq{{\Bes_0}}
\def\fti{\mathcal{F}^{-1}}
\def\LT{\mathcal{LM}}
\def\lt{\mathcal{L}}
\def\ssp{\mathcal{G}}
\def\RR{\mathrm{R}}
\def\supp{\operatorname{supp}}
\def\sp{\operatorname{sp}}

\begin{document}

\title[A Besov algebra calculus]
{A Besov algebra calculus for generators of operator semigroups and related norm-estimates}

\author{Charles Batty}
\address{St. John's College\\
University of Oxford\\
Oxford OX1 3JP, UK
}

\email{charles.batty@sjc.ox.ac.uk}

\author{Alexander Gomilko}
\address{Faculty of Mathematics and Computer Science\\
Nicolas Copernicus University\\
Chopin Street 12/18\\
87-100 Toru\'n, Poland
}

\email{alex@gomilko.com}

\author{Yuri Tomilov}
\address{
Institute of Mathematics\\
Polish Academy of Sciences\\
\'Sniadeckich 8\\
00-956 Warsaw, Poland
}

\email{ytomilov@impan.pl}

\begin{abstract}
We construct a new bounded functional calculus for the generators of bounded semigroups on Hilbert spaces
and generators of bounded holomorphic semigroups on Banach spaces.
The calculus is a natural (and strict) extension of the classical Hille-Phillips functional calculus,
and it is compatible with the other well-known functional calculi.
It satisfies the standard properties of functional calculi,
provides a unified and direct approach to a number of norm-estimates in the literature,
and allows improvements of some of them.
\end{abstract}



\thanks{This work was partially supported financially by a Leverhulme Trust Visiting Research Professorship and an NCN grant UMO-2017/27/B/ST1/00078, and inspirationally by the ambience of the Lamb \& Flag, Oxford.  \\ We are very grateful to two referees for helpful comments on the first version of this paper, in particular those relating to Subsections 4.4 and 4.5.}

\date\today

\maketitle
\tableofcontents

\section{Introduction}

Given a linear operator $A$ on a Banach space $X$, a fundamental matter in operator theory is to define
a functional calculus for $A$ and to get reasonable norm-estimates for functions of $A$.
A rich enough functional calculus for $A$ yields various spectral decomposition properties
and leads to a detailed spectral theory. One well-known instance of that is the
functional calculus for normal operators on Hilbert spaces.

A common procedure to create a functional calculus for an unbounded operator $A$ is to define a function algebra $\mathcal A$ on the spectrum $\sigma (A)$ of $A$ and a homomorphism from $\mathcal A$ into the space of bounded linear operators on $X$ with suitable continuity properties and norm-estimates.  One instance of this is the classical Riesz-Dunford calculus given by  Cauchy's reproducing formula
\begin{equation}\label{cauchy}
\mathcal A \ni f \,\, \mapsto \,\, f(A):=\frac{1}{2\pi i}\int_\Gamma f(\lambda)(\lambda -A)^{-1}\, d\lambda,
\end{equation}
where $\mathcal A$ can be chosen to be the algebra $\operatorname{Hol}\,(\mathcal O)$ of  functions
holomorphic in a neighbourhood $\mathcal O$ of $\sigma(A)\cup\{\infty\}$, and $\Gamma \subset \mathcal O$ is an appropriate contour.  Unfortunately, for unbounded $A$ \eqref{cauchy} produces bounded operators $f(A)$
only in very restricted settings, and thus \eqref{cauchy} is only a starting point for the construction of an extended functional calculus, for example, for sectorial or similar classes of operators $A$.  Here sharp norm-estimates are hardly available.

The major problem in defining a functional calculus as in \eqref{cauchy} (or similar contexts) is that one has to start from appropriate norm bounds for the resolvent of $A$, thus having some information about at least one function  of $A$ in advance.   A very natural and useful class of unbounded operators leading to a holomorphic functional calculus via \eqref{cauchy} is the class of sectorial operators $A$ on Banach spaces having their spectrum in $\Sigma_\theta \cup \{0\}$ where $\theta \in (0,\pi)$ and $\Sigma_\theta :=\{\lambda \in \mathbb C\setminus \{0\}: |\arg(\lambda)|<\theta\}$,  and allowing a linear resolvent estimate $\|(\lambda-A)^{-1}\| \le C/|\lambda|$ for  $\lambda \in \mathbb C \setminus \Sigma_{\theta}\cup\{0\}$.
A homomorphism from the algebra $\mathcal A_\sigma$ of holomorphic functions on $\Sigma_\sigma, \sigma >\theta$, decaying sufficiently fast at zero and at infinity,
\begin{equation}\label{regular}
\mathcal A_\sigma:=\left \{ f:\Sigma_\sigma \to \mathbb C\,\,\, \text{holomorphic and}\, |f(\lambda)|\le C \frac{|\lambda|^{\ep}}{1+|\lambda|^{2\ep}}\right \}
\end{equation}
(where $\ep, C >0$ may  depend on $f$), is only the starting point here. The holomorphic functional calculus for $A$ is constructed as a mapping taking values in the closed operators on $X$ and extending the  homomorphism from \eqref{cauchy} in a canonical way.
The so-called holomorphic (or sometimes McIntosh) functional calculus for sectorial operators became an indispensable tool
in applications of operator theory to PDEs  and harmonic analysis. For more information, we refer to \cite{KuWe}, \cite{HaaseB} and \cite{Denk}.

If $-A$ generates a $C_0$-semigroup $(e^{-tA})_{t \ge 0}$ on $X$, then
different reproducing formulas with  the semigroup
can be used. One of them leads to the classical Hille-Phillips (or HP-) functional calculus from the 1950s, see e.g.\ \cite[Chapter XV]{HP}. The calculus is given by a bounded homomorphism  from the space of bounded Borel measures $M(\R_+)$ into $\B(X)$:
 \begin{equation}\label{laplace}
M(\R_+) \ni \mu \,\, \mapsto \,\, \mathcal L \mu (A)= \int_{\R_+} e^{-tA} \,d\mu(t),
\end{equation}
so \eqref{laplace} is an operator counterpart of the Laplace transform $\mathcal L \mu$ of $\mu \in
M(\R_+)$.

Despite its very direct nature, the HP-calculus has proved to be
crucial in many areas of analysis, including
probability theory, approximation theory, theory of Banach algebras, spectral theory, etc.
It is a base for so-called transference principles and their applications in harmonic analysis and operator theory.
While classical aspects of the transference techniques can be found in \cite{Coifman}, its modern treatment and applications to semigroup
theory are contained in \cite{Haase} (see also \cite{Merdy}).

In most approaches to functional calculi one starts with a reproducing formula for sufficiently regular functions, such as \eqref{regular}, and then extends it via a regularisation procedure to include
more singular functions of interest, for example $z^\alpha$ and $\log z$. The reproducing formula determines the function algebra for the extended functional calculus, and thus it is basic for the calculus construction.  Using the regularisation idea, the HP-calculus was extended
by Balakrishnan \cite{Balakrishnan}, while the holomorphic functional calculus for unbounded sectorial operators
has been developed by McIntosh (preceded by pioneering work of Bade), and then further developed and extended to other geometries (strip-type, half-plane type, etc) by a number of mathematicians.

Other functional calculi include the Hirsch calculus where a reproducing formula arises as the Stieltjes-type representation formula for complete Bernstein functions; the Bochner-Phillips calculus, where following Bochner's representation formula for Bernstein functions, a measure $\mu$ in \eqref{laplace} is replaced
by  a convolution semigroup of measures $(\mu_t)_{t \ge0}$, and a Bernstein function $f(A)$ of $A$ arises as the negative generator of a strongly continuous semigroup $(\int_{0}^{\infty} e^{-sA} \, d\mu_t(s))_{t \ge 0}$; and the Davies--Dyn'kin (or Helffer--Sj\"ostrand) calculus based on the Cauchy-Green reproducing formula and thus even allowing non-holomorphic functions.  See \cite[Chapter 13]{Schill}, and also \cite{GT15} and \cite{BGTMZ}, for more on the Hirsch and Bochner-Phillips calculi, and \cite{Davies} and \cite{ASS} for the third calculus.

Having defined $f(A)$ in any calculus, it is natural and useful  to look for sharp norm-estimates for $f(A)$ in terms of $f$.   Note that the HP-calculus produces a norm-estimate in terms of the representing measure $\mu$ rather than $\mathcal L \mu$, and this can hardly be optimal apart from a few very special cases.
The question for which $A$ the holomorphic calculus is bounded  in the sense that $\|f(A)\|\le C_A \|f\|_\infty$ for all $f \in \mathcal A_\sigma$, with $C_A$ depending only on $A$, is of major importance, since the uniform estimate allows one to avoid a drawback of ``a priori'' estimates in reproducing formulas.
The boundedness of $H^\infty$-calculus is linked to square function estimates on Hilbert and Banach spaces,
and thus eventually to hard problems in harmonic analysis involving boundedness of singular integrals.
We refer to \cite{Denk}, \cite{Haase} and \cite{KuWe}  for more details and pertinent comments.

While the boundedness of $H^\infty$-calculus is very useful, for example in the study of maximal regularity properties in PDE theory, it imposes stringent conditions on $A$
that are hard to verify in the abstract context.
(Nevertheless it is well known that many concrete differential operators do have that property.)  This remark extends to the other calculi, including the HP-calculus where the uniform estimate $\|\mathcal L \mu (A)\|\le C_A \|\mathcal \mu\|$ allows one to deal with norm bounds for functions
of $A$ as if $A$ generates a unitary $C_0$-group.
Examples of such a situation are generators of $C_0$-contraction semigroups,
where the uniform estimate for $\lt\mu(A)$ arises from a unitary dilation.
This suggests an important task of identifying classes of functions (or even single functions),
and classes of semigroup generators, such that the corresponding functional calculus satisfies operator norm-estimates which are weaker than those given by the $H^\infty$-norm, but as close as possible to the $H^\infty$-norm.  The problem of obtaining sharp norm-estimates for functions of semigroup generators has been addressed in a series of recent papers  \cite{Haase},  \cite{HRoz},  \cite{Sch1}, \cite{V1}, and \cite{Zwart}.

Ideally, the functional calculus should be defined for substantial classes of operators and functions,
take values in the algebra of bounded operators, and provide sharp enough estimates for the operator norms.
This paper offers a new functional calculus that fits this description to a large extent.
More precisely, let $-A$ be  the generator of a bounded $C_0$-semigroup on a Hilbert space $X$
or  the generator of a bounded holomorphic semigroup on a Banach space $X$.
Then for all $x \in X$ and $x^* \in X^*$ its weak resolvent $\langle (\cdot + A)^{-1}x,  x^*\rangle$ belongs to the space $\mathcal E$ given by
\begin{equation}
\mathcal E:= \left\{g \in \operatorname{Hol}(\mathbb C_+): \|g\|_{\Bov_0} :=\sup_{\a >0} \a\int_{\mathbb R} |g'(\a+i\b)|\, d\b < \infty \right\}.
\end{equation}
Note that $\|\cdot\|_{\Bov_0}$ is a seminorm on $\Bov$ vanishing on the constant functions.
As observed in \cite{V1}, a Banach algebra $\mathcal B$ can be defined as
\begin{equation}  \label{bdef0}
\mathcal B:= \left\{f \in \operatorname{Hol}\, (\mathbb C_+): \|f\|_\Bq:=\int_{0}^{\infty} \sup_{\b \in \R }|f'(\a+i\b)| \, d\a <\infty \right \},
\end{equation}
with the norm $\|f\|_{\Bes}:=\|f\|_\infty+ \|f\|_\Bq$.  Then  $\mathcal E$ can be paired  with $\mathcal B$ via a duality $\langle\cdot, \cdot \rangle_{\Bes}$ given by
\begin{equation}
\langle g, f \rangle_{\Bes}: = \int_{0}^{\infty} \a \int_{-\infty}^{\infty} g'(\a-i\b) f'(\a+i\b)\, d\b\,d\a, \qquad g \in \Bov, f \in \Bes.
\end{equation}

The Laplace transforms of measures $\mu \in M(\R_+)$ belong to $\Bes$, and one has a  reproducing formula of Cauchy type
\begin{equation}\label{lapmeas}
\mathcal L \mu (z)=\mathcal{L}\mu(\infty)+\frac{2}{\pi}\langle \left (\cdot+z)^{-1}, \mathcal L \mu \right \rangle_{\mathcal B}.
\end{equation}
Consequently, for $A$ as above, we define a mapping $\Phi_A : \mathcal B \to  L(X,X^{**})$, $\Phi_A(f)=f(A)$,
as a $w^*$-integral:
\begin{align}\label{bcalculus}
\langle f(A)x,x^*\rangle=f(\infty)+ \frac{2}{\pi}\left \langle \langle (\cdot +A)^{-1}x, x^* \rangle , f \right \rangle_{\mathcal B}
\end{align}
for all  $x \in X$ and $x^* \in X^*$.
However it is not immediately clear that it yields a functional calculus for $A$, since
   $\mathcal L(M(\R_+))$ is not dense in $\mathcal B$.
 The duality $\langle\cdot, \cdot \rangle_{\mathcal B}$ is only partial, since
 the spaces $\mathcal E$ and $\mathcal B$ have rather complicated structures;  see Section \ref{besov} for  details of the spaces.
To make our observations rigorous and to establish functional calculus properties of $\Phi_A$,
we prove several intermediate statements of independent interest leading to the following theorem which summarises the main results of this paper (see Section \ref{b-fc}).

\begin{thm}\label{homomor_intro}
Let $-A$ be either the generator of a bounded $C_0$-semigroup on a Hilbert space $X$
or of a (sectorially) bounded holomorphic $C_0$-semigroup on a Banach space $X$.  Then the following hold.
\begin{enumerate}[\rm a)]
\item
The formula \eqref{bcalculus} defines a bounded algebra homomorphism
\[
\Phi_A: \Bes \to  L (X),\qquad \Phi_A (f):=f(A),
\]
and
\begin{equation}\label{b_estim}
\|f(A)\|\le C_A \|f\|_{\mathcal B}
\end{equation}
for a constant $C_A$ depending only on $A$.

\item The $\Bes$-calculus defined in {\rm  a)} (strictly) extends the HP-calculus, and it is compatible with
the holomorphic functional calculi for sectorial and half-plane type operators.

\item The spectral inclusion (spectral mapping, in the case of bounded holomorphic semigroups) theorem and a Convergence Lemma hold for $\Phi_A$.
\end{enumerate}
\end{thm}

Theorem \ref{homomor_intro} can be applied to $C_0$-semigroups which are not bounded, since any $C_0$-semigroup can be made bounded by rescaling, and most norm-estimates are of interest even for semigroups which decay exponentially.

The use of Besov functions for functional calculus goes back to Peller's foundational paper \cite{Peller} in the discrete case.  Peller defined and explored the functional calculus for power bounded operators on a Hilbert space $X$, based on the space $B^0_{\infty,1}(\mathbb D)$ which is the analogue of $\mathcal B$ for the unit disc.  He proved a counterpart of \eqref{b_estim} for power bounded operators on $X$, and also obtained several generalizations of \eqref{b_estim}  using several specific algebras related to $B^0_{\infty,1}(\mathbb D)$.  This line of research was continued, in particular, in  \cite{Haase}, \cite{Vitse}, \cite{Vitse2}, and \cite{Vitse1}, where similar classes of operators have been considered.  The polynomials are dense in $B^0_{\infty,1}(\mathbb D)$, and this simple but important fact greatly simplifies the  construction in \cite{Peller}.   Unfortunately, in the setting of $C_0$-semigroups, there is only a partial replacement for polynomials provided by those entire functions of exponential type which are bounded on $\C_+$. Consequently we have to use a different approach via a duality. It leads to the reproducing formula (\ref{bcalculus}) which is apparently new and crucial for the calculus bound (\ref{b_estim}).

There was a substantial contribution to this topic in the PhD thesis of S. White \cite{White} from 1989.  He adapted a large part of Peller's estimates \cite{Peller} to the more demanding context of semigroup generators.  Unfortunately,  the results were not published in journals, and thus were overlooked  by the mathematical community until very recently.  Employing the ideas from \cite{Peller} in the semigroup setting, a calculus for the generators of bounded holomorphic semigroups on a Banach space $X$ was constructed  by Vitse in \cite{V1}. Vitse's results were put in a wider context of transference methods by Haase in \cite{Haase}. In \cite{Haase}, functional calculus estimates were reduced to estimates of the Fourier multiplier norms, and the semigroup was not assumed to be holomorphic.  Instead there were additional geometric assumptions on $X$, and so the most complete results were for Hilbert spaces.   The question of constructing a full Besov functional calculus for generators of bounded semigroups on Hilbert space and its compatibility to other calculi was raised explicitly in \cite[p.2992]{Haase}.  Other contributions were made in \cite{Haase}, \cite{HRoz}, \cite{Sch1}, \cite{White} and \cite{Zwart}, where operator norm-estimates were obtained for particular classes of Besov functions for $A$ as in Theorem \ref{homomor_intro}.
The emphasis in those papers was on functions of the form  $\mathcal L \mu$ for $\mu \in M(\R_+)$ treated by either the HP-calculus or the holomorphic functional calculus, and the arguments there relied on the Littlewood-Paley decomposition of Besov functions (see Section \ref{appx} for more on that).
Thus the generality of Theorem \ref{homomor_intro} was out of reach. Related functional calculi for generators of $C_0$-groups of polynomial growth
(thus having their spectrum on $i\mathbb R$) and for sectorial operators of zero angle were studied in \cite{Cowling}, \cite{Kriegler} and \cite{KrieglerWeis}, again by means of Fourier analysis.

Once the ${\mathcal B}$-calculus has been established, it leads to a number of sharp norm-estimates.
The estimates are direct consequences of \eqref{b_estim} and certain elementary (but not straightforward) estimates for norms of functions in $\mathcal B$.  As a sample result, we formulate
the next statement which generalises results in \cite[Theorem 1.1(c)]{HRoz} (see Corollary \ref{CH}).

\begin{cor}
Let $f \in H^\infty (\mathbb C_+)$ be such that
\begin{equation*}
\int_{0}^{\infty}\frac{h(s)}{1+s}\, ds<\infty,
\end{equation*}
where $h(s):= \operatorname{ess\,sup}_{|t| \ge s} |f(it)|$ and $f(i\cdot)$ is the boundary value of $f$.
If $-A$ is the generator of an exponentially stable $C_0$-semigroup on a Hilbert space $X$, then $f(A) \in L (X)$ and
\begin{equation*}
\|f(A)\| \le C_A \int_{0}^{\infty}\frac{h(s)}{1+s}\, ds.
\end{equation*}
Here $C_A$ is a constant which depends on $A$, but not on $f$.
\end{cor}

For generators of bounded semigroups on Hilbert space,
other applications include sharp norm-estimates for functions which have  holomorphic extensions to larger half-planes,
for the function $e^{-1/z}$ and its regularisations, and for powers of Cayley transforms $(z-1)^n/(z+1)^n, n \in \mathbb N$.
In the context of generators of bounded holomorphic semigroups, we give sharp constants for the norms of $\mathcal B$-functions of the generators of bounded holomorphic semigroups,
 and provide sharp estimates for resolvents of Bernstein functions for the same class of semigroup generators.
All of these applications have substantial motivations, and they recover and/or improve notable known results,
in particular from  \cite{Haase},  \cite{HRoz},  \cite{Sch1},  \cite{V1},  \cite{Zw05}, and \cite{Zwart}.
We refer to Section \ref{apps} for a fuller discussion of the norm-estimates provided by the $\mathcal B$-calculus.  However, we emphasize that the most attractive feature of the $\mathcal B$-calculus is not the estimates as such, but the fact that all of them can be obtained in a single, technically simple, manner.

Although we use the terms \emph{Besov algebra} for the Banach algebra $\Bes$, and \emph{Besov functions} for its elements, we use only the definition of $\Bes$ in (\ref{bdef0}), together with techniques from complex function theory, Fourier analysis, and operator theory.  In particular, we do not use the Littlewood-Paley decompositions that frequently appear in connection with Besov spaces.  In an appendix (Section \ref{appx}), we show how $\Bes$ (or more precisely, a subspace of codimension $1$), is a realization of the analytic part of a conventional Besov space on $\R$ defined via the Littlewood-Paley dyadic decomposition.   In \cite{V1}, the algebra $\mathcal B$ arises  as a half-plane realization of the analytic Besov algebra $B^0_{\infty,1}(\mathbb C_+)$, and the space $\mathcal E$ is shown there to contain a continuously embedded analytic Besov space $B^0_{1,\infty}(\mathbb C_+)$.

\subsection*{Notation}
Throughout the paper, we shall use the following notation:
\begin{enumerate}[\phantom{X}]
\item $\R_+ :=[0,\infty)$,
\item $\C_+ := \{z \in\C: \Re z>0\}$, $\overline{\C}_+ = \{z \in\C: \Re z\ge0\}$,
\item $\Sigma_\theta := \{z\in\C: z \ne 0, |\arg z|<\theta\}$ for $\theta \in (0,\pi)$; $\Sigma_0 := (0,\infty)$,
\item $\mathrm{R}_a := \{z\in\C: \Re z > a\}$.
\item
\item $\supp(f)$ denotes the support of a function or distribution $f$ on $\R$,
\item For $f : \C_+ \to \C$, we write
\[
f(\infty) = \lim_{\Re z \to \infty} f(z)
\]
if this limit exists in $\C$.
\end{enumerate}
For $a \in \overline{\C}_+$, we define functions on $\C_+$ by
\[
e_a(z) = e^{-az}, \; r_a(z) = (z+a)^{-1}.
\]

\noindent
We use the following notation for spaces of functions or measures, and transforms, on $\R$ or $\R_+$:
\begin{enumerate}[\phantom{X}]
\item $\mathcal{S}(\R)$ denotes the Schwartz space on $\R$,
\item $\operatorname{BUC}(\R)$ denotes the space of bounded, uniformly continuous, functions on $\R$, with the sup-norm,
\item  $\operatorname{Hol}(\Omega)$ denotes the space of holomorphic functions on an open subset $\Omega$ of $\C_+$,
\item $M(\R)$ denotes the Banach algebra of all bounded Borel measures on $\R$ under convolution, and $M(\R_+)$ denotes the corresponding algebra for $\R_+$.   We shall identify $L^1(\R_+)$ with a subalgebra of $M(\R_+)$ in the usual way.
\item $\lt$ denotes the Laplace transform applied to distributions, measures or functions on $\R_+$,
\item $\mathcal{F}$ denotes the Fourier transform on $\R$.  For $f \in L^1(\R)$,
\[
(\mathcal{F}f)(s) := \int_\R f(t)e^{-ist} \,dt.
\]
We shall also consider $\mathcal{F}$ applied to measures and distributions on $\R$, and
$\mathcal{F}^{-1}$ will be the inverse Fourier transform on $\R$.
\end{enumerate}

\noindent
For a Banach space $X$, $L(X)$ denotes the space of all bounded linear operators on $X$. The domain, spectrum and resolvent set of an (unbounded) operator $A$ on $X$ are denoted by $D(A)$, $\sigma(A)$ and $\rho(A)$, respectively.

\section{The Besov algebra, related spaces and a duality}  \label{besov}

In this section we present material which we shall need for what follows.  Some of this material is quite standard, but we think it will be helpful to readers with various backgrounds if the material is collected together.

\subsection{The algebra $H^\infty(\C_+)$}
Let $H^\infty(\C_+)$ be the Banach algebra of bounded holomorphic functions on $\C_+$, equipped with the supremum norm
\[
\|f\|_\infty := \sup \{ |f(z)| : z \in \C_+\}.
\]
Each function $f \in H^\infty(\C_+)$ has a boundary function on $i\R$, given by
\[
f^b(s) = f(is):= \lim_{t \to 0+} f(t+is) \quad\text{a.e.}
\]
We need the Poisson kernel and Poisson semigroup
\[
P_t(y) = \dfrac {t}{\pi(t^2+y^2)}, \quad P(t)g = P_t * g,  \quad t>0,\, y \in \R,\, g \in L^\infty(\R).
\]
We will use the following standard facts.

\begin{lemma} \label{ho}
Let $f \in H^\infty(\C_+)$.
\begin{enumerate}[\rm1.]
\item \label{ho1} $\displaystyle f(t+is) =  (P(t) f^b)(s) = \frac{1}{\pi} \int_\R \frac{t f^b(y)}{t^2 + (y-s)^2} \,dy$,
\item \label{ho2} $\|f\|_\infty = \|f^b\|_{L^\infty(\R)}$ \; (maximum principle);
\item \label{derbd} There are absolute constants $C_n$ such that $|f^{(n)}(z)| \le \dfrac{C_n \|f\|_\infty}{(\Re z)^n}$, $n \in\N$.   One may take  $C_1=1/2$ and $C_2 = 2/\pi$.
\item \label{ho4} Let  $a>0$, $b>0$, and $n> m \ge 0$.  Then
\[
\sup_{\Re z = a+b} |f^{(n)}(z)| \le \frac{C_{n-m}}{b^{n-m}}  \sup_{\Re z = a} |f^{(m)}(z)|.
\]
\end{enumerate}
\end{lemma}

\begin{proof}   (\ref{ho1}) and (\ref{ho2}) are standard facts; (\ref{derbd}) is easily seen from Cauchy's integral formula.  Then (\ref{ho4}) can be deduced by applying (3) to $f^{(m)}(z+a)$.
\end{proof}

The proof of (\ref{derbd}) via Cauchy's integral formula around a circular contour shows that $C_n$ can be taken to be $n!$.   However if $f \in H^\infty(\C_+)$ and $f(\infty):=\lim_{\Re z \to \infty} f(z)$ exists, then $f$ satisfies the Cauchy integral representation
\[
f(z) = \frac{f(\infty)}{2} - \frac{1}{2\pi} \int_\R \frac{f(is)}{is-z} \,ds,
\]
as a principal value integral.  This may be seen by applying Cauchy's integral formula around large semi-circles.  Similarly, or by differentiating the formula
\[
f(z) - f(1) = - \frac{1}{2\pi} \int_\R \frac{(z-1) f(is)}{(is-z)(is-1)} \, ds,
\]
one obtains
\begin{equation} \label{P2}
f'(z) = - \frac{1}{2\pi} \int_\R \frac{f(is)}{(is-z)^2} \,ds,
\end{equation}
as an absolutely convergent integral, for any $f \in H^\infty(\C_+)$, and higher derivatives can be obtained by repeated differentiation of the formula.  This shows that  $C_n$ can be
\[
\frac{n!}{2\pi} \int_\R (t^2+1)^{-(n+1)/2} \, dt = \frac{n! \Gamma(n/2)}{2\sqrt\pi \Gamma((n+1)/2)}.
\]
We shall use only the cases $n=1$ and $n=2$, giving $C_1 =  1/2$ and $C_2= 2 /\pi$.

Let $H^\infty(\R)$ denote the space of $g \in L^\infty(\R)$ such that $\supp(\mathcal{F}^{-1}g) \subset \R_+$.   Then the mapping $f \mapsto f^b$ is an isometric isomorphism from $H^\infty(\C_+)$ onto  $H^{\infty}(\R)$ and its inverse is given in Lemma \ref{ho}(\ref{ho1}) \cite[Section II.1.5]{Havin}.

If $f \in H^\infty(\C_+)$ and $f^b \in L^1(\mathbb R)$, then $f(\infty) = 0$ by the Poisson integral formula in Lemma \ref{ho}(1).  Then, for every $z\in \C_+$, the Cauchy representation becomes
\begin{equation}\label{cauchy2}
f(z)= - \frac{1}{2 \pi}\int_{\mathbb R}\frac{f^b(t) \,dt}{it-z}=\int_{0}^{\infty}e^{-sz} (\mathcal{F}^{-1} f^b) (s)\, ds,
\end{equation}
(see \cite[p.170]{Havin}).  Thus the Fourier-Laplace transform on the right hand side of \eqref{cauchy2}  provides the analytic  extension of $f^b$ to $\C_+$.

\subsection{The Banach algebra $\Bes$}  \label{bdef}

We define $\Bes$ to be the space of those holomorphic functions $f$ on $\C_+$ such that
\begin{equation} \label{Bfin}
\|f\|_\Bq := \int_0^\infty \sup_{y\in\R} |f'(x+iy)| \, dx < \infty.
\end{equation}
We note the following elementary properties of functions $f \in \Bes$.

\begin{prop} \label{besprop}
Let $f \in \Bes$.
\begin{enumerate} [\rm1.]
\item $f(\infty):= \lim_{\Re{z}\to\infty} f(z)$ exists in $\C$.
\item $f$ is bounded, and $\|f\|_\infty \le |f(\infty)| + \|f\|_\Bq$.
\item $f(is) = \lim_{x\to0+} f(x+is)$ uniformly for $s \in \R$.
\item The extended function $f$ is uniformly continuous on $\overline{\C}_+$, and so $f^b \in \operatorname{BUC}(\R)$.
\item If $U$ is an open set containing the range of $f$, and $h$ is a bounded holomorphic function with bounded derivative on $U$, then $h \circ f \in \Bes$.
\item If $f$ is bounded away from $0$, then $1/f \in \Bes$.
\item Assume that the range of $f$ is contained in $\Sigma_\pi$.  If $\beta>1$, then $f^\beta(z) := f(z)^\beta \in \Bes$.  If  $f$ is bounded away from $0$, then $f^\beta \in \Bes$ for all $\beta \in \R$.
\end{enumerate}
\end{prop}

\begin{proof}
We prove only the first two statements.  Note first that (\ref{Bfin}) implies that $f' \in L^1(\R_+)$, and hence
\[
f(\infty-) := \lim_{\tau\to\infty} f(\tau) = f(x) + \int_x^\infty f'(t) \, dt
\]
for any $x>0$.  Now let $z = x+iy$, and assume, without loss of generality, that $y>0$.   Let $\ep>0$.
Integrating $f'$ along the line segment from $z$ to $x+\ep^{-1}y$, and then  the horizontal half-line $[x+\ep^{-1}y,\infty)$ gives
\[
f(\infty-)  - f(z) = \int_x^{x+\ep^{-1}y} (1 - i\ep) f'\big(t + i(y-\ep(t-x)\big) \, dt + \int_{x+\ep^{-1}y}^\infty f'(t) \, dt.
\]
Hence
\begin{align*}
|f(\infty-)-f(z)| \le (1+\ep^2)^{1/2} \int_x^\infty \sup_{s\in\R} |f'(t+is)|\,dt.
\end{align*}
Letting $\ep\to0+$ gives
\[
|f(\infty-)-f(z)| \le \int_x^\infty \sup_{s\in\R} |f'(t+is)|\,dt.
\]
The first two statements follow immediately.  The remaining statements are straightforward.
\end{proof}

We define a norm on $\Bes$ as follows:
$$
\|f\|_\Bes = \|f\|_\infty + \|f\|_\Bq = \|f\|_\infty + \int_0^\infty \sup_{y\in\R} |f'(x+iy)| \, dx.
$$
The proof of Proposition \ref{besprop} shows that $\|\cdot\|_\Bes$ is equivalent to each norm of the form
\[
|f(a)| + \|f\|_\Bq \qquad ( a \in \overline{\C}_+ \cup \{\infty\}).
\]
It is easy to see that $\Bes$ is an algebra and $\|\cdot\|_\Bes$ is an algebra norm.  We shall use this norm on $\Bes$ unless stated otherwise.   This choice is convenient for the theory developed in the rest of Section \ref{besov} and in Section \ref{b-fc}, although the algebra norm is not optimal for the norm-estimates in Section \ref{apps}.

The space $\Bes$ has been denoted in a different way in some papers, as if it were a Besov space, but this is questionable; see the Appendix for further discussion of this point.

We now show that the normed algebra $(\Bes, \|\cdot\|_\Bes)$ is complete, and establish two related properties.

\begin{prop} \label{balg}
\begin{enumerate}[\rm1.]
\item The closed unit ball $U$ of $\Bes$ is compact in the topology of uniform convergence on compact subsets of $\C_+$.
\item  $\Bes$ is a Banach algebra.
\item  For $f \in \Bes$, the spectrum of $f$ in $\Bes$ is the closure of the range of $f$ (considered as a function on $\C_+$) and the spectral radius of $f$ in $\Bes$ is $\|f\|_\infty$.
\end{enumerate}
\end{prop}

\begin{proof}
1.  Since $U$ is contained in the closed unit ball of $H^\infty(\C_+)$, which is compact in the topology of uniform convergence on compact sets, by Montel's Theorem, it suffices to consider a sequence $(f_n)$ in $U$ which converges to a holomorphic function $f$ uniformly on compact sets and to show that $f \in U$.

Clearly, $\|f\|_\infty \le \liminf_{n\to\infty} \|f_n\|_\infty$.  Moreover, for $x+iy \in \C_+$,
\[
|f'(x+iy)| = \lim_{n\to\infty} |f_n'(x+iy)| \le \liminf_{n\to\infty} \sup_{s\in\R} |f_n'(x+is)|.
\]
From this and Fatou's Lemma, it follows that
\begin{align*}
\int_0^\infty \sup_{y\in \R} |f'(x+iy)|\,dx &\le \int_0^\infty \liminf_{n\to\infty} \sup_{s\in\R} |f_n'(x+is)|\, dx \\
&\le \liminf_{n\to\infty} \int_0^\infty \sup_{s\in\R} |f_n'(x+is)|\, dx.
\end{align*}
Thus
\begin{equation} \label{bf}
\|f\|_\Bes \le \liminf_{n\to\infty} \|f_n\|_\infty + \liminf_{n\to\infty} \int_0^\infty \sup_{s\in\R} |f_n'(x+is)|\, dx \le \liminf_{n\to\infty} \|f_n\|_\Bes \le 1,
\end{equation}
as required.

\noindent 2.   This can be easily deduced from (1) above.  Alternatively, let $(f_n)$ be a Cauchy sequence in $\Bes$.    Then $(f_n)$ converges pointwise to a function $f$, which is holomorphic on $\C_+$, by Vitali's theorem.   Then as in \eqref{bf} , we have
\[
\|f_n-f\|_{\mathcal B}\le \liminf_{m \to \infty}  \|f_n -f_m\|_{\mathcal B} \le \sup_{m \ge n}\|f_n-f_m\|_{\mathcal B},
\]
and the claim follows.

\noindent 3. This follows easily from Proposition \ref{besprop}(6).
\end{proof}

There is a superficial resemblance between (\ref{Bfin}) and the definition of the Hardy space $H^1(\C_+)$ on the half-plane $\C_+$, which consists of the holomorphic functions $f$ on $\C_+$ such that
\[
\|f\|_{H^1} := \sup_{x>0} \int_\R |f(x+iy)| \, dy < \infty.
\]
See \cite[Chapter 11]{Duren} or \cite[Chapter II]{Garnett} for details of the Hardy space, noting that those references consider the spaces on the upper half-plane.  It is well known that if $f' \in H^1(\C_+)$ then $f' \in L^1(\R_+)$.  For example this can be shown by applying the Carleson embedding theorem \cite[Theorem II.3.9]{Garnett} to Lebesgue measure on $\R_+$; see also \cite[p.198]{Duren}.  Here we show the stronger result that $f \in \Bes$.

\begin{prop}  \label{prim}
If $f$ is holomorphic on $\C_+$ and $f' \in H^1(\C_+)$, then $f \in \Bes$, $f' \in L^1(\R_+)$, and
\[
\|f'\|_{L^1(\R_+)}  \le  \|f\|_{\Bes_0} \le \frac{\|f'\|_{H^1}}{2}.
\]
Moreover,  $\lim_{|z|\to \infty} f(z)$ exists in $\C$.
\end{prop}

\begin{proof}
Since $f' \in H^1(\C_+)$, the boundary function $(f')^b \in L^1(\R)$.  By a standard result \cite[Theorem 11.10]{Duren},  $g:= \fti (f')^b$ vanishes on $(-\infty,0]$, $g \in L^\infty(\R_+)$ and $f' = \lt{g}$.  Hence
\[
\|f'\|_{L^1(\R_+)} \le \|f\|_\Bq \le \int_0^\infty \int_0^\infty e^{-xt} |g(t)| \,dt\,dx = \int_0^\infty \frac{|g(t)|}{t} \,dt \le \frac{\|f'\|_{H^1}}{2}.
\]
Here the final inequality is an application of the half-plane analogue of Hardy's inequality \cite[p.198]{Duren}, \cite[Theorems 4.2 and 3.1]{Hille1}.   It follows that $f' \in L^1(\R_+)$ and $f \in \Bes$.

Now as $(f')^b \in L^1(\R)$ and $f$ is uniformly continuous on $\overline{\C}_+$, we have
$(f^b)'=i(f')^b$ a.e. by a simple limiting argument. Then the two  limits $\lim_{s \to \pm \infty}f^b(s)$ exist in $\C$,
 and they are equal, since $g(0)=0$.
By a consequence of the Phragm\'en-Lindel\"of principle \cite[Theorem 5.63]{Titchmarsh}, it follows
that $\lim_{|z| \to \infty}f(z)=\lim_{s \to \pm\infty} f^b(is)$.  (The formulation of \cite[Theorem 5.63]{Titchmarsh} requires $f$ to be holomorphic on $\overline{\C}_+$, but it clearly extends to the case when $f$ is continuous on $\overline{\C}_+$ and holomorphic on $\C_+$.)
\end{proof}

In Section \ref{LT} and  Section \ref{subss} we shall consider some other classes of functions which are in $\Bes$.

\subsection{Spectral decompositions} \label{pwd}

For any closed subset $I$ of $\R_+$, we define the spectral subspace $H^{\infty}(I)$ of $H^\infty(\C_+)$
by
\[
H^{\infty}(I)=\{f \in H^\infty(\C_+): \operatorname{supp} \big(\fti f^b\big) \subset I\}.
\]
The sense in which these subspaces are ``spectral'' will become clear later in this section.

The subspaces $H^\infty(I)$ for closed intervals $I$ will play a special role in what follows.
Some of them allow a very simple description. For example, by the Phragmen-Lindel\"of theorem (see e.g.\ \cite[p.175, F]{Havin}), one has
\[
H^{\infty}[\sigma, \infty)=e_\sigma H^\infty(\mathbb C_+),
\]
where  $e_\sigma(z) = e^{-\sigma z}$, reflecting the fact that shifts of the distributional Fourier transform correspond to multiplication by exponential functions.   So
\begin{equation} \label{expdef}
|f(x+iy)| \le e^{-\sigma x}\|f\|_\infty, \qquad f \in H^\infty[\sigma,\infty).
\end{equation}
Using the distributional Paley-Wiener-Schwartz theorem (\cite[Section 7.4]{Hoermander}; see also \cite[Theorem 3.5]{Carlsson}), a function $f \in H^\infty(I)$, where $I \subset \R_+$ is compact, can be recovered from its boundary function $f^b$ by means of the distributional Fourier-Laplace transform:
\begin{equation*}
f(z)= \fti f^b (e_z), \qquad z \in \mathbb C_+,
\end{equation*}
where $e_z(\l) = e^{-z\l}$ and the right-hand side is well-defined since the distribution $\fti f^b$ has compact support, so it can be applied to $C^\infty$-functions without ambiguity.   It follows that the Fourier transforms of restrictions of $f \in H^\infty(I)$ to vertical lines are essentially given by the Fourier transform of $f^b$, i.e., denoting $f_x(y):=f(x+iy), x>0$, we have
$$
\fti f_x =e_x \fti f^b.
$$
Alternatively, one may observe that $f(x+iy) = (P_x * f^b)(y)$ and use that $(\mathcal F P_x)(s)= e^{-x|s|} ,\, x>0$, and that $\fti f^b$ has compact support in $\R_+$. The support of $\fti f^b$ is invariant under vertical translations replacing $f(z)$ by $f(z+is)$.  It follows, from the Poisson formula (Lemma \ref{ho}) for example, that if $f \in H^\infty(I)$, then $y \mapsto f(x+iy)$ and $y \mapsto f'(x+iy)$ (for fixed $x>0$) also have distributional inverse Fourier transforms with support in $I$.    For similar statements for more general distributions see \cite[Proposition 2.10]{Carlsson}.

The Paley-Wiener theorem shows that, if $g \in L^\infty(\mathbb R)$ and $\sigma>0$, then $\operatorname{supp} (\fti g) \subset [-\sigma,\sigma]$ if and only if  $g$ extends  to an entire function of exponential type not exceeding $\sigma$
\cite[Chapter II.5.7]{Havin}. This can be easily transformed into a characterization of  $H^\infty(I)$ for arbitrary $I$ by multiplication with an exponential function.  The conclusion is that the smallest closed interval $I$ such that an entire function $f$ of exponential type belongs to $H^\infty(I)$ has
$$
\sup \{I\} = -\limsup_{x \to \infty}  \frac{\log |f(x)|}{x}\qquad \text{and} \qquad \inf \{I\} = - \limsup_{x \to -\infty} \frac{\log |f(x)|}{x}
$$
(see \cite[Theorem 3.5]{Carlsson}).

Recall that if $g \in L^2(i\R)$ and $\fti( g(-i\cdot))$ has support in $[0,\sigma]$, then $\fti (g(-i\cdot)) \in L^2(0,\sigma)$ by Plancherel's theorem, and $g$ extends to $\mathbb C_+$ with the representation
$$
g(z)= \frac{1}{2\pi} \int_0^{\sigma} e^{-s z} \left(\fti(g(-i\cdot))\right) (s)\, ds, \qquad z \in \mathbb C_+.
$$
This classical Paley-Wiener theorem will be used in Lemma \ref{PW}  and then in Lemma \ref{Sconv} which plays an important role in Section \ref{b-fc}.

We recall Bernstein's inequality (see \cite[Lemma 2.1]{BCD}, \cite[Theorem 11.1.2]{Boas}), which gives the following for $f \in H^\infty[0,\sigma]$:
\begin{equation} \label{Bernstein}
\|f'\|_\infty \le \sigma \|f\|_\infty,
\end{equation}
Applying this to $f'(z+x)$ gives
\[
|f''(x+iy)| \le \sigma \sup_{s\in\R} |f'(x+is)|.
\]
Hence, if $f \in \Bes \cap H^\infty[0,\sigma]$, then $f' \in \Bes$ and $\|f'\|_\Bes \le \sigma \|f\|_\Bes$.

We note also Bohr's inequality (see \cite[Theorem 2]{Logan}), which gives the following for $f \in H^\infty[\ep,\infty), \, \ep>0$:
\begin{equation} \label{Bohr}
|f(z)| \le \frac{\pi}{2\ep} \sup_{s\in\R} |f'(x+is)|, \qquad z = x+iy \in \C_+.
\end{equation}

We now show that $H^\infty[\ep,\sigma] \subset \Bes$, with a continuous embedding, when $0 < \ep < \sigma < \infty$, directly from our definition of $\Bes$ in Section \ref{bdef}.   Using an approach via dyadic decompositions (see the Appendix of this paper), a similar result was obtained in \cite[Lemma 5.5.10]{White} and related ideas appear in  \cite{V1} and \cite{Haase}.

\begin{lemma}\label{L1}
Let $f\in H^\infty[\ep,\sigma]$, $0<\ep<\sigma$.  Then $f \in \Bes$ and
\begin{equation}\label{AB}
\|f\|_\Bes \le \|f\|_{\infty}\left(1+2\log\left(1+\frac{2\sigma}{\ep}\right)\right).
\end{equation}
\end{lemma}
\begin{proof}
Let $x > 0$, $y \in \R$.  Since $f\in H^\infty[\ep,\sigma]$, (\ref{expdef}) gives
\[
\sup_{s\in\R} |f(x+is)|\le e^{-\ep x}\|f\|_{\infty},
\]
and then, by Bernstein's inequality (\ref{Bernstein}),
\begin{equation}\label{A1}
|f'(x+iy)|\le \sigma
\sup_{s\in {\mathbb R}}\,|f(x+is)|\le
\sigma e^{-\ep x}\|f\|_{\infty}.
\end{equation}
On the other hand, by Lemma \ref{ho}(\ref{derbd}),
\begin{equation}\label{A2}
|f'(x+iy)|\le \frac{\|f\|_{\infty}}{2x}.
\end{equation}
Thus we obtain
\[
|f'(x+iy)|\le \|f\|_\infty \min( \sigma e^{-\ep x}, (2x)^{-1}) \le \frac{2 \|f\|_\infty}{2x+e^{\ep x}/\sigma}.
\]
So,
\begin{equation}\label{B1}
\|f\|_\Bes\le
\|f\|_{\infty}\left(1+2\int_0^\infty \frac{dx}{2x+e^{\ep x}/\sigma}\right).
\end{equation}
Since
\begin{align}
\int_0^\infty \frac{dx}{2x+e^{\ep x}/\sigma}&=
\int_0^\infty \frac{dt}{2t+(\ep/\sigma)e^t}\label{B2+}\\
&\le
\int_0^1\frac{dt}{2t+(\ep/\sigma)}+ \int_1^\infty \frac{dt}{2 + (\ep/\sigma)e^t} \notag\\
&=\frac{1}{2}\log\left(1+\frac{2\sigma}{\ep}\right)+ \frac{1}{2}\log\left(1+\frac{2\sigma}{e\ep}\right)\notag \\
&\le
\log\left(1+\frac{2\sigma}{\ep}\right), \notag
\end{align}
 (\ref{AB}) follows.
\end{proof}

If $f \in H^\infty(I)$ and $g \in H^\infty(J)$, then $fg \in H^\infty(I+J)$ \cite[Lemma VI.4.7]{Katz}.  Let
\begin{align*}
\ssp &:= \left\{f \in H^\infty(\C_+) : \text{$\operatorname{supp} (\fti f^b)$ is a compact subset of $(0,\infty)$}\right \} \\
& = \bigcup_{0<\ep<\sigma<\infty} H^\infty[\ep,\sigma].
\end{align*}
Then $\ssp$ is a subalgebra of $\Bes$ by Lemma \ref{L1}, and all functions $f \in \ssp$ extend to entire functions of exponential type on $\C$, they are bounded on $i\R$, and they decay exponentially as $x \to \infty$.   We shall show in Proposition \ref{densehinf} that the closure of $\ssp$ in $\Bes$ is
\begin{equation} \label{gbar}
\overline{\ssp} = \Bq := \left\{ f \in \Bes:  f(\infty)=0 \right\}.
\end{equation}
So any function $f \in \Bes$ is the sum of a function in $\overline\ssp$ and a constant function.

Our proof of Proposition \ref{densehinf} will use some techniques from general operator theory, applied to the shift operators to the right and vertically, on $\Bes$. There is a less abstract but much more technical proof of Proposition \ref{densehinf}.

 The following lemma is quite simple but it plays a crucial role here and in the development of the functional calculus in Section \ref{b-fc}.

\begin{lemma} \label{shifts01}
Let
\[
  (T_\Bes(a)f)(z):=f(z+a), \qquad  f \in \Bes, \; a\in \overline\C_+, \; z\in \C_+.
	\]
\begin{enumerate}[\rm1.]
\item   For each $f \in \Bes$,
\[
\|T_\Bes(a)f\|_\Bes \le \|f\|_\Bes, \quad \lim_{a\in\overline\C_+, a\to0}\,\|T_\Bes(a)f-f\|_\Bes=0.
\]
\item  Let $-A_\Bes$ be the generator of the $C_0$-semigroup $(T_\Bes(t))_{t\ge0}$ on $\Bes$.  Then
\[
D(A_\Bes) = \{ f \in \Bes: f' \in \Bes\},  \quad A_\Bes f = -f'.
\]
\item The generator of the $C_0$-group $(T_\Bes(-is))_{s\in\R}$ is $iA_\Bes$.
\item $\sigma(A_\Bes) = \R_+$.
\item \label{012} Let $f\in\Bes$ and $(S_\Bes(b)f)(z) = f(bz)$, $b>0$.  Then $S_\Bes(b)f \in \Bes$ and $\|S_\Bes(b)f\|_\Bes = \|f\|_\Bes$.
\end{enumerate}
\end{lemma}

\begin{proof}
1.  It is clear from the definition of $\|\cdot\|_\Bes$ that $\|T_\Bes(t)f\|_\Bes \le \|f\|_\Bes$.

Since $f$ is uniformly continuous on $\overline\C_+$, $\|T_\Bes(a)f - f\|_\infty \to 0$ as $a\to0$.  For $x>0$ and $y \in \R$, integrating $f''$ along a line-segment and applying Lemma \ref{ho}(\ref{derbd})  gives
\[
|f'(a+x+iy) - f'(x+iy)| \le \frac{2 |a|\|f\|_\infty}{\pi x^2},
\]
so
\[
\lim_{a\to0}\, \sup_{y\in\R} |f'(a+x+iy) - f'(x+iy)| = 0.
\]
By the maximum principle for $f'$ on $\{z :\Re z \ge x\}$,
\begin{multline*}
\sup_{y\in\R} |f'(a+x+iy) - f'(x+iy)| \\
\le \sup_{y\in\R} |f'(a+x+iy)| + \sup_{y\in\R} |f'(x+iy)| \le 2\sup_{y\in\R} |f'(x+iy)|,
\end{multline*}
and this function is integrable over $\R_+$.  By the dominated convergence theorem,
\[
\lim_{a\to0}\|T_\Bes(a)f-f\|_\Bq = \lim_{a\to0} \int_0^\infty \sup_{y\in\R} |f'(a+x+iy) - f'(x+iy)| \, dx = 0.
\]

\noindent 2.
It is clear that if $f \in D(A_\Bes)$, then $f' = -A_\Bes f \in \Bes$.   Conversely if $f, f' \in \Bes$, then
\[
t^{-1} \left( T_\Bes(t)f - f \right) = t^{-1} \int_0^t T_\Bes(s)f' \, ds.
\]
This formula holds pointwise, and the integral is a $\Bes$-valued integral with continuous integrand.  Letting $t\to0+$ gives $f \in D(A_\Bes)$.

\noindent 3.
This follows immediately from \cite[Section 3.9]{ABHN} or from direct calculations.

\noindent 4.  It follows from (2) and (3) that $\sigma(A_\Bes) \subset \R_+$.  On the other hand, for any $a\in\R_+$, $e_a$ is an eigenvector of $A_\Bes$ with eigenvalue $a$.

\noindent 5.  This is simple to check.
\end{proof}
		
We shall show in Proposition \ref{shiftsg} that $(T_\Bes(a))_{a\in\C_+}$ is a holomorphic $C_0$-semigroup.

Next we recall the notion of spectral subspaces introduced by Arveson in the context of bounded representations of locally compact abelian groups on Banach spaces, particularly operator algebras.  We need only the special case of $C_0$-groups which is described in \cite{Zsido} and \cite[Chapter 8]{Dav80}.  To our knowledge this theory has not previously been applied to the study of holomorphic function spaces.

Let $(T(t))_{t \in \mathbb R}$ be a bounded $C_0$-group on a Banach space $X$, with generator $A$, so the spectrum $\sigma(A) \subset i\R$.  For $x \in X$ define the spectrum $\sp_T(x)$ of $x \in X$  as
\begin{equation*}
\sp_T (x):= \supp \mathcal F(T(\cdot)x).
\end{equation*}
For a closed subset $I$ of $\R$, the spectral subspace is defined to be
\begin{equation*}
X_T(I):=\{x\in X : \sp_T(x) \subset I \}.
\end{equation*}
There are numerous equivalent definitions; see \cite[Section 8.2]{Dav80}, \cite[Section 2]{Ole}, \cite[Section 8.1]{Ped}, for example.  In particular, $X_T(I)$ is the largest closed $T$-invariant subspace $Y$ of $X$ such that $\sigma(A_Y) \subset \{is: s\in I\}$, where $A_Y$ is the generator of $T$ restricted to $Y$.  Moreover $\{s \in \R : is \in \sigma(A)\}$ is the smallest closed set $J \subset \R$ such that $X_T(J) = X$ (see \cite[Lemma 8.17]{Dav80}, \cite[Remark 4.6.2, Lemma 4.6.8]{ABHN}).

Let $(S(t))_{t \in \mathbb R}$ be the $C_0$-group of shifts on $\operatorname{BUC}(\mathbb R)$:
\begin{equation} \label{bucs}
(S(t)f)(s) = f(s-t).
\end{equation}
A simple calculation shows that
\[
(\mathcal F(S(\cdot)f)(s))(t) = 2 \pi (\fti f)(s) e^{-ist}, \qquad s,t \in \R.
\]
Hence
\begin{equation}\label{spectrums}
\supp(\fti f)  = \sp_S(f).
\end{equation}

\begin{rem} \label{arveson}
The notion of spectral subspaces also applies to the duals of $C_0$-groups, and in particular to the shifts on $L^\infty(\R)$  regarded as the dual of a $C_0$-group on $L^1(\R)$.  Since (\ref{spectrums}) is also valid for $f \in L^\infty(\R)$, the space $H^\infty(\R)$ is the spectral subspace of $L^\infty(\R)$ corresponding to $I = \R_+$.  We will not use that fact, but we will use Lemma \ref{spsubs} below which is related to it.
\end{rem}

The following abstract result is a consequence of \cite[Corollary 3.5]{Zsido} or \cite[Corollary 8.1.8]{Ped}, but those results are set in more general contexts and they rely on different definitions of the spectral subspaces.   We give a simple direct proof.

\begin{prop}\label{dense}
Let $(T(t))_{t \ge 0}$ be a bounded $C_0$-group on a Banach space $X$, with generator $A$, and assume that
the range of $A$ is dense in $X$.
\begin{enumerate}[\rm1.]
\item The set $\bigcup \{X_T(I) : I \; {\rm compact}, 0\notin I \}$ is dense in $X$.
\item If, in addition, $\sigma(A)\subset i [0,\infty)$, then $\bigcup \{X_T(I) : I \; {\rm compact}, I \subset (0,\infty) \}$ is dense in $X$.
\end{enumerate}
\end{prop}

\begin{proof}
Let $f \in \mathcal{S}(\R)$ be such that  $\mathcal F f$ has compact support and is constantly $1$ in a neighbourhood of $0$.  Define $\{g_a: a >0\} \subset L^1(\mathbb R)$ by
\begin{equation*}
g_a(t)= a f (at)- a^{-1} f(t/a), \qquad t \in \mathbb R.
\end{equation*}
Then $\supp(\mathcal F g_a)$ is a compact subset of $\R \setminus \{0\}$.   Let $x = Ay$, where $y \in D(A)$, and set
\begin{equation*}
 x_a = \int_{\mathbb R} g_a(t) T(t)x\, dt.
\end{equation*}
By  \cite[Lemma 2.2.1]{Ole},
\begin{equation} \label{2part}
\sp_T(x_a)\subset \supp(\mathcal F g_a)\cap \sp_T(x)
\subset \operatorname{supp} (\mathcal F g_a)\cap \left(-i\sigma (A)\right).
\end{equation}
Since  $\int_{\mathbb R} f(t)\, dt=1$ and $(T(t))_{t \ge 0}$ is strongly continuous,
\begin{equation*}
\lim_{a \to \infty} \int_{\mathbb R} a f(at)T(t)x\, dt -x=\lim_{a \to \infty}  \int_{\mathbb R} f(t)(T(t/a)x -x)\,dt=0.
\end{equation*}
Moreover, using integration by parts we obtain
\begin{multline*}
\lim_{a \to \infty} \int_{\mathbb R} a^{-1} f(t/a) T(t)x\, dt = \lim_{a \to \infty} \int_{\mathbb R}  f(t) T(at)Ay\, dt\\
= \lim_{a \to \infty}  a^{-1} \int_{\mathbb R}  f(t) (T(at)y)'\, dt=
- \lim_{a \to \infty} a^{-1} \int_{\mathbb R}  f'(t) T(at)y \, dt \notag
=0.\notag
\end{multline*}
Thus we infer that
\begin{equation*}
\lim_{a \to \infty} x_a    = x.
\end{equation*}
Now (\ref{2part}) and the density of the range of $A$ imply both claims.
\end{proof}

Consider the $C_0$-semigroup $(T_{\mathcal B}(t))_{t\ge0}$  and the $C_0$-group of isometries
\begin{equation*}
G(t) :=T_{\mathcal B}(-it), \qquad t \in \mathbb R,
\end{equation*}
on $\Bes$, with generators $-A_\Bes$ and $iA_\Bes$ respectively, as in Lemma \ref{shifts01}.

\begin{lemma} \label{spsubs}
Let $I$ be a compact subset of $(0,\infty)$.   Then  $\Bes_G(I) = H^\infty(I)$.
\end{lemma}

\begin{proof}
Let $(S(t))_{t \in \R}$ be the $C_0$-group of shifts on $\operatorname{BUC}(\R)$, as in (\ref{bucs}).   Let $K :\Bes \to \operatorname{BUC}(\R)$ be the injection given by $Kf = f^b$.   Then
\[
S(t)K = KG(t),
\]
and hence, for $f \in \Bes$,
\begin{equation} \label{speq}
\sp(f^b) = \supp\mathcal F(S(\cdot)Kf) = \supp(K \mathcal F(G(\cdot)f)) = \supp \mathcal F(G(\cdot)f) = \sp_{G}(f).
\end{equation}
This implies that $\Bes_G(I) \subset H^\infty(I)$ for every closed subset $I$ of $\R_+$.

On the other hand, if $I$ is a compact subset of $(0,\infty)$, then $H^\infty(I) \subset \Bes$ by Lemma \ref{L1}, and then (\ref{speq}) implies that $H^\infty(I) \subset \Bes_G(I)$.
\end{proof}

\begin{prop}\label{densehinf}
The set $\ssp :=\bigcup_{0<\ep <\delta} H^{\infty}[\ep,\delta]$ is dense in $\mathcal B_0$.
\end{prop}

\begin{proof}
We consider the restrictions of $G(t)$ to $\mathcal B_0$ denoted by the same symbol, and note that this does not change the spectral subspaces $\Bes_G[\ep,\delta]$.

Let $f \in \mathcal B_0$.  By \ Lemma \ref{ho}(\ref{derbd}), one has
\[
\lim_{t \to \infty}\|f'(x + t +i\cdot)\|_{\infty}=0,
\]
 for every $x>0$, and
$$\|f'(x + t +i\cdot)\|_{\infty} \le \|f'(x  +i\cdot)\|_{\infty}, \qquad t \ge 0.$$
Then, by the dominated (or monotone) convergence theorem,
$$
\lim_{t \to \infty} \|T_{\mathcal B}(t)f\|_{\Bq} = \lim_{t \to \infty} \int_{0}^{\infty}\|f'(x + t +i\cdot)\|_{\infty}\,dx=0.
$$
Hence
$$
f=\lim_{t \to \infty} (f-T_\Bes(t)f) = \lim_{t \to \infty} A_{\mathcal B} \int_{0}^{t} T_\Bes(s)f\, ds.
$$
This shows that the range of  $A_{\mathcal B}$ is dense in $\mathcal B_0$.
It now suffices to apply Proposition \ref{dense} to $(G(t))_{t \in \mathbb R}$ on $\Bq$, and then apply Lemma \ref{spsubs}.
\end{proof}

\begin{rem} \label{shiftgp}
We can now give a quantified version of Lemma \ref{shifts01}(1), estimating $\|T_\Bes(t)f-f\|_\Bes$ in terms of the ``$K$-functional'' $K(f;t)$ which is a basic tool in approximation theory.   A number of regularity properties of $f$ can be described in terms of $K$ (and similar quantities).

Let $f \in \Bes$, and define
\[
K(f;t) := \inf \left\{\|f-g\|_\Bes+t\sigma \|g\|_\Bes : \sigma \in (0,\infty), g \in \Bes \cap H^\infty[0,\sigma] \right\}.
\]
Note that $K(f;t) \to 0$ as $t\to0+$, by Proposition \ref{densehinf}.

Let  $g\in \Bes\cap H^\infty[0,\sigma]$, so $g' \in \Bes$ and $\|g'\|_\Bes \le \sigma \|g\|_\Bes$, by (\ref{Bernstein}).   For $h\in H^\infty(\mathbb C_+)$ and $x>0$,
\begin{align*}
\sup_{y\in  {\mathbb R}}\, |h(x+iy+t)-h(x+iy)| &\le
\int_0^t \sup_{y\in {\mathbb R}}|h'(x+iy+s)|\,ds \label{S11} \\
&\le  t\sup_{y\in  {\mathbb R}}\, |h'(x+iy)|.\notag
\end{align*}
Applying this with $h =g$ and with $h=g'$, we obtain that
\[
\|T_\Bes(t)g-g\|_\Bes\le t\|g'\|_\Bes\le t\sigma \|g\|_\Bes.
\]
Then
\begin{align*}
\|T_\Bes(t)f-f\|_\Bes &\le
\|f-g\|_\Bes+\|T_\Bes(t)(f-g)\|_\Bes+
\|T_\Bes(t)g-g\|_\Bes\\
&\le 2\|f-g\|_\Bes+t\sigma \|g\|_\Bes.
\end{align*}
So, we have
\[
\|T_\Bes(t)f-f\|_\Bes\le 2 K(f;t) \to 0.
\]
\end{rem}

\subsection{Laplace transforms of measures}  \label{LT}

Let $\LT$ be the Hille-Phillips algebra, which is the subalgebra of $H^\infty(\C_+)$ consisting of Laplace transforms $\lt\mu$ of measures $\mu \in M(\R_+)$.  Let  $m = \lt\mu$ and $\|m\|_{\text{HP}} = \|\mu\|$.  Then
\[
|m(z)| \le |\mu|(\R_+), \quad m'(z) = - \int_{\R_+} t e^{-tz} \, d\mu(t),  \quad z\in\C_+,
\]
and
\begin{equation} \label{LTest}
 \int_0^\infty \sup_{y>0} |m'(x+iy)| \, dx \le \int_0^\infty \int_{\R_+} t e^{-tx} \,d|\mu|(t) \,dx =  |\mu|(0,\infty).
\end{equation}
So $m \in \Bes$, $\|m\|_\Bes \le 2\|m\|_{\text{HP}}$ and  $m(\infty) = \mu(\{0\})$.  If $\mu$ is a positive measure, then equality holds in (\ref{LTest}) and $\|m\|_\Bes = \mu(\R_+) + \mu(0,\infty)$.

Since the Laplace transform $\mu \mapsto \lt\mu$ is a bounded algebra homomorphism from $M(\R_+)$ to $\Bes$,  $\LT$ is a subalgebra of $\Bes$.  The convolution-exponential $\exp_*(\mu)$ of $\mu$ in $M(\R_+)$ may be calculated in $\LT$, and its Laplace transform is the function $z \mapsto e^{m(z)}$.

\begin{exas}  \label{LTex}
\begin{enumerate}[1.]
\item
For $a \in \R_+$, let  $e_a(z) = e^{-az}$.   Then $e_a$ is the Laplace transform of the Dirac delta measure $\delta_a$, $\|e_a\|_\Bes = 2$ if $a>0$,  \label{LTex1} and $\|e_a- \nolinebreak e_b\|_\Bes = 4$ whenever $a,b$ are distinct and non-zero.

\item \label{LTex2}  For $a = x+iy \in \C_+$, let $r_a(z) = (z+a)^{-1}$.  Then $r_a$ is the Laplace transform of the function $e_a \in L^1(\R_+)$ and  $\|r_a\|_\Bes = 2/x$.   Hence  $e^{ t  r_a} \in \LT$ and $\|e^{t r_a}\|_\Bes \le e^{2|t|/x}, \, t \in\R$.    We shall give a sharper estimate in  Lemma \ref{HZL11}.

\item \label{LTex3}  Let $\chi = 1 - 2r_1$, so
\[
\chi(z) = \frac{z-1}{z+1}.
\]
Then $\chi$ is the Laplace transform of the measure $\delta_0 - 2e^{-t}$, and $\|\chi\|_\Bes = 3$.

\item \label{LTex4}
Let
\begin{equation} \label{eee}
\eta(z):=\frac{1-e^{-z}}{z}.
\end{equation}
Then $\eta$ is the Laplace transform of Lebesgue measure on $[0,1]$, and $\|\eta\|_\Bes = 2$.
\end{enumerate}
\end{exas}

The following lemma is not new; for example, a similar argument is given in \cite[p.250]{V1}.  It plays an important role in this subject, and we give a precise statement and proof here.

\begin{lemma} \label{PW}
Let $f \in H^\infty[0,\sigma]$, $g \in \LT$ with $\fti g^b \in L^2[0,\delta]$, where $\sigma, \delta>0$.  Then $f g \in \LT$.
\end{lemma}

\begin{proof}
Since $f \in H^\infty[0,\sigma]$, $\fti f^b$ is a distribution on $\R$  with support in $[0,\sigma]$.  Then $\psi := \fti f^b * \fti g^b$ is a distribution on $\R$ with support in $[0,\sigma+\delta]$, and $\mathcal{F}\psi = f^b g^b$.  Since $f^b$ is bounded and $g^b \in L^2(\R)$ by Plancherel, $f^b g^b \in L^2(\R)$.  Then by Plancherel, $\psi \in L^2(\R)$ and $fg = \lt\psi$.  Since $\psi$ has support in $[0,\sigma+\delta]$, $\psi  \in L^1 (\R_+)$, so $fg \in \LT$.
\end{proof}

We will now consider some topological properties of $\Bes$ and $\LT$.

\begin{lemma}  \label{nondense}
Consider $\Bes$ with its norm-topology.
\begin{enumerate}[\rm1.]
\item  The Banach space $\Bes$ is not separable.
\item  The subspace $\LT$ is not closed in $\Bes$.
\item  $\LT$ is not dense in $\Bes$ in the norm-topology.
\item   $\LT$ is dense in $\Bes$ in the topology of uniform convergence on compact subsets of $\C_+$.
\end{enumerate}
\end{lemma}

\begin{proof}
1.  The first statement is an immediate consequence of Example \ref{LTex}(\ref{LTex1}).

\noindent
2.  If $\LT$ were closed in $\Bes$, then $\|\cdot\|_{\text{HP}}$ and $\|\cdot\|_\Bes$ would be equivalent norms on $\LT$.  We shall show in Sections \ref{Cayley1} and \ref{Cayley2} that this is not so.

\noindent
3.  For this we recall some facts about spaces of functions on $\R$ and apply them to boundary functions.  For each measure $\mu \in M(\R)$, $\fti\mu$ is weakly almost periodic \cite[Theorem 11.2]{Eber}, \cite[Corollary 4.2.4]{DuRa}.  The weakly almost periodic functions form a proper closed subspace of $\operatorname{BUC}(\R)$  \cite[Theorems 5.3, 12.1]{Eber}.  Moreover, the space of functions in $\operatorname{BUC}(\R)$ which are restrictions to $\R$ of entire functions of exponential type is dense in $\operatorname{BUC}(\R)$, by the Bernstein-Kober theorem (see \cite[Theorem 12.11.1]{Boas}).  Hence there exist an entire function $G$ of exponential type and $\delta>0$ such that, for all $\mu \in M(\R)$,
\[
\|G_\R - \fti\mu\|_{L^\infty(\R)} \ge \delta,
\]
where $G_\R$ is the restriction of $G$ to $\R$.  Let $\sigma$ be greater than the exponential type of $G$, so that the support of $\fti G_\R$ is contained in $[-\sigma,\sigma]$.  Let $g(z)=e^{-2i\sigma z} G(z)$ and note that  $\operatorname{supp} (\fti g)\subset [\sigma, 3\sigma]$.  For $\mu \in M(\R)$, define  $\mu_\sigma$ by  $\mu_\sigma(A)=\mu (A + 2\sigma)$ for each Borel subset $A$ of $\R$.  Then
\begin{equation*}\label{appr}
\|g_\R - \fti\mu\|_{L^\infty(\R)} = \|G_\R -\fti\mu_\sigma\|_{L^\infty(\R)} \ge  \delta.
\end{equation*}
Now let $f(z) = g(-iz)$ for $z \in \C_+$.  Then $f \in H^\infty [\sigma, 3\sigma] \subset \Bes$, and $\|f-\nolinebreak m\|_\Bes \ge \|f-  m\|_\infty \ge \delta$ for all $m \in \LT$.

\noindent
4 .  Let $\eta$ be as in (\ref{eee}), and for $\delta>0$, let $\eta_\delta(z) = \eta(\delta z)$.  For $f \in \ssp$, let $f_\delta(z) =  f(z) \eta_\delta(z)$.   Then $\eta_\delta(z) \to 1$ and $f_\delta(z) \to f(z)$ uniformly on compact sets, as $\delta\to0+$.  By Lemma \ref{PW}, $f_\delta \in \LT$.  By Proposition \ref{densehinf}, $\ssp$ is norm-dense in $\Bq$.  Since $1 \in \LT$, the result follows.
\end{proof}

\subsection{Duality and weak approximation} \label{dual}

Let $\Bov$ be the space of holomorphic functions $g$ on $\C_+$ such that
\begin{equation} \label{defe0}
\|g\|_{\Bov_0} :=\sup_{x>0}\,x\int_{-\infty}^\infty |g'(x+iy)|\,dy<\infty.
\end{equation}
The constant functions are in $\Bov$, and $\|\cdot\|_{\Bov_0}$ is a seminorm which vanishes on the constant functions.

Let $g \in \Bov$, $h\in L^\infty(\C_+)$ and $x>0$.   The following simple estimate will be very  useful:
\begin{equation} \label{boves}
\int_\R \left| g'(x \pm iy) h(x+iy) \right| \, dy \le \frac{\|g\|_{\Bov_0}}{x} \sup_{s\in\R} |h(x+is)|.
\end{equation}

Functions in $\Bov$ have additional properties as follows.

\begin{prop} \label{space_e}
Let $g \in \Bov$, and let $g_a(z) = g(a+z), \, z \in \C_+, a>0$.
\begin{enumerate}[\rm1.]
\item For each $a>0$, $g_a \in \Bes$ and $\|g_a\|_{\Bes_0} \le \|g\|_{\Bov_0}/(2a)$.
\item  $g(\infty) := \lim_{\Re z\to\infty} g(z)$ exists in $\C$.
\item  There exists $h \in L^\infty(\R_+)$ such that $g(z) = g(\infty) + \lt{h}(z)$.
\item  There exists $C_g>0$ such that $|g(z) - g(\infty)| \le C_g(\Re z)^{-1}$ and $|g^{(n)}(z)| \le n!C_g (\Re z)^{-(n+1)}, \, z \in \C_+,\, n\ge1$ .
\end{enumerate}
\end{prop}

\begin{proof}
1.  By the definition of $\Bov$, the function $g_a'$ belongs to the Hardy space $H^1(\C_+)$ and $\|g_a'\|_{H^1} \le \|g\|_{\Bov_0}/a$.   Then the first statement follows from Proposition \ref{prim}.

\noindent 2.  This follows from the first statement and Proposition \ref{besprop}(1).

\noindent 3.  Replacing $g(z)$ by $g(z)-g(\infty)$, we may assume that $g(\infty)=0$.  Then, by \cite[Theorem 2.2]{Krol}, $g$ satisfies Widder's growth condition
\[
\sup_{x>0,k\ge0} \frac{x^{k+1}}{k!} \big|g^{(k)}(x) \big| < \infty.
\]
By Widder's theorem \cite[Theorem 2.2.1]{ABHN}, $g = \lt h$ for some $h \in L^\infty(\R_+)$.

\noindent 4.  This is immediate from (2), with $C_g = \|h\|_\infty$.
\end{proof}

We can now define a norm $\|\cdot\|_\Bov$ on $\Bov$ by
\begin{equation} \label{defe}
\|g\|_\Bov := |g(\infty)| + \|g\|_{\Bov_0}, \qquad g \in \Bov.
\end{equation}
Then $\Bov$ becomes a Banach space.   To see this, let $(g_n)$ be a Cauchy sequence in $\Bov$.  Then $(g_n(\infty))$ is convergent in $\C$.   Let $a> 0$, and $f_n(z) = g_n(a+z) - g_n(\infty)$.  By  Proposition \ref{space_e}(1),  $(g_n)$ is Cauchy in $\Bes_0$, and hence in $H^\infty(\C_+)$.
Thus, the limit $f(z):=\lim_{n \to \infty}f_n(z)$ exists in $\C$ for every $z \in \mathbb C_+$, and $f$ is holomorphic.
As in the proof of Proposition \ref{balg}(2), one infers that
$\|f\|_{\mathcal E} \le \liminf_{n \to \infty}\|f_n\|_{\mathcal E}$, and
\[
\|f-f_n\|_{\mathcal E}\le \sup_{m \ge n}\|f_m-f_n\|_{\mathcal E}.
\]
Nevertheless we shall work mainly with the seminorm $\|\cdot\|_{\Bov_0}$.

The proof of Proposition \ref{space_e} has shown a relation between $\Bov$ and $H^1(\C_+)$, and the following proposition gives another relation between them ({\it cf}.\ Proposition \ref{prim}).

\begin{prop} \label{H1bov}
If $f \in H^1(\C_+)$, then $f \in \Bov$.
\end{prop}

\begin{proof}
Let $f \in H^1(\mathbb C_+)$, $\varphi \in L^\infty (\mathbb R)$ with $\|\varphi\|_{\infty}=1$, and
define
\begin{equation*}
g(z):=\int_{\mathbb R} f(z+is)\varphi(s)\, ds, \qquad z \in \mathbb C_+.
\end{equation*}
Then $g \in H^\infty(\mathbb C_+)$, and for $x>0$,
\begin{equation*}
x \Big|\int_{\mathbb R} f'(x+is)\varphi(s)\, ds \Big|=  x |g'(x)| \le \|g\|_{\infty} \le \|f\|_{H^1}.
\end{equation*}
Since the choice of $\varphi$ with $\|\varphi\|_\infty=1$ was arbitrary, we infer that
\begin{equation*}
x \int_{\mathbb R} |f'(x+is)| \, ds \le \|f\|_{H^1}.
\end{equation*}
It follows that $f \in \Bov$ and $\|f\|_{\Bov_0} \le \|f\|_{H^1}$.
\end{proof}

\begin{exas} \label{bovex}
\begin{enumerate}[\rm1.]
\item
The functions $r_a(z) := (z+a)^{-1} \quad (a \in \overline\C_+)$ are in $\Bov$ and $\|r_a\|_{\Bov_0} = \pi$.   Moreover, the function $a \mapsto r_a$ is continuous from $\C_+$ to $\Bov$.  Hence for any bounded Borel measure $\mu$ on $\C_+$, the function
\[
f(z) := \int_{\C_+}  \frac{d\mu(a)}{z+a}
\]
is in $\Bov$.
\item  The functions $e_a(z) := e^{-az} \quad (a \in \overline\C_+)$ are not in $\Bov$.
\item  The function $(z+1)^{-2}\log z$ is in $H^1(\C_+) \subset \Bov$, but it is unbounded near $0$.
\item  The function $e^{-1/z}$ is in $\Bov$ and is bounded on $\C_+$.  It is not in $\Bes$ as it is not uniformly continuous near $0$.
\end{enumerate}
\end{exas}

There is a (partial) duality between $\Bov$ and $\Bes$ given by
\begin{equation} \label{dualdef}
\langle g,f\rangle_\Bes:=\int_0^\infty x\int_{-\infty}^\infty\,g'(x-iy) f'(x+iy)\,dy \,dx,  \qquad g\in \Bov, \quad f\in \Bes.
\end{equation}
It follows from the definition of $\Bes$, and \eqref{boves} for $h = f'$, that the integral exists and
\[
|\langle g,f \rangle_\Bes |\le \|g\|_{\Bov_0}  \|f\|_{\Bes_0}.
\]
However the duality is only partial in the sense that $\Bes$ and $\Bov$ are not the dual or predual of each other with respect to this duality, and the constant functions in each space are annihilated in the duality.  This duality appeared in \cite[p.266]{V1}  (where it is presented in slightly different form, but (\ref{dualdef}) can be converted into the formula in \cite{V1}  by putting $g(z) = \overline{G(\overline{z})}$).  It was also noted in \cite{V1} that Green's formula on $\C_+$ transforms (\ref{dualdef}) into
\begin{equation} \label{Green}
\langle g,f \rangle_\Bes = \frac{1}{4} \int_\R  g(-iy)f(iy) \, dy
\end{equation}
for ``good functions'' (the argument in \cite{V1} requires a correction:
In the notation of \cite{V1}, one should set $u_1=x, u_2=F \cdot \overline G$).  See \cite[Lemma 17 and Remark, p.456]{Taib} for a precise statement and a different proof.   The approach of extending pairings of  functions on $\mathbb R$ into $\mathbb C^+$ via Green's formula has been used systematically in \cite{FS}.   We shall use (\ref{Green}) only in cases where $f$ and $g$ have holomorphic extensions across $i\R$ and they and their derivatives decay at infinity sufficiently fast.

\begin{lemma} \label{shifts}
Let $g \in \Bov$, $a \in \overline{\C}_+$ and $b>0$.   Define $T_\Bov(a)g$ and $S_\Bov(b)g$ on $\C_+$ by
\[
(T_\Bov(a)g)(z) = g(a+z), \qquad (S_\Bov(b)g)(z) = g(bz).
\]
Then the following hold:
\begin{enumerate}[\rm1.]
\item $T_\Bov(a) g \in \Bov$ and $\|T_\Bov(a) g\|_\Bov \le \|g\|_\Bov$.
\item Let $f \in \Bes$ and $T_\Bes(a)f$ be as in Lemma  \ref{shifts01}.  Then
\begin{equation} \label{shifts22}
\langle T_\Bov(a)g,f \rangle_\Bes = \langle g, T_\Bes(a)f \rangle_\Bes.
\end{equation}
\item \label{shifts3} The semigroup $(T_\Bov(t))_{t\ge0}$ is not strongly continuous on $(\Bov, \|\cdot\|_\Bov)$.
\item \label{shifts4} $S_\Bov(b)g \in \Bov$ and $\|S_\Bov(b)g\|_{\Bov} = b^{-1}\|g\|_\Bov$.
\end{enumerate}
  \end{lemma}

\begin{proof} 1.  This is very simple.

\noindent 2.  For $a = is \in i\R$, the statement follows from a simple change of variable.   So we may assume that $a>0$.

Let $f \in \Bes$, and $x>0$.  Let $h(z) := g'(a+x-z) f'(x+z)$ for $-x < \Re z < a+x$.  Since $f \in \Bes$ and $g \in \Bov$, the integrals of $h$ along $i\R$ and $a+i\R$ are absolutely convergent, and
\begin{align*}
\int_\R \left |\int_0^a  h(t+iy) \, dt\right|\,dy
&\le \int_0^a \frac{\|g\|_{\Bov}\|f\|_\infty}{(a+x-t)(x+t)} \,dt < \infty.
\end{align*}
By applying Cauchy's theorem to the integral of $h$ around the rectangles with vertices at $\pm iy_n$ and $a \pm iy_n$, for suitable $y_n\to\infty$, we conclude that the integrals of $h$ along $i\R$ and $a+i\R$ coincide,
 so that
\[
\int_\R g'(a+x-iy) f'(x+iy) \, dy = \int_\R g'(x-iy) f'(a+x+iy) \, dy.
\]
Multiplying by $x$ and integrating with respect to $x$ over $\R_+$ gives (\ref{shifts22}).

\noindent 3.   Consider $g = r_0 \in \Bov$, so $g(z)= 1/z$.    If  $a>0$ and
\begin{align*}
J(x;a):=\int_{\mathbb R}\left|\frac{1}{(x+iy)^2}-\frac{1}{(x+a+iy)^2}\right|\,dy,
\end{align*}
then $xJ(x;a)=sJ(s;1),$ where $s=x/a$. So
\begin{align*}
\|T_\Bov(a)r_0 - r_0\|_\Bov = \sup_{x>0}\,xJ(x;a)=\sup_{s>0} sJ(s,1)>0,
\end{align*}
and hence $\mathbb R_+ \ni a \mapsto T_\Bov(a) r_0$ is not continuous at 0.

\noindent 4.  This is very simple.
\end{proof}

Next we show that we can approximate functions in $\Bes$ weakly with respect to our duality by multiplying them by a suitably chosen approximate identity.   This will play a crucial role at several places in the construction of the functional calculus.

\begin{lemma}\label{Sconv}
Let $\mu \in M(\R_+)$ satisfy
\[
\quad \mu(\R_+)= 1, \quad  \int_{\R_+} t\,d|\mu|(t)  <\infty.
\]
Let $m = \lt\mu \in \LT$, and $f\in \Bes$.  For $\delta>0$, let
\[
m_\delta(z) = m(\delta z), \qquad f_\delta(z)  = m_\delta(z) f(z) \qquad (z \in \C_+).
\]
Then, for all $g \in \Bov$,
\[
\lim_{\delta\to0+} \langle g, f_\delta \rangle_\Bes = \langle g,f \rangle_\Bes.
\]
\end{lemma}

\begin{proof}
We may assume that $f(\infty)=0$, so $f \in \overline{\ssp}$ by (\ref{gbar}).  Since $\{m_\delta : 0< \delta < 1\}$ is a bounded set in $\Bes$ (Lemma \ref{shifts01}(\ref{012})), it suffices to consider the case when $f \in H^\infty[\ep,\sigma]$ for some $0 < \ep < \sigma$.  Using Bohr's inequality (\ref{Bohr}) we obtain
\[
|f'(z)-f_\delta'(z)|  \le |1- m_\delta(z)|\, |f'(z)| + |m_\delta'(z)|\,|f(z)| \le \gamma_\delta(z) \sup_{s\in\R} |f'(z+is)|,
\]
where
\[
\gamma_\delta(z) = |1-m(\delta z)| + \frac{\pi\delta}{2\ep} |m'(\delta z)|.
\]
From the assumptions,  $\gamma_\delta(z)$ is bounded uniformly in $z$ and $\delta$, and
\[
\lim_{\delta\to0+} \gamma_\delta(z) = 0, \qquad z \in \C_+. \label{gd0}
\]
Now
\[
 |\langle g,f - f_\delta \rangle_\Bes| \le \int_0^\infty x \int_{-\infty}^\infty   |g'(x-iy)| \, \gamma_\delta(x+iy) \sup_{s\in\R} |f'(x+is)| \,dy\,dx.
\]
From these properties, the absolute convergence of the repeated integral in (\ref{dualdef}) and the dominated convergence theorem, we obtain that
\[
\lim_{\delta\to0+} \langle g,f - f_\delta \rangle_\Bes = 0.  \qedhere
\]
\end{proof}

\subsection{Representation of Besov functions}

Here we show that any function in $\Bes$ can be represented in terms of the duality as in (\ref{dualdef}).  When $g = r_z$ where $r_z(\l) = (z+\l)^{-1}$ for some $z \in \C_+$ and (\ref{Green}) holds, one obtains the Cauchy integral formula:
\[
\frac{2}{\pi} \langle r_z,f \rangle_\Bes = \frac{1}{2\pi} \int_\R \frac{f(iy)}{z-iy} \, dy = \frac{1}{2\pi i} \int_{i\R}\frac{f(\l)}{\l-z} \, d\l = f(z),
\]
where the contour integral on $i\R$ is in the downward direction.  We make this more precise in the following.

\begin{prop} \label{BHP2}
Let $f \in \Bes$, $z \in \overline\C_+$ and  $r_z(\l) = (\l+z)^{-1}$.  Then
\[
f(z) = f(\infty) + \frac{2}{\pi} \langle r_z,f \rangle_\Bes.
\]
\end{prop}

\begin{proof}
First assume that $f = \widehat\mu \in \LT$ where $\mu(\{0\})=0$, or equivalently $f(\infty)=0$.  Then
\begin{align*}
\langle r_z,f \rangle_\Bes
&=  \int_\R \int_0^\infty \frac{x}{(x -iy + z)^2}    \int_0^\infty t e^{-(x + iy)t}\, d\mu(t)\, dx\,dy \\
&= \int_0^\infty \int_0^\infty x t e^{-x t} \int_\R \frac{e^{- iy t}}{(x -iy + z)^2}   \,dy \,   \,dx   \,d\mu(t) \\
&= 2\pi \int_0^\infty \int_0^\infty  x t^2  e^{-2x t} e^{-zt} \,dx\,d\mu(t) \\
&= \frac{\pi}{2}\int_{0}^\infty e^{-zt}  \,d\mu(t)\\
&= \frac{\pi}{2}f(z),
\end{align*}
where we have used that $y \mapsto (x-iy+z)^{-2}$ is the inverse Fourier transform of $t \mapsto 2 \pi t e^{-x t} e^{-zt}$ on $\R_+$ (extended to $\R$ by $0$).

Next we consider $f \in \ssp$, so $f(\infty)=0$. Let $\delta >0$ and $\eta_\delta(z) = \eta(\delta z)$, where $\eta$ is as in (\ref{eee}).   By Lemma \ref{PW}, $f_\delta \in \LT$ and $f_\delta(\infty)=0$.  By applying the case above to $f_\delta$ and using Lemma \ref{Sconv},
\[
f(z) = \lim_{\delta\to0+} f_\delta(z) = \lim_{\delta\to0+} \frac{2}{\pi} \langle r_z,f_\delta \rangle_\Bes = \frac{2}{\pi} \langle r_z,f \rangle_\Bes.
\]

The formula extends by continuity to $f \in \Bq$, and then to $f \in \Bes$ by adding constants.
\end{proof}

\begin{rem}
If $z,a \in \C_+$, we now have two formulas for $f(z+a)$:
\[
f(z+a) = f(\infty) + \frac{2}{\pi}\langle r_{z+a}, f \rangle_\Bes = f(\infty) + \frac{2}{\pi} \langle r_z, T_\Bes(a)f \rangle_\Bes.
\]
Since $r_{z+a} = T_\Bov(a) r_z$, this agrees with (\ref{shifts22}).
\end{rem}

\subsection{Dual Banach spaces}  \label{nondual}

Our partial duality induces a contractive map $\Psi_\Bes: \Bov \to \Bq^*$, where $\Bq^*$ can be identified in the natural way with the space of functionals in $\Bes^*$ which annihilate the constant functions.  It follows from Proposition \ref{BHP2} that the range of $\Psi_\Bes$ is weak*-dense in $\Bq^*$.  However it is not norm-dense.

\begin{prop} \label{nondens1}
The range of $\Psi_\Bes$ is not norm-dense in $\Bq^*$.
\end{prop}

\begin{proof}
For all $f \in \Bes$, $f^b \in \operatorname{BUC}(\R)$ and $\|f^b\|_\infty \le \|f\|_\Bes$.  For $a \in \R_+$, the function $e_a$ of Example 1.8(\ref{LTex1}) belongs to $\Bes$, and its boundary function is $e_a^b(y) = e^{-iay}$ which belongs to the Banach algebra $\operatorname{AP}(\R)$ of almost periodic functions on $\R$.

Take any $\chi$ in the Bohr compactification of $\R$ which is not in $\R$, so $\chi$ is a character of $\operatorname{AP}(\R)$ which is not evaluation at any point in $\R$.   Extend $\chi$ to a bounded linear functional $\tilde\chi$ on $\operatorname{BUC}(\R)$ and  define $\psi \in \Bq^*$ by
\[
\psi(f) = \tilde\chi(f^b) - f(1).
\]
The function
\[
a \mapsto \psi(e_a) = \chi(e_a^b) - \exp(-ia)
\]
is not measurable on $\R_+$.  If it were measurable, the map $a \mapsto \chi(e_a^b)$ would be a measurable semigroup homomorphism from $\R_+$ to the unit circle $\mathbb{T}$, and so there would exist $s \in \R$ such that $\chi(e_a^b) = \exp(-ias)$ for all $a \in \R_+$ and then for all $a \in \R$.  By linearity and continuity of $\chi$, $\chi(g) = g(s)$ for all $g \in \operatorname{AP}(\R)$.  This contradicts the choice of $\chi$.

For any $g \in \Bov$, $a \mapsto \langle g, e_a \rangle_\Bes$ is measurable, and so $a \mapsto \varphi(e_a)$ is measurable for any $\varphi$ which is in the norm-closure of the range of $\Psi_\Bes$, and also for the functional $\xi(f) = f(\infty)$ on $\Bes$.  Since $\Bes^*$ is spanned by $\Bq^* \cup \{\xi\}$, the norm-closure of the range of $\Psi_\Bes$ cannot be $\Bq^*$.
\end{proof}

We now give an alternative proof of Proposition \ref{nondens1}.   Suppose that the range of $\Psi$ is norm-dense in $\Bq^*$ (for a contradiction).  Let $f \in \ssp$.  Then Lemma \ref{Sconv} would imply that $f_\delta$ tends to $f$ weakly in $\Bq$, and Lemma \ref{PW} shows that $f_\delta \in \LT$.  Hence $\LT$ would be weakly dense in $\ssp$ and then in $\Bq$.  By Mazur's Theorem, $\LT$ would be norm-dense in $\Bes$.  This would contradict Lemma \ref{nondense}.  Thus we conclude that the range of $\Psi_\Bes$ is not norm-dense in $\Bq^*$.

Let $\Bov_0 = \{ g \in \Bov: g(\infty) = 0\}$, with the norm $\|\cdot\|_\Bov$ which coincides with $\|\cdot\|_{\Bov_0}$ on $\Bov_0$.  Identify the dual ${\Bov_0}^*$ with the space of linear functionals in $\Bov^*$ which annihilate the constant functions.   Our duality then provides a contractive map $\Psi_\Bov$ from $\Bes$ to ${\Bov_0}^*$.

\begin{prop}  \label{nondens2}
The range of $\Psi_\Bov$ is not norm-dense in ${\Bov_0}^*$.
\end{prop}

\begin{proof}
We argue by contradiction.  Suppose that the range of $\Psi_\Bov$ is norm-dense in ${\Bov_0}^*$.
Let $(T_\Bes(t))_{t\ge0}$ and $(T_\Bov(t))_{t\ge0}$ be the shift semigroups defined in Lemmas \ref{shifts01} and \ref{shifts}.  Let $g \in \Bov$. By (\ref{shifts22}) and Lemma \ref{shifts01},
\[
\langle T_\Bov(t)g,f \rangle_\Bes = \langle g, T_\Bes(t)f\rangle_\Bes \to \langle g,f \rangle_\Bov
\]
as $t\to0+$.  It follows from our hypothetical assumption and a simple approximation that $\lim_{t\to0+} \psi(T_\Bov(t)g)= \psi(g)$ for all $\psi\in{\Bov_0}^*$ and $g \in \Bov_0$.   By \cite[Proposition 1.23]{Dav80}, this implies that $\lim_{t\to0+} \|T_\Bov(t)g-g\|_\Bov=0$ for all $g\in\Bov_0$.  This contradicts Lemma \ref{shifts}(\ref{shifts3}).
\end{proof}

\section{Norm-estimates for some subclasses of $\Bes$}  \label{subss}

In this section we obtain estimates of the $\Bes$-norms of some more specific classes of functions in $\Bes$.   We give explicit forms for most of the estimates, showing how they depend on any parameters and giving explicit (but not necessarily optimal) values for any absolute constants.   The estimates will be applied to operators in Sections \ref{b-fc} and \ref{apps}.

\subsection{Holomorphic extensions to the left}  \label{left}
In this subsection we consider some functions which have holomorphic extensions to a strip in the left half-plane.  For fixed $\omega >0$, we let
\begin{align*}
&\RR_{-\omega} := \{z\in\C : \Re z > -\omega\}, &&H^\infty_\omega:= H^\infty(\RR_{-\omega}), \\
&\|f\|_{H_\omega^\infty}:=\sup_{z \in \RR_{-\omega}}\,|f(z)|, && \|f\|_{\infty}:=\sup_{z \in \C_+}\,|f(z)|.
\end{align*}

\begin{lemma}  \label{deriv}
Let $f \in H^\infty_\omega$, where $\omega > 0$.   Then $f' \in \Bes$ and
\[
\|f'\|_\Bes \le \frac{3}{2 \omega} \|f\|_{H^\infty_\omega}.
\]
\end{lemma}

\begin{proof}
1. Applying Lemma \ref{ho}(\ref{derbd}) to $g(z):=f(z-\omega)$ shows that, for $x+iy \in \C_+$,
\begin{equation} \label{lest}
|f'(x+iy)| \le \frac{\|f\|_{H^\infty_\omega}}{2(\omega+x)}, \quad |f''(x+iy)| \le \frac{2\|f\|_{H^\infty_\omega}}{\pi(\omega+x)^2}.
\end{equation}
This implies that $f' \in \Bes$ and
\[
\|f'\|_\Bes \le \frac{\|f\|_{H^\infty_\omega}}{2\omega} + \int_0^\infty \frac{2\|f\|_{H^\infty_\omega}}{\pi(\omega+x)^2} \,dx \le \frac{3}{2 \omega} \|f\|_{H^\infty_\omega}. \qedhere
\]
\end{proof}

\begin{lemma}\label{H2}
\begin{enumerate}[\rm1.]
\item  Let $f \in \Bes$, let \label{H21}
\[
 \varphi(x) = \sup_{y\in\R} |f(x+iy)| , \qquad x>0,
\]
and assume that
\[
\int_0^\infty \frac{\varphi(x)}{1+x}\,dx < \infty.
\]
Let $g \in H^\infty_\omega$ for some $\omega>0$.
Then $fg \in \Bes$ and
\[
\|fg\|_\Bes \le \|f\|_\Bes\|g\|_\infty + \frac{\|g\|_{H^{\infty}_\omega}}{2} \int_0^\infty \frac{\varphi(x)}{\omega+x} \,dx.
\]
\item \label{exp} Let $f \in H^\infty[\tau,\infty) \cap H_\omega^\infty$, for some $\tau,\omega>0$.  Then $f \in \Bes$ and
\[
\|f\|_\Bes \le e^{-\omega\tau}\left( 2 + \frac{1}{2} \log \left( 1 + \frac{1}{\tau\omega} \right)\right)  \|f\|_{H_\omega^\infty}.
\]
\end{enumerate}
\end{lemma}

\begin{proof}
1.  Applying the estimate (\ref{lest}) for $g$ we obtain that
\begin{align*}
\|fg\|_\Bes &\le \|f\|_\infty \|g\|_\infty + \int_0^\infty \sup_{y\in\R} |f'(x+iy)|\, \|g\|_\infty \, dx + \frac{\|g\|_{H^{\infty}_\omega}}{2}
\int_0^\infty \frac{\varphi(x)}{\omega+x} \,dx \\
&= \|f\|_\Bes \|g\|_\infty + \frac{\|g\|_{H^{\infty}_\omega}}{2}
\int_0^\infty \frac{\varphi(x)}{\omega+x}\,dx,
\end{align*}
as required.

\noindent 2. By (\ref{expdef}), $f = e_\tau g$ where $e_\tau(z)=e^{-\tau z}$ and $g \in H_\omega^\infty$.  Let $\varphi(x)=e^{-\tau x} = |e_\tau(x+iy)|$, $x>0, y\in \R$.   Then $\|e_\tau\|_\Bes = 2$, and
\[\int_0^\infty \frac{\varphi(x)}{\omega+x}\,dx = \int_0^\infty \frac{e^{-\tau x}}{\omega+x} \,dx
= \int_0^\infty \frac{e^{-u}}{\tau\omega+u}\, du \le \log \left(1+ \frac{1}{\tau\omega}\right),
\]
by an inequality for the (modified) integral exponential function \cite[5.1.20]{Abram}.
Moreover $g(z) = e^{\tau z} f(z)$ and
\[
 \|g\|_{H_\omega^\infty} = \sup_{s\in\R}|g(-\omega+is)| = \sup_{s\in\R} e^{-\omega\tau} |f(-\omega+is)| = e^{-\omega\tau} \|f\|_{H^\infty_\omega}.
\]
By (\ref{H21}), we obtain that $f \in \Bes$ and
\[
\|f\|_\Bes \le 2 e^{-\omega\tau} \|f\|_{H_\omega^\infty} + \frac{1}{2} \log \left( 1 + \frac{1}{\tau\omega} \right) e^{-\omega\tau} \|f\|_{H_\omega^\infty}. \qedhere
\]
\end{proof}

\subsection{Functions with decay on $i\R$}

In the following result we give an estimate for the $\Bes$-norm of certain functions which decay at infinity in an appropriate sense.

\begin{lemma}\label{HZL1}
Let  $f\in H^\infty(\mathbb C_{+})$, let
\[
h(t) = \operatorname{ess\,sup}_{|s|\ge t}|f(is)|,
\]
and assume that
\begin{equation} \label{HZLass}
\int_0^\infty \frac{h(t)}{1+t} \,dt < \infty.
\end{equation}
Let $\omega>0$ and $g(z) = f(z+\omega)$.  Then $g \in \Bes_0$ and
\[
\|g\|_{\Bq}\le 3 \int_0^\infty \frac{h(t)}{\omega+t} \,dt.
\]
\end{lemma}

\begin{proof}
The assumption (\ref{HZLass}) and monotonicity of $h$ imply that $f(is) \to0$ as $|s|\to\infty$.  Then by the Poisson integral formula, $f(\infty) =0$ and therefore $g(\infty)=0$.

Let $z = x+iy \in \C_+$.  By (\ref{P2}),
\[
|f'(z+\omega)|\le \frac{1}{2\pi}\int_{-\infty}^\infty
\frac{{h}(|s|)\,ds}{(\omega+x)^2+(s-y)^2}.
\]
For fixed $x>0$ we estimate $\sup_{y\in \mathbb R}\,|f'(x+\omega+iy)|$.
Without loss of generality we may assume that $y\ge0$.
We consider two cases.

First, if $0 \le y \le x/\sqrt{2}$, then
\[
(\omega+x)^2+(s-y)^2\ge \omega^2+x^2-y^2+ \frac{s^2}{2} \ge \omega^2 + \frac{x^2+s^2}{2},
\]
and
\[
2\pi|f'(z+\omega)|\le \int_{-\infty}^\infty
\frac{h(|s|)}{\omega^2+(x^2+s^2)/2} \,ds =
2\int_0^\infty
\frac{{h}(s)}{\omega^2+(x^2+s^2)/2} \,ds.
\]
Second, if  $y\ge x/\sqrt{2}$, then
\begin{align*}
\lefteqn{2\pi |f'(z+\omega)|} \\
&\le \int_{2y/3}^{2y}
\frac{{h}(s)}{(\omega+x)^2+(s-y)^2} \,ds
+\int_{|s-y|>|s|/2}
\frac{{h}(|s|)}{(\omega+x)^2+(s-y)^2} \,ds \\
&\le \int_{-y/3}^y
\frac{{h}(\sqrt2x/3)}{(\omega+x)^2+t^2} \,dt
+2\int_0^\infty
\frac{{h}(s)}{\omega^2+x^2+s^2/4} \,ds \\
&\le \pi\frac{{h}(\sqrt2 x/3)}{\omega+x}+
2\int_0^\infty
\frac{{h}(s)}{\omega^2+x^2+s^2/4} \,ds.
\end{align*}
Combining the two estimates above, we infer that
\[
\sup_{y\in \mathbb R}\,|f'(x+\omega+iy)|\le \frac{{h}(\sqrt2 x/3)}{2(\omega+x)}
+\frac{1}{\pi}\int_0^\infty
\frac{{h}(s)}{\omega^2+x^2/2+s^2/4} \,ds,
\]
and hence
\begin{align*}
\|g\|_{\Bq}
&\le \frac{1}{2}\int_0^\infty \frac{{h}(\sqrt{2}x/3)}{\omega+x}\,dx
+\frac{1}{\pi}\int_0^\infty\int_0^\infty
\frac{{h}(s)\,ds}{\omega^2+x^2/2+s^2/4}\,dx\\
&= \frac{3}{2\sqrt{2}}\int_0^\infty \frac{{h}(x)}{\omega+x}\,dx
+\frac{1}{\pi}\int_0^\infty\left(\int_0^\infty
\frac{dx}{\omega^2+x^2/2+s^2/4}\right)\,{h}(s)\,ds\\
&= \frac{3}{2\sqrt{2}}\int_0^\infty \frac{h(s)}{\omega+ s} \,ds
+\int_0^\infty
\frac{{h}(s)}{\sqrt{2\omega^2+s^2/2}} \,ds\\
&\le  \left(\frac{3}{2\sqrt{2}} + \sqrt{\frac{5}{2}}\right) \int_0^\infty \frac{h(s)}{\omega+ s} \,ds \le 3 \int_0^\infty \frac{h(s)}{\omega+ s} \,ds.  \qedhere
\end{align*}
\end{proof}

\subsection{Exponentials of inverses}
Let $f(z) = \exp(-1/z)$.  Then $f \in H^\infty(\C_+)$, but it does not have a continuous extension to $\overline{\C}_+$, and hence $f \notin \Bes$. \label{expinv} This can also be seen by observing that $|f'(x+iy)| = (ex)^{-1}$ when $x \in (0,1)$ and $y^2 = x - x^2$.

As observed in Example \ref{LTex}(\ref{LTex2}),  the function $e^{- t r_1}(z) = \exp(-t/(z+1)) \in \LT$ for $t\in\R$.      In Lemma \ref{HZL11} we estimate the $\Bes$-norm of these functions for $t>0$, and then Lemma \ref{HZLex} gives a stronger result with a more complicated proof (see Remark \ref{cval}).   In Section \ref{igp} we shall show that Lemma \ref{HZL11} leads easily to a known result concerning the inverse generator problem for operator semigroups (Corollary \ref{CH2}), while Lemma \ref{HZLex} leads to a more general result on the same problem (Corollary \ref{HZCex}).

\begin{lemma}\label{HZL11}
Let
\[
\tilde f_t(z)=\exp{(-t/(z+1))}, \quad z \in \C_+,\quad t>0.
\]
Then
\[
\|\tilde f_t\|_\Bes =\begin{cases} 2-e^{-t}, \quad &t\in (0,1],\\
 2-e^{-1}+e^{-1}\log t,\quad &t>1.
\end{cases}
\]
\end{lemma}

\begin{proof}
Setting
\[
z=x+iy,\quad z\in \mathbb C_{+},
\]
we have
\[
t^{-1} \big|\tilde f'_t(z)\big|= \frac{e^{-t(x+1)/((x+1)^2+y^2)}}{(x+1)^2+y^2}.
\]
For $t>0$, $s>1$, let
\[
g_{t,s}(r):=\frac{e^{-ts/(s^2+r)}}{s^2+r},\qquad r>0,
\]
so that
\[
\int_0^\infty \sup_{y\in \R}\,|f_t'(x+iy)|\,dx=\int_1^\infty \sup_{r\ge 0}g_{t,s}(r)\,ds.
\]
By a simple argument,
\[
\sup_{r\ge 0}g_{t,s}(r)= \begin{cases}  g_{t,s}(0)=\dfrac{e^{-t/s}}{s^2}, \qquad &0<t\le s,  \\
g_{t,s}(ts-s^2) =\dfrac{e^{-1}}{ts}, \qquad &t\ge s\ge 1. \end{cases}
\]
Hence if $t\in (0,1]$, then
\[
\|\tilde f_t\|_\Bes = 1+t\int_1^\infty
\frac{e^{-t/s}\,ds}{s^2}=2-e^{-t},
\]
and for $t>1$ we obtain
\[
\|\tilde f_t\|_\Bes = 1+e^{-1}\int_1^t\frac{ds}{s}+
t\int_t^\infty\frac{e^{-t/s}\,ds}{s^2}=2-e^{-1}+e^{-1}\log t.  \qedhere
\]
\end{proof}

It was shown in the proof of \cite[Theorem 3.3]{deL2} that the functions $g_t(z) = z(z+1)^{-1} e^{-t/z}$ are  in $\LT$.  In the next lemma, we consider the function $f_t := g_{t/2}^2 \in \LT$.  We estimate $\|f_t\|_\Bes$ directly, as this gives a sharp estimate in (\ref{logest}) in a fairly simple way, while estimating the $L^1$-norm of the inverse Laplace transform of $f_t$ does not appear to lead to such an estimate.

\begin{lemma}\label{HZLex}
Let
\[
f_t(z)= \left( \frac{z}{z+1}\right)^2 \exp{(-t/z)}, \qquad t>0,\, z \in \C_+.
\]
Then there is a constant $C$ such that
\begin{equation} \label{logest}
\|f_t\|_\Bes\le  C (1 + \log(1+t)),  \qquad t>0.
\end{equation}
\end{lemma}

\begin{proof}
Let $z = x+iy \in \C_+$.  Note first that $\|f_t\|_\infty = 1$ and
\[
f_t'(z) = \frac{t+(t+2)z}{(1+z)^3} e^{-t/z}.
\]
Let
\[
g_t(z) := \left| \frac{(t+2)e^{-t/z}}{(1+z)^2} \right|.
\]
Since
\[
|t + (t+2)z|  = (t+2) \left| \frac{t}{t+2} + z\right| \le (t+2) |1+z|,
\]
we obtain
\[
|f_t'(x+iy)| \le g_t(x+iy) \le \frac{t+2}{(1+x)^2+y^2} \exp\left(-\frac{tx}{x^2+y^2} \right)  \le \frac{(t+2)}{(1+x)^2}.
\]
It is already apparent from this that $f_t \in \Bes$ and
\begin{equation} \label{est1}
\|f_t\|_\Bes \le 1 + (t+2)\int_0^\infty (1+x)^{-2} \, dx = t+3.
\end{equation}

To obtain a logarithmic estimate, we shall compare $g_t(x+iy)$ with
\[
\varphi_t(x) := \frac{t+2}{(1+x)^2+tx}
\]
for $x>0$ and $y\in\R$. Note that
\begin{multline*}
\exp \left( \frac{tx} {x^2+y^2} \right) \ge 1 + \frac{tx}{x^2+y^2}\\
 \ge \frac{(1+x)^2}{(1+x)^2 + y^2} +  \frac{tx}{(1+x)^2 + y^2} = \frac{(1+x)^2 +tx}{(1+x)^2 + y^2}.
\end{multline*}
Hence
\[
g_t(x+iy) \le \varphi_t(x).
\]

 Let
\[
a_t = \frac{t+2}{\sqrt{t^2+4t}} = \left( 1+ \frac{4}{t^2+4t} \right)^{1/2} .
\]
Then
\begin{multline}  \label{est2}
\|f_t\|_\Bes \le 1 + 2\int_0^\infty \varphi_t(x) \,dx = 1 + (t+2) \int_0^\infty \frac{dx}{(1+x)^2+tx}  \\
= 1+ 2a_t \int_{a_t}^\infty \frac{du}{u^2-1}
= 1 + 2a_t \log \left(1+ \frac{2}{a_t-1} \right), 
\end{multline}
which is asymptotically equivalent to $2\log t$ for large $t$.  Now the estimates (\ref{est1}) for small $t$, and (\ref{est2}) for large $t$, imply (\ref{logest}).
\end{proof}

\begin{rem} \label{cval}
It is possible to deduce the estimate in Lemma \ref{HZL11} from the estimate (\ref{logest}), up to a multiplicative constant,  since
\begin{equation} \label{fts}
\tilde f_t(z) =  (1+r_1(z))^2 f_t(z+1).
\end{equation}
We have included the direct proof of Lemma \ref{HZL11} as it may be of interest to some readers.

The estimate in Lemma \ref{HZLex} is sharp because
\[
|f_t'(x+i\sqrt{tx})|  \ge (2e)^{-1} \varphi_t(x), \qquad x>0, t\ge 2,
\]
so
\[
\|f_t\|_\Bes \ge (2e)^{-1} \int_0^\infty \varphi_t(x)\,dx \ge c \log(1+t), \qquad t\ge2,
\]
for some $c>0$. 
\end{rem}

\subsection{Cayley transforms}  \label{Cayley1}

As observed in Example \ref{LTex}(\ref{LTex3}), the function $\chi = 1 - 2r_1 : z \mapsto (z-1)/(z+1)$ belongs to $\LT$ and therefore so do its powers.  Estimating $\|\chi^n\|_{\text{HP}}$ is quite complicated, as it involves the asymptotics of Laguerre polynomials, but it is known that they grow like $n^{1/2}$ (see Section \ref{Cayley2}).    The $\Bes$-norms of these functions grow only logarithmically, as shown in the following lemma.

\begin{lemma}\label{G}
Let
\[
f_n(z):=\left(\frac{z-1}{z+1}\right)^n, \qquad n\in \mathbb N.
\]
Then
\[
 f_n \in \Bes \quad \text{and} \quad \|f_n\|_\Bes \le  3+2\log(2n),
\]
for each $n \in \mathbb N$.
\end{lemma}

\begin{proof}
Note that  $\|f_1\|_\Bes = 3$, by Example \ref{LTex}(\ref{LTex3}).   For $n\ge2$, let
\[
a_n=n-\sqrt{n^2-1}, \quad  b_n= a_n^{-1} = n+\sqrt{n^2-1}.
\]

Letting $z=x+iy$ and $s=y^2,$ we have
\[
f'_n(z)=2n\frac{(z-1)^{n-1}}{(z+1)^{n+1}},
\]
and
\[
\frac{|f'_n(z)|^2}{4n^2} = \frac{((x-1)^2+y^2)^{n-1}}{((x+1)^2+y^2)^{n+1}}=: g_{n,x}(s).
\]
Simple calculations reveal that
\[
\max_{s\ge0}\,g_{n,x}(s)=g_{n,x}(0)=
\frac{(x-1)^{2(n-1)}}{(x+1)^{2(n+1)}},\quad
x\not\in [a_n,b_n],
\]
and
\[
\max_{s\ge0}\,g_{n,x}(s)=g_{n,x}(2nx-x^2-1) = \frac{ (2(n-1)x)^{n-1}}{(2(n+1)x)^{n+1}}
,\quad
x\in [a_n,b_n].
\]
Hence
\[
\sup_{y\in\R}\,|f'(x+iy)|=
2n\frac{|x-1|^{n-1}}{(x+1)^{n+1}},\qquad
x\not\in [a_n,b_n],
\]
and
\[
\sup_{y\in \R}\,|f'(x+iy)|= n\frac{(n-1)^{(n-1)/2}}{(n+1)^{(n+1)/2}}\frac{1}{x},\qquad x\in [a_n,b_n].
\]

Now
\begin{align*}
\|f_n\|_\Bes &= 1+n\frac{(n-1)^{(n-1)/2}}{(n+1)^{(n+1)/2}}
\int_{a_n}^{b_n}\frac{dx}{x}\\
&\phantom{XXX}+ 2n\int_0^{a_n}\frac{(1-x)^{n-1}}{(x+1)^{n+1}} \,dx
+2n\int_{b_n}^\infty\frac{(x-1)^{n-1}}{(x+1)^{n+1}} \,dx\\
&=  1+n\frac{(n-1)^{(n-1)/2}}{(n+1)^{(n+1)/2}}
\log\frac{b_n}{a_n}+ 1-\left(\frac{1-a_n}{1+a_n}\right)^n
+1-\left(\frac{b_n-1}{b_n+1}\right)^n,
\end{align*}
so that
\[
\|f_n\|_\Bes \le 3+ \frac{2n}{n+1}
\log(n+\sqrt{n^2-1})\le 3+2\log(2n). \qedhere
\]
\end{proof}

\subsection{Bernstein to Besov}

A {\it Bernstein function} is a holomorphic function $f :\C_+ \to \C_+$ of the form
\[
f(z) = a + bz + \int_{(0,\infty)} (1-e^{-zs}) \, d\mu(s),
\]
where $a\ge0$, $b\ge0$ and $\mu$ is a positive Borel measure on $(0,\infty)$ such that $\int_{(0,\infty)} \frac{s}{1+s} \,d\mu(s) < \infty$.  We note the following properties of $f$:
\begin{enumerate}
\item[(B1)] $f$ maps $\Sigma_\theta$ to $\Sigma_\theta$ for each $\theta \in [0,\pi/2]$ \cite[Proposition 3.6]{Schill}.
\item[(B2)]  On $(0,\infty)$, $f$ is increasing and $f'$ is\setcounter{tocdepth}{1} decreasing.
\item[(B3)]  On $\C_+$,
\begin{align*}
\Re f(z) &= a + b \Re z + \int_{(0,\infty)} \left( 1 -\Re e^{-zs} \right) \,d\mu(s) \\
&\ge a + b \Re z + \int_{(0,\infty)} \left( 1 - e^{-\Re zs} \right) \,d\mu(s) = f(\Re z), \\
\left|f'(z)\right| &= \Big|b + \int_{(0,\infty)} s e^{-s z} \,d\mu(s) \Big| \\
&\le b + \int_{(0,\infty)} s e^{- s \Re z} \,d\mu(s) = f'(\Re z).
\end{align*}
\end{enumerate}
For further information on Bernstein functions, see \cite{Schill}.

\begin{prop} \label{BB}
Let $f$ be a Bernstein function, $\alpha \in (0,1)$,  $\beta \in (1,1/\a]$,  $\theta \in (0,\pi/2)$, $\l \in \Sigma_\theta$.   Define $g$ and $h$ from $\C_+$ to $\C_+$ by
\[
g(z) = f(z^\alpha)^\beta,  \qquad  h(z) = \frac{1}{\lambda + g(z)}.
\]
Then $h \in \Bes$ and $\|h\|_\Bes \le C|\lambda|^{-1}$,  where
\[
C= 2\beta \sec(\a\pi/2) \sec^2 ((\a\b\pi/2) +\theta)/2).
\]
\end{prop}

\begin{proof}
Note that $g$ maps $\C_+$ into $\Sigma_{\alpha\beta\pi/2} \subset \C_+$.  Hence
\[
|\lambda + g(z)| \ge \kappa(|\l|+ |g(z)|),
\]
where $\kappa = \cos((\a\beta\pi/2 + \theta)/2)$.   Then
\[
\|h\|_\infty \le \sup_{z\in \C_{+}} \frac{1}{\kappa(|\lambda|+|g(z)|)} \le \frac{1}{\kappa|\lambda|}.
\]
Moreover
\[
h'(z) = - \frac{g'(z)}{(\l+g(z))^2} = - \frac{\alpha \beta f(z^\alpha)^{\beta-1} f'(z^\alpha)} {\left(\lambda + f(z^\alpha)^\beta \right)^2 z^{1-\alpha}}.
\]
Writing $z = x+iy$, we will estimate $|h'(z)|$ for fixed $x>0$ and arbitrary $y \in \R$.
Since $\Re (z^\a) = (x^2+y^2)^{\a/2} \cos(\a \arg z) \ge c_\a x^\a$ where $c_\alpha = \cos (\a\pi/2)$, the properties (B1)-(B3) above imply that
\begin{gather*}
|f'(z^\a)| \le f'(\Re(z^\a)) \le f'(c_\a x^\a), \\
|f(z^\a)| \ge \Re f(z^\a) \ge f(\Re(z^\a)) \ge f(c_\a x^\a).
\end{gather*}
Thus
\begin{align*}
|h'(x+iy)| &\le \frac{\a\beta |f(z^\a)|^{\beta-1}  |f'(z^\a)|}{|\l + f(z^\a)^\beta|^2 |z|^{1-\a}}  \\
&\le \frac{\a\beta (|\l|+|f(z^\a)|^\beta)^{1-1/\beta}  |f'(z^\a)|}{(\kappa(|\l| + |f(z^\a)|^\beta|)^2 |z|^{1-\a}}\\
&\le \frac{\a\beta f'(c_\a x^\a)} {\kappa^2 (|\l| + f(c_\a x^\a)^\beta)^{1+1/\beta} x^{1-\a}}. \\
\end{align*}
Now, making the change of variable $t = f(c_\a x^\a)^\beta$,
\begin{align*}
\int_0^\infty \sup_{y\in\R} |h'(x+iy)|\,dx &\le \frac{\a\beta}{\kappa^2} \int_0^\infty \frac{ f'(c_\a x^\a)x^{\a-1}}{ \left( |\l| + f(c_\a x^\a)^\beta \right)^{1+1/\beta}} \,dx \\
&\le \frac{1}{c_\a \kappa^2} \int_0^\infty \frac{1}{(|\l|+t)^{1+1/\beta}t^{1-1/\beta}} \,dt  \\
&= \frac{\beta}{c_\a \kappa^2 |\l|}.
\end{align*}
Noting that $\beta/(c_\alpha\kappa) > 1$, this completes the proof.
\end{proof}

\begin{remark}  Note that the constant $C$ in Proposition \ref{BB} is independent of the function $f$, provided that $f$ is Bernstein.

Although we have stated Proposition \ref{BB} for Bernstein functions, its proof uses only the properties (B1)-(B3).  Slightly more general properties would suffice, although the constant $C$ might then depend on the conditions.  For results in this direction, see \cite{BGT17}.
\end{remark}

\section{Functional calculus for $\Bes$}  \label{b-fc}

In this section $X$ is a complex Banach space, $X^*$ is its Banach dual, and the duality between $X$ and $X^*$ is written as $\langle x,x^* \rangle$ for $x \in X,\, x^* \in X^*$.   We shall write $z \in \C_+$ as $z = \a + i\b$ in this section, as we shall be using $x$ to denote vectors in $X$.

\subsection {Definition}  \label{sfcdef}

 Let $A$ be a closed operator on a Banach space $X$, with dense domain $D(A)$.  We assume that the spectrum $\sigma(A)$ is contained in $\overline\C_+$ and
\begin{equation} \label{8.1}
\sup_{\alpha >0} \alpha \int_{\mathbb R} |\langle (\alpha +i\beta + A)^{-2}x, x^* \rangle| \, d\beta <\infty
\end{equation}
for all $x \in X$ and $x^* \in X^*$.   By the Closed Graph Theorem, there is a constant $c$ such that
\begin{equation} \label{8.2}
 \frac{2}{\pi}\alpha \int_{\mathbb R} |\langle (\alpha +i\beta + A)^{-2}x, x^* \rangle| \, d\beta \le c \|x\|\,\|x^*\|
\end{equation}
for all $\a>0$, $x \in X$ and $x^* \in X^*$.  Note that (\ref{8.1}) says precisely that the function
\begin{equation} \label{resbov}
g_{x,x^*} : z \mapsto \langle (z+A)^{-1}x, x^* \rangle
\end{equation}
belongs to $\Bov$.  We let $\gamma_A$ be the smallest value of $c$ such that (\ref{8.2}) holds, so
\begin{equation}  \label{gammaa}
\gamma_A := \frac{2}{\pi} \sup \{ \|g_{x,x^*}\|_{\Bov_0} : \|x\|= \|x^*\|=1 \}.
\end{equation}

It was shown in \cite{Gom} and \cite{SF} that if $A$ satisfies the assumptions above then $-A$ is the generator of a bounded $C_0$-semigroup $(T(t))_{t \ge 0}$.  The Hille-Yosida conditions  can be obtained by applying Proposition \ref{space_e}(3) and using \eqref{8.2}.   Moreover, the following holds for all $x \in X$, $x^* \in X^*$ and $\alpha>0$:
\begin{equation} \label{gsff}
\langle T(t)x, x^* \rangle = \frac{1}{2\pi t} \int_\R \langle (\alpha+i\beta+ A)^{-2} x, x^* \rangle e^{(\alpha+i\beta)t} \, d\beta.
\end{equation}
This representation is not explicit in \cite{Gom} or \cite{SF}, but it is at the core of the theory.   To see that (\ref{gsff}) is true, multiply the equation by $t$, and then take Laplace transforms of each side with respect to $t$.  The resulting functions of $z$ are both $\langle (z+A)^{-2}x, x^* \rangle$, so uniqueness of Laplace transforms implies (\ref{gsff}) (see \cite[p.505]{CT} for the full argument).

Let
\begin{equation} \label{ka}
K_A := \sup_{t\ge0} \|T(t)\|.
\end{equation}
Putting $\a=1/t$ in (\ref{gsff}) and then using (\ref{gammaa}), we see that
\[
|\langle T(t)x,x^* \rangle| \le \frac{e}{2\pi} \|g_{x,x^*}\|_{\Bov_0} \le \frac{e \gamma_A}{4}\|x\|\,\|x^*\|.
\]
It follows that $1 \le K_A \le e\gamma_A/4$.   This implies that
\begin{equation} \label{res}
 \gamma_A \ge 4e^{-1} > 1,  \qquad
\|(z+A)^{-1}\| \le \frac{e \gamma_A}{4 \Re z}, \quad z \in \C_+.
\end{equation}

\begin{exas} \label{hilbert}
1.  If $-A$ generates a bounded $C_0$-semigroup on a Hilbert space $X$, then (\ref{8.1}) necessarily holds.  This was observed in \cite{Gom} and \cite{SF}, but we give the argument here.  Firstly, Plancherel's Theorem in the Hilbert space $L^2(\R_+,X)$ gives, for any $\a>0$ and $x, y \in X$,
\begin{align*}
\int_\R \|(\a+i\b+A)^{-1}x\|^2 \,d\b &= 2\pi \int_0^\infty e^{-2\a t} \|T(t)x\|^2 \, dt \le \frac{\pi K_A^2}{\a} \|x\|^2, \\
\int_\R \|(\a+i\b+A^*)^{-1}y\|^2 \,d\b &= 2\pi \int_0^\infty e^{-2\a t} \|T(t)^*y\|^2 \, dt \le \frac{\pi K_A^2}{\a} \|y\|^2.
\end{align*}
Then, by Cauchy-Schwarz,
\begin{equation} \label{GSFhil}
\int_\R \left| \langle (\a+i\b+A)^{-2}x, y \rangle \right| \, d\b \le  \frac{\pi K_A^2}{\a} \|x\|\,\|y\|.
\end{equation}
Hence
\begin{equation} \label{gaka}
\gamma_A \le 2K_A^2.
\end{equation}

\noindent 2. If $-A$ generates a (sectorially) bounded holomorphic $C_0$-semigroup, then a very simple argument shows that (\ref{8.2}) holds.  We will discuss this further and give an estimate of $\gamma_A$ in Section \ref{holsgs}.
\end{exas}

Assume that (\ref{8.1}) holds, so that the functions $g_{x,x^*}$ in (\ref{resbov}) are in $\Bov$, and let $f \in \Bes$.  We aim to define $f(A)$ by replacing  $z$ by $A$ and $r_z$ by $(z+A)^{-1}$ in the representation formula in Proposition \ref{BHP2}.  To ensure that the corresponding integrals are convergent, we need to work in the weak operator topology, and we use the duality between $\Bes$ and $\Bov$ considered in Section \ref{dual}.

Let $f \in \Bes$ and $A$ be as above.  Define
\begin{align}  \label{fcdef}
\lefteqn{\langle f(A)x, x^* \rangle} \quad\\
&=  f(\infty) \langle x, x^* \rangle - \frac{2}{\pi}  \int_0^\infty  \alpha \int_{\R}   \langle (\alpha -i\beta +A)^{-2}x, x^* \rangle  {f'(\alpha +i\beta)} \, d\beta\,d\alpha  \nonumber \\
&=  f(\infty) \langle x, x^* \rangle + \frac{2}{\pi} \langle g_{x,x^*}, f \rangle_\Bes,  \nonumber
\end{align}
for all $x \in X$ and $x^* \in X^*$.  It is easily seen that this defines a bounded linear mapping $f(A) : X \to X^{**}$, and that the linear mapping
\[
\Phi_A : \Bes \to \mathcal L(X,X^{**}), \qquad f \mapsto f(A),
\]
is bounded.  Indeed, by \eqref{boves} and \eqref{gammaa},
\begin{equation} \label{const}
\|f(A)\| \le |f(\infty)| + \gamma_A \|f\|_\Bq \le \gamma_A\|f\|_\Bes.
\end{equation}
 We will show that $f(A) \in \B(X)$ and $\Phi_A$ is an algebra homomorphism.   First we establish consistency of the definition in (\ref{fcdef}) with the Hille-Phillips (or HP-) calculus, by following the argument in Proposition \ref{BHP2} for $f \in \LT$.

\begin{lemma}\label{BHP3}
Let $f \in \LT$ and $f(A)$ be as defined as in {\rm (\ref{fcdef})}.   Then $f(A)$ coincides with the operator $x \mapsto \int_{\R_+} T(t)x \,d\mu(t)$ as defined in the Hille-Phillips functional calculus. In particular, $f(A) \in \B(X)$.
\end{lemma}

\begin{proof}
Since $f \in \LT$, $f = \lt\mu$ for some $\mu \in M(\R_+)$.  First assume that $f(\infty)=0$, or equivalently $\mu(\{0\})=0$.  Then
\begin{align*}
\frac{2}{\pi} \langle g_{x,x^*}, f \rangle_\Bes
&= \frac{2}{\pi} \int_0^\infty \alpha \int_\R  \langle (\a-i\b+A)^{-2}x, x^* \rangle  \int_0^\infty t e^{-(\alpha + i\beta)t}\, d\mu(t)\, d\beta\,d\alpha \\
&= \frac{2}{\pi}\int_0^\infty \int_0^\infty \alpha t e^{-\alpha t} \int_\R \langle (\a-i\b+A)^{-2}x, x^* \rangle e^{- i\beta t}   \,d\beta \,   \,d\alpha   \,d\mu(t) \\
&= 4 \int_0^\infty \int_0^\infty  \alpha t^2  e^{-2\alpha t} \langle T(t)x,x^* \rangle \,d\alpha\,d\mu(t) \\
&= \int_{0}^\infty \langle T(t)x,x^* \rangle  \,d\mu(t),
\end{align*}
where we have used that $\beta \mapsto \langle (\alpha-i\beta+A)^{-2}x,x^* \rangle$ is the inverse Fourier transform of $t \mapsto 2\pi t e^{-\alpha t} \langle T(t)x,x^* \rangle$ on $\R_+$ (extended to $\R$ by $0$), and both functions are in $L^1(\R)$.   We now have the required result when $f(\infty)= 0$.  The general case follows by applying the above case to $f - f(\infty)$.
\end{proof}

In particular, Lemma \ref{BHP3} implies that $(\l+A)^{-1} = r_\l(A)$, where $\l\in\C_+$, and $T(t) = e_t(A)$, where $e_t(z) = e^{-tz}$.   Consequently, from here on we will write $e^{-tA}$ in place of $T(t)$.

Our next step is a version of Lemma \ref{Sconv} for operators.  Let
\[
\eta(z) = \frac{1-e^{-z}}{z}, \qquad \eta_\delta(z) = \eta(\delta z), \quad \delta>0.
\]
As in Example \ref{LTex}(\ref{LTex4}) and Lemma \ref{shifts01}(\ref{012}), $\|\eta_\delta\|_\Bes = \|\eta\|_\Bes = 2$.

\begin{lemma} \label{Sext}
Let $f \in \Bq$, and assume that $ f \eta_\delta \in \LT$ for each $\delta>0$.
 Then
$\lim_{\delta\to0+} (f \eta_\delta)(A)x$ exists in $X$ for every $x \in X$.
 Moreover  $f(A)x=\lim_{\delta\to0+} (f \eta_\delta)(A)x$ for all $
 x \in X,$ and thus $f(A)\in L(X)$.
\end{lemma}

\begin{proof}
Let $r_1(z) = (1+z)^{-1}$, so $r_1 = \lt e_1$ where $e_1(t) = e^{-t} \in L^1(\R_+)$, and $r_1(A) = (1+A)^{-1}$.   For $\delta>0$, $\eta_\delta$ is the Laplace transform of $\delta^{-1} \chi_{[0,\delta]}$, and these functions form an approximate unit for $L^1(\R_+)$ as $\delta\to0+$.  Since the Laplace transform is a bounded algebra homomorphism from $L^1(\R_+)$ to $\Bes$, it follows that
$$
\|\eta_\delta r_1 - r_1\|_\Bes \to 0  \qquad (\delta\to0+).
$$
Since $f\eta_\delta \in \LT$ and the HP-calculus is multiplicative on $\LT$, it follows from Lemma \ref{BHP3} that
$$
(f \eta_\delta r_1)(A) = (f \eta_\delta)(A) r_1(A) = (f \eta_\delta)(A) (1+A)^{-1}.
$$
For $x \in D(A)$, we obtain
$$
(f \eta_\delta)(A)x = (f \eta_\delta r_1)(A) (1+A) x \to (f r_1)(A) (1+A) x   \qquad (\delta\to0+).
$$
Moreover, $\|(f \eta_\delta)(A)\| \le 2 \gamma_A \|f\|_\Bes$ by \eqref{const}, and $D(A)$ is dense in $X$, so it follows that $\lim_{\delta\to0+}  (f \eta_\delta)(A)x$ exists in $X$ for all $x \in X$.

By Lemma \ref{Sconv},  for $x \in X$ and $x^* \in X^*$, we now have
$$
\langle \lim_{\delta\to0+} (f\eta_\delta)(A)x, x^*\rangle = \lim_{\delta\to0+} \langle g_{x,x^*}, f\eta_\delta \rangle_\Bes = \langle g_{x,x^*}, f \rangle_\Bes = \langle f(A)x,x^* \rangle.
$$
Then $f(A) = \lim_{\delta\to0+} (f\eta_\delta)(A) \in \B(X)$.
\end{proof}

Now we present the main result showing that the map $\Phi_A$ has the essential properties of a bounded functional calculus.  We shall subsequently refer to $\Phi_A$ as the \emph{$\Bes$-calculus} for $A$.

\begin{thm} \label{besc}
Under the assumptions {\rm (\ref{8.1})}, the map $\Phi_A : f \mapsto f(A)$ is a bounded algebra homomorphism from $\Bes$ into $\B(X)$, which extends the Hille-Phillips calculus.  Moreover, $\|\Phi_A\| \le \gamma_A$.
\end{thm}

\begin{proof}
The mapping $\Phi_A: f \mapsto f(A)$ is bounded from $\Bes$ into $\B(X,X^{**})$.  If $f \in \ssp$, then $f\eta_\delta \in \LT$ by Lemma \ref{PW}, so Lemma \ref{Sext} applies to all such $f$.  In particular, $\Phi_A$ maps $\ssp$ into $\B(X)$.  Since $\ssp$ is norm-dense in $\Bq$, it follows that $\Phi_A(f) \in \B(X)$ whenever $f(\infty)=0$.   Since $\Phi_A(1) = I$, it follows that $\Phi_A$ maps $\Bes$ into $\B(X)$.

To show that $\Phi_A$ is multiplicative, it suffices to consider $f,g \in \ssp$.  Then $fg \in \ssp$. Take $\delta, \delta' > 0$.  Since $f \eta_\delta, g \eta_{\delta'} \in \LT$,
\[
(f\eta_\delta g \eta_{\delta'})(A) = (f\eta_\delta)(A) (g \eta_{\delta'})(A).
\]
Letting $\delta' \to 0+$ with $\delta$ fixed and using Lemma \ref{Sext}, we obtain $(f\eta_\delta g)(A) = (f \eta_\delta)(A) g(A)$.  Letting $\delta\to0+$ gives that $(fg)(A) = f(A)g(A)$.

Lemma \ref{BHP3} shows that this functional calculus agrees with the HP-calculus on $\LT$.  The final statement follows from (\ref{const}).
\end{proof}

The $\Bes$-calculus defined by (\ref{fcdef}) and Theorem \ref{besc} applies to all operators satisfying (\ref{8.1}), and to all functions in $\Bes$.  In particular it applies to all generators of bounded $C_0$-semigroups on Hilbert space with estimates $\|f(A)\|\le 2K_A^2 \|f\|_\Bes$, by (\ref{gaka}).  In that context White \cite[Section 5.5]{White} and  Haase \cite[Theorem 5.3, Corollary 5.5]{Haase} independently had obtained estimates of the form $\|f(A)\| \le CK_A^2 \|f\|_\Bes$ but only for subclasses of $\Bes$ which are not norm-dense (for $\LT$ in \cite{Haase}, and a smaller class in \cite{White}).   Moreover they used an equivalent norm arising from Littlewood-Paley decompositions (see the Appendix).   We obtain the estimate for all functions in $\Bes$.   Applications to bounded $C_0$-semigroups on Hilbert spaces are given in Section \ref{apps}.

We note two other simple properties of the $\Bes$-calculus.

\begin{lemma} \label{shifts2}
Let $a>0$, $f \in \Bes$, $f_a(z) = f(z+a)$, $g_a(z) = f(az)$.
\begin{enumerate}[\rm1.]
\item $f(A+a) = f_a(A)$.
\item $f(aA) = g_a(A)$.
\end{enumerate}
\end{lemma}

\begin{proof}
The first statement follows from (\ref{fcdef}) and an application of (\ref{shifts22}) with $g = g_{x,x^*}$.   The second statement is a simple change of variables in (\ref{fcdef}): $t = a\a, \, s=a\b$.
\end{proof}

We shall consider some more general properties of the $\Bes$-calculus in Sections \ref{compat2} (compatibility with other calculi), \ref{conver} (convergence lemmas) and \ref{spinc} (spectral inclusion).

\subsection{Bounded holomorphic semigroups} \label{holsgs}
 Recall that a densely defined operator $A$ on a Banach space $X$ is {\it sectorial of angle} $\theta \in [0,\pi/2)$ if $\sigma(A) \subset \overline \Sigma_\theta$ and, for each $\theta' \in (\theta,\pi]$, there exists $C_{\theta'}$ such that
\begin{equation} \label{R01}
\|z(z+A)^{-1}\| \le C_{\theta'}, \qquad z \in \Sigma_{\pi-\theta'}.
\end{equation}
We shall write $\Sect(\theta)$ for the class of all sectorial operators of angle $\theta$ for $\theta \in [0,\pi/2)$ on Banach spaces, and $\Sect(\pi/2-) := \bigcup_{\theta \in [0,\pi/2)} \Sect(\theta)$.   Then $A \in \Sect(\pi/2-)$ if and only if  $-A$ generates a (sectorially) bounded holomorphic $C_0$-semigroup on $X$.  Some authors call these semigroups {\it sectorially} bounded holomorphic semigroups, but we shall adopt the common convention that bounded holomorphic semigroups are bounded on sectors.   We refer to \cite{HaaseB} for the general theory of sectorial operators, and to \cite[Section 3.7]{ABHN} for the theory of holomorphic semigroups.

Let $A$ be an operator and  assume that $\sigma(A) \subset \overline \C_+$ and
\begin{equation}\label{R11}
M_A := \sup_{z\in \C_+} \|z(z+A)^{-1}\| < \infty.
\end{equation}
It follows from (\ref{R11}) and Neumann series (see \cite[Lemma 1.1]{V1}) that $\sigma(A) \subset \Sigma_\theta \cup \{0\}$ and
\begin{equation}\label{Nv1}
\|z(z+A)^{-1}\|\le 2M_A,\quad 0\not=z\in \Sigma_{\pi-\theta},
\end{equation}
where
\[
\theta:=\arccos(1/(2M_A))<\pi/2.
\]
So $A \in \Sect(\theta) \subset \Sect(\pi/2-)$.   Conversely, if $A \in \Sect(\theta)$ where $\theta \in [0,\pi/2)$, then \eqref{R11} holds, with $M_A$ equal to $C_{\pi/2}$ in (\ref{R01}).   Thus $-A$ generates a bounded holomorphic semigroup if and only if $\sigma(A) \subset \overline\C_+$ and (\ref{R11}) holds.  In that case, $M_A$ is a basic quantity associated with $A$, which we call the {\it sectoriality constant} of $A$.

Before discussing bounded holomorphic semigroups in general,  we show that the semigroup $(T_\Bes(a))_{a\in\C_+}$ of shifts is a bounded holomorphic semigroup on $\Bes$, as mentioned after Lemma \ref{shifts01}.   A short proof of holomorphy appeals to abstract theory.  Since the functionals $f \mapsto f(z)$  for $z \in \C_+$ form a separating subspace of $\Bes^*$ and since the functions $a \mapsto (T_\Bes(a)f)(z)  = f(z+a)$ are holomorphic on $\C_+$, and $T_\Bes(a)$ is a contraction, it follows from \cite[Theorem A.7]{ABHN} that $T_\Bes$ is holomorphic on $\C_+$.  Another proof proceeds by applying Proposition \ref{shifts01}(2),(3) and general semigroup theory (see \cite[Section 3.9]{ABHN}).

We now give a more explicit proof of holomorphy exhibiting some estimates which may be useful for other purposes.

\begin{prop} \label{shiftsg}
The family $(T_\Bes(a))_{a\in\C_+}$ of shifts is a holomorphic $C_0$-semigroup of contractions on $\Bes$.
\end{prop}

\begin{proof}
We will show directly that $T_\Bes : \C_+ \to L(\Bes)$ is holomorphic.  Let $f \in \Bes$  and $a \in \C_+$. Since $T_\Bes(a)f \in H^\infty_\omega$ for $\omega = \Re a$, it follows from Lemma \ref{deriv} that $(T_\Bes(a)f)' \in \Bes$.  It remains to show that
\[
\lim_{h\to0} \left\| h^{-1} (T_\Bes(a+h)f - T_\Bes(a)f) - (T_\Bes(a)f)' \right\|_\Bes = 0.
\]
Since the operators $T_\Bes(is)$ are invertible isometries on $\Bes$, we may assume that $a>0$.

Let $h = re^{i\theta}$ where $0< r < a/2$, and let
\[
g_h(z) =  \frac{ f(z+a+h) - f(z+a)}{h} - f'(z+a) = \frac{e^{i\theta}}{r} \int_0^r (r-s) f''(z + a + se^{i\theta}) \,ds.
\]
Then by Lemma \ref{ho}(\ref{derbd})
\[
|g_h(z)| \le \frac{|h|}{2} \sup_{|\l-(z+a)|<a/2} |f''(\l)|  \le \frac{4|h|\|f\|_\infty}{\pi a^2}.
\]
By the same argument applied to $g_h'$,
\[
|g_h'(z)| \le \frac{|h|}{2} \sup_{|\l-(a+z)|<a/2} |f'''(\l)| \le  \frac{4|h|}{\pi a^2} \sup_{\Re \l = \Re z+ a/2} |f'(\l)|.
\]
It follows that
\[
\int_0^\infty \sup_{\Re z = x} |g_h'(z)| \, dx \le \frac{4|h|}{\pi a^2} \int_0^\infty \sup_{\Re \l = x+a/2} |f'(\l)| \, dx.
\]
Hence
\[
\|g_h\|_\Bes \le \frac{4|h|}{\pi a^2} \left(\|f\|_\infty + \|f\|_{\Bes_0}\right) = \frac{4|h|}{\pi a^2}  \|f\|_\Bes.
\]
Thus $T$ is a holomorphic function from $\C_+$ to $L(\Bes)$.
\end{proof}

We return to the general theory.   Let $A \in \Sect(\pi/2-)$.  Then (\ref{8.1}) holds, because $\|(\a+i\b+A)^{-1}\| \le M_A|\a+i\b|^{-1}$, where $M_A$ is the sectoriality constant of $A$.  This easily leads to (\ref{8.2}) with $\gamma_A \le 2M_A^2$.  We will give a sharper estimate in Corollary \ref{sch} based on the following integral estimate.

\begin{lemma}\label{EsRR}
Let $A\in \Sect(\pi/2-)$.  Then
\begin{equation*}\label{R112}
\a\int_\R \|(\a+i\beta+A)^{-2}\|\,d\beta\le 4(2\pi + 3 \log2) M_A(\log M_A+1), \quad \a>0.
\end{equation*}
\end{lemma}

\begin{proof}
Let $\theta = \arccos(1/(2M_A)) < \pi/2$.  Set
\[
z= \a + i\beta = r e^{i\phi}, \;\;r>0,\;\; \phi\in (-\pi/2,\pi/2),\quad \a>0,\quad \beta\in \mathbb R.
\]
We use the representation
\[
(z+A)^{-2} = - \frac{1}{2\pi i}\int_\gamma \frac{(\lambda-A)^{-1}}{(\lambda+z)^2} \,d\lambda ,
\]
({\it cf}.\ (\ref{P2})), where
\[
\gamma=\gamma_{+}\cup\gamma_0\cup\gamma_{-},
\quad \gamma_{\pm}=\{\rho e^{\pm i\theta}:\,r/2\le \rho<\infty\},
\]
and
\[
\gamma_0=\{re^{i\psi}/2:\,\psi\in [\theta, 2\pi-\theta]\}.
\]
Using (\ref{Nv1}), we have
\[
\|(z+A)^{-2}\|\le\frac{M_A}{\pi} \big( J_{-}(z)+J_{+}(z)+J_0(z) \big),
\]
where
\begin{gather*}
J_0(z)=\int_\theta^{2\pi-\theta}\frac{d\psi}{|re^{i\psi}/2+re^{i\phi}|^2} = \frac{4}{r^2}\int_\theta^{2\pi-\theta}\frac{d\psi}{|e^{i(\psi-\phi)}+2|^2}, \\
J_{\pm}(z)=\int_{r/2}^\infty \frac{d\rho}{\rho|\rho e^{\pm i\theta} +z|^2}.
\end{gather*}
We estimate these integrals separately.

First, note that
\[
J_0(z)\le\frac{8\pi}{r^2} = \frac{8\pi}{\a^2+\beta^2}
\]
and
\[
\a\int_{\R} J_0(\a+i\beta)\,d\beta
\le 8\pi \a\int_\R \frac{d\beta}{\a^2+\beta^2}= 8\pi^2.
\]
For $\rho > r/2$,
\[
\rho>\frac{\rho+r}{3}\ge \frac{\rho+\a}{3},
\]
and we obtain
\begin{align*}
\int_\R J_{\pm}(\a+i\beta)\, d\beta
&=
\int_{-\infty}^\infty
\int_{r/2}^\infty \frac{d\rho}{\rho|\rho e^{\pm i\theta}+\a+i\beta|^2}\,d\beta\\
&\le
3\int_0^\infty \left(
\int_{-\infty}^\infty \frac{d\b}{|\rho e^{\pm i\theta}+\a +i\b|^2}\right)\,\frac{d\rho}{\rho+\a}.
\end{align*}
Observe that
\begin{align*}
\int_\R \frac{d\b}{|\rho e^{\pm i\theta}+\a+i\b|^2}
&=\int_\R \frac{d\b}{(\rho\cos\theta+\a)^2 +(\pm\rho\sin\theta+\b)^2}\\
&=\int_\R \frac{d\b}{(\rho\cos\theta+\a)^2 +\b^2}\\
&= \frac{\pi}{\rho\cos\theta+\a}.
\end{align*}
Thus,
\begin{align*}
\int_\R J_{\pm}(\a+i\beta)\,d\beta &\le
3\pi\int_0^\infty \frac{d\rho}{(\rho+\a)(\rho\cos\theta+\a)} \\
&=\frac{3\pi}{\a\cos\theta}\int_0^\infty \frac{d\rho}{(\rho+1)(\rho+\sec\theta)} \\
&=\frac{3\pi}{\a(1-\cos\theta)} \log (\sec\theta).
\end{align*}
Taking into account that $\cos\theta=1/(2M_A)$ and $M_A\ge 1$,
\[
\a\int_\R J_{\pm}(\a+i\beta)\,d\beta\le
\frac{3\pi}{1-1/(2M_A)}\log (2M_A)\le 6\pi \log(2M_A).
\]
Thus,
\begin{align*}
\a\int_\R \|(\a+i\beta+A)^{-2}\|\,d\beta
&\le 4M_A \left(2\pi+3\log(2M_A)\right) \\
& \le 4(2\pi + 3 \log 2)M_A(1 + \log M_A).
  \qedhere
\end{align*}
\end{proof}

From this we deduce the following result which was first obtained by Schwenninger \cite[Theorem 5.2]{Sch1} with an estimate of the same form as (\ref{V31}) (up to a multiplicative constant).  Previously Vitse \cite[Theorem 1.7]{V1} had obtained an estimate $\|f(A)\| \le 31M_A^3\|f\|_\Bes$. The arguments in \cite{Sch1} and \cite{V1} were based on the sectorial functional calculus and Littlewood-Paley decompositions (see the Appendix).

\begin{cor}  \label{sch}
Let $A \in \Sect(\pi/2-)$.  Then
\begin{equation} \label{V30}
\gamma_A \le 8(2 + \log 2)M_A (\log M_A + 1).
\end{equation}
Hence, for all $f \in \Bes$,
\begin{equation} \label{V31}
\|f(A)\| \le 8(2 + \log 2) M_A (\log M_A+1) \|f\|_\Bes.
\end{equation}
Moreover,
\begin{equation} \label{V32}
f(A) =   f(\infty)
- \frac{2}{\pi} \int_0^\infty\int_{\R}   \alpha (\a-i\b+A)^{-2} {f'(\alpha +i\beta)} \, d\beta\,d\alpha,
\end{equation}
where the integral converges in the operator norm.
\end{cor}

\begin{proof}
Lemma \ref{EsRR} shows that (\ref{V30}) holds, and the estimate (\ref{V31}) follows immediately.   The  convergence in operator-norm of the integral in (\ref{V32}) follows from either the estimate $\|(\a -i\beta +A)^{-2}\| \le M_A^2 (\a^2+\beta^2)^{-1}$ or from Lemma \ref{EsRR}, and then (\ref{V32}) follows from (\ref{fcdef}).
\end{proof}

Some applications of the $\Bes$-calculus for operators in $\Sect(\pi/2-)$ are discussed in Section \ref{apps}.

\subsection{Compatibility with other calculi}  \null\label{compat2}
It is important to compare the $\mathcal B$-calculus with the other calculi in the literature.
Here the rule of thumb is that all calculi are compatible whenever
they are well-defined, and the $\mathcal B$-calculus does not deviate
from that general principle, as we will see below.

We first remark that according to Theorem \ref{besc}
the $\mathcal B$-calculus is a strict extension of the HP-calculus.
Thus all of the compatibility results for the HP-calculus are valid
in the setting of the $\mathcal B$-calculus restricted to the subalgebra $\LT$ of $\Bes$.
Compatibility results for the HP-calculus can be found in \cite[Proposition 3.3.2]{HaaseB}, \cite{KuWe} and  \cite[Sections 4,5]{BGTMZ}, for example.

In this section we will show that the $\Bes$-calculus is compatible with two other functional calculi, namely the classical sectorial holomorphic functional calculus (see \cite{HaaseB} and \cite{KuWe}) and the half-plane holomorphic functional calculus (see \cite{BHM}).  Their constructions involve a primary functional calculus defined by contour integration, followed by an extension procedure, and the resulting operators may themselves be unbounded.  We refer to \cite[Chapter 1]{HaaseB} for the general background of this theory of functional calculi for unbounded operators, in particular, the notions of primary and extended functional calculus, and the use  of regularisers in the extension procedure.  For the definitions of the primary calculus and the properties of the extended sectorial and half-plane calculi, we refer to \cite[Chapter 2, etc] {HaaseB} and \cite{BHM} respectively.   We shall
 let $f \mapsto \Psi_A(f)$ stand for the (extended) sectorial holomorphic functional calculus, and $f \mapsto \Psi^*(f)$ stand for the (extended) half-plane calculus  whenever they are defined as closed operators.

In considering compatibility, we have to take into account that functions in $\Bes$ are bounded functions  on $\overline{\C}_+$ while the functions appearing in the sectorial calculus are defined on open sectors properly containing $\sigma(A) \setminus \{0\}$, and functions in the half-plane calculus are defined on open half-planes containing $\sigma(A)$.     Consequently for compatibility with each of the two calculi, we consider two situations; one where the spectrum of $A$ is smaller than $\overline{\C}_+$, and one where $f \in \Bes$ extends holomorpically to a larger domain.  In the latter situation, the extension of $f$ may not be bounded.     The techniques of proof follow the pattern of an abstract result in \cite[Proposition 1.2.7]{HaaseB}, but we present them here in more concrete forms.   In particular, we do not present the details in the most general possible form, using arbitrary regularisers which depend on $A$ and $f$, but we present the results for classes of functions $f$ and specific regularisers which do not depend on  $A$.   The proofs can easily be extended to general regularisers.

Thus we give four compatibility results.  The two results for the sectorial calculus are Proposition \ref{sector_comp} where $A$ is sectorial of angle less than $\pi/2$, and Proposition \ref{compat3} for functions in  $\Bes$ which extend holomorphically to a sector of angle greater than $\pi/2$.    The two results for the half-plane calculus are both covered by Proposition \ref{compat}, as it is easy to pass from one to the other, using shifts and replacing  $A$ by $A+\eta$ and correspondingly for the functions.

For $A \in \Sect(\pi/2-)$, a complication is that a Besov function is not necessarily in the domain of  $\Psi_A$.
  If $A$ is injective  then $\Bes$ is contained in the domain of $\Psi_A$, but this is not true if $A$ is not injective (see \cite[Lemma 2.3.8]{HaaseB}, \cite[Example 5.2]{Haaed}).  In that case the domain of the extended calculus is quite small, and there are technical problems when the sector is smaller than $\C_+$.   Consequently our first result is confined to injective sectorial  operators.     It shows that the $\Bes$-calculus and the extended sectorial calculus are compatible in that case.   Variants of this result were shown by Vitse in \cite{V1}, where the construction of the calculus for Besov functions was based on the sectorial calculus.   It is possible to establish compatibility of our $\Bes$-calculus and Vitse's calculus, and hence with sectorial calculus, initially on $\ssp$ and then by approximation arguments on $\Bes$.  Here we give a more direct argument.

\begin{prop}\label{sector_comp}
Let $A \in \Sect(\pi/2-)$ be injective.   Then $\Phi_A(f) = \Psi_A(f)$ for all $f \in \Bes$.
\end{prop}

\begin{proof}
Let $A$ be an injective operator in $\Sect(\theta)$, where $\theta \in [0,\pi/2)$.   Then, for any $\theta'\in (\theta,\pi/2)$,
\begin{equation*}\label{RR}
\|A(I+A)^{-1}(z-A)^{-1}\| = \|(z(z-A)^{-1}-1)(1+A)^{-1}\| \le \frac{C_{\theta'}}{1+|z|},\quad z\notin \Sigma_{\pi-\theta'}.
\end{equation*}

Let $f \in \Bes$, and let $g = r_1(1-r_1)^2f  \in \Bes  \cap \mathcal A_{\pi/2}$, where $\mathcal A_{\pi/2}$ is given by \eqref{regular}.  It follows from the definition of $\Psi_A(g)$, Proposition \ref{BHP2} and Fubini's theorem that
\begin{multline*}
 A(I+A)^{-1}\Psi_A(g) = \frac{1}{2\pi i} \int_{\partial \Sigma_{\theta'}} A(I+A)^{-1}(z-A)^{-1} g(z) \, dz  \\
= - \frac{1}{\pi^2 i}
\int_0^\infty \int_{\mathbb R}\left(\int_{\partial\Sigma_{\theta'}} A(1+A)^{-1}(z-A)^{-1}(\alpha-i\beta+z)^{-2} \,dz\right)
\alpha g'(\alpha+i\beta)\,d\beta d\alpha.
\end{multline*}
From the primary functional calculus for invertible sectorial operators, we obtain
\[
\int_{\partial\Sigma_{\theta'}} A(1+A)^{-1} (z-A)^{-1}(\alpha-i\beta+z)^{-2} \,dz=  2\pi i A(1+A)^{-1} (\alpha-i\beta+A)^{-2},
\]
and then, from Corollary \ref{sch},
\begin{align*}\label{equal}
A(1+A)^{-1}\Psi_A(g) &=-\frac{2}{\pi}
\int_0^\infty \int_{\mathbb R}A(1+A)^{-1}(\alpha-i\beta+A)^{-2}\,
\alpha g'(\alpha+i\beta)\,d\beta d\alpha  \\
&= A(1+A)^{-1} \Phi_A(g) = A^2(1+A)^{-3} \Phi_A(f).
\end{align*}
Since $A(1+A)^{-1}$ is injective this implies that $\Psi_A(g) = A(1+A)^{-2}\Phi_A(f)$.   It follows from this and the definition of the extended sectorial calculus that $\Psi_A(f)$ is defined on $X$ and $\Psi_A(f) = \Phi_A(f)$.
\end{proof}

In particular, Proposition \ref{sector_comp} implies that the Composition Rule \cite[Theorem 2.4.2]{HaaseB} for
the sectorial functional calculus can be used within the  $\mathcal B$-calculus for injective $A \in \operatorname{Sect}(\pi/2-)$.

Another compatibility result between the $\Bes$-calculus and the sectorial calculus is the following.  Here we consider functions defined on $\Sigma_\theta$ where $\theta \in (\pi/2,\pi)$.  In this case the domain of the primary calculus is contained in $\Bes$, which simplifies the arguments.  We again denote the corresponding operators by $\Phi_A(f)$ and $\Psi_A(f)$, respectively.\footnote{We are grateful to one of the referees for providing the proof given here, as our original proof was more complicated.}

\begin{prop} \label{compat3}
Assume that $A$ satisfies {\rm(\ref{8.1})}.  Let $f \in \Bes$ and assume that $f$ extends to a function in $ \operatorname{Hol}(\Sigma_\theta)$ for some $\theta \in (\pi/2,\pi)$ and $\Psi_A(f)$ is defined.  Then  $\Phi_A(f) = \Psi_A(f)$.
\end{prop}

\begin{proof}
For $\theta \in (\pi/2,\pi]$, $\mathcal{A}_\theta \subset \LT$, by \cite[Lemma 3.3.1]{HaaseB} or \cite[Theorem 2.6.1]{ABHN}.  If $f \in \mathcal{A}_\theta$, then $\Phi_A(f)$ and $\Psi_A(f)$ both agree with the Hille-Phillips  calculus applied to $f$, by Lemma \ref{BHP3} and  \cite[Proposition 3.3.2]{HaaseB}, and hence $\Psi_A(f) = \Phi_A(f)$.

Now let $f$ be in the domain of the extended sectorial calculus $\Psi_A$.  Then there is a regulariser  $h\in \mathcal{A}_\theta \subset \Bes$ such that $\Psi_A(h)$ is injective and $h f \in \mathcal{A}_\theta$.    Then
\[
\Psi_A(h)\Phi_A(f) = \Phi_A(h) \Phi_A(f) = \Phi_A(h f) = \Psi_A(h  f).
\]
It follows from this that $\Psi_A(f)$ is defined on $X$ and coincides with $\Phi_A(f)$.
\end{proof}

When  $A$ is injective and $f$ is bounded on $\Sigma_\theta$ one may take the regulariser $h$ in the proof of Proposition \ref{compat3} to be $r_1(1-r_1)$.   Using instead $\left(r_1(1-r_1)\right)^k$ for $k \ge1$, one can apply the result if $f$ is polynomially bounded on $\Sigma_\theta$.

We now consider compatibility of the $\Bes$-calculus $\Phi_A$ with the half-plane holomorphic calculus  $\Psi^*_A$.    Assume that $\sigma(A) \subset \overline{\C}_+$, (\ref{8.1}) holds, and $f \in \Bes$.  We may also assume either of the following additional conditions:
\begin{enumerate}
\item[\rm (C1)] \label{HI} Assume that for some $\ep>0$, $\sigma(A) \subset \RR_\ep := \{z\in \C: \Re z > \ep\}$ and $\|(z+A)^{-1}\|$ is bounded for $z \in \RR_{-\ep}$; or
\item[\rm (C2)]  \label{HII} Assume that $f \in \operatorname{Hol}(\RR_{-\ep})$ for some $\ep>0$, there exist $C,\a\ge0$ such that $|f(z)| \le C(1+|z|)^\a$ for $z \in \RR_{-\ep}$ (or more generally, that $\Psi^*_A(f)$ is defined).
\end{enumerate}
The conditions (C1) and (C2) are, loosely speaking, half-plane versions of the conditions in Propositions \ref{sector_comp} and \ref{compat3}, respectively.   However there are some differences from the sectorial situation.  Firstly, $\Bes$ is contained in the domain of the extended half-plane calculus in both cases, irrespective of whether $A$ is injective.   Moreover it is easy to pass between the conditions (C1) and (C2) by shifts.   

\begin{prop} \label{compat}
Assume that $\sigma(A) \subset \overline{\C}_+$, {\rm(\ref{8.1})} holds, $f \in \Bes$ and either {\rm(C1)} or {\rm(C2)} above holds.  Then $\Psi_A^*(f) = \Phi_A(f)$.
\end{prop}

\begin{proof}  First, we assume that (C1) holds.    Let $f \in \Bes$ and $g = r_1^2 f$.   Then $g \in \Bes$, $|g(z)| \le C(1+|z|^2)^{-1}$ and $|g'(z)| \le C (\Re z)^{-1} (1+|z|^2)^{-1}$ for some $C$.

 Let $\eta \in (0,\ep)$.  It follows from the definition of $\Psi_A^*(g)$, Proposition \ref{BHP2} and Fubini's theorem that
\begin{multline*}
 \Psi_A^*(g) = \frac{1}{2\pi i} \int_{\Re z = \eta} (z-A)^{-1} g(z) \, dz  \\
= - \frac{1}{\pi^2 i}
\int_0^\infty \int_{\mathbb R}\left(\int_{\Re z = \eta} (z-A)^{-1}(\alpha-i\beta+z)^{-2} \,dz\right)
\alpha g'(\alpha+i\beta)\,d\beta \, d\alpha.
\end{multline*}
From the primary functional calculus for half-plane operators, we obtain
\[
\int_{\Re z = \eta}  (z-A)^{-1}(\alpha-i\beta+z)^{-2} \,dz=  2\pi i  (\alpha-i\beta+A)^{-2}.
\]
For $x \in X$ and $x^* \in X^*$, we have
\begin{align*}\label{equal}
\langle \Psi_A^*(g)x, x^*\rangle  &=-\frac{2}{\pi}
\int_0^\infty \int_{\mathbb R} \langle (\alpha-i\beta+A)^{-2}x,x^* \rangle \,
\alpha g'(\alpha+i\beta)\,d\beta d\alpha  \\
&= \langle \Phi_A(g)x, x^* \rangle.
\end{align*}
Hence $\Psi_A^*(g) = \Phi_A(g) = (1+A)^{-2} \Phi_A(f)$.  It follows from this that $\Psi_A^*(f)$ is defined on $X$ and $\Psi_A^*(f) = \Phi_A(f)$.

Next we assume that (C2) holds for some $\ep>0$ (but (C1) does not hold), and initially we assume that $f \in H^\infty(\RR_{-\ep})$.  Let $\eta \in (0,\ep/2)$.   The operator $A+\eta$ satisfies (C1) with $\ep$ replaced by $\eta$.  From the case above we obtain that
\[
\Phi_{A+\eta}(f) = \Psi_{A+\eta}^*(f).
\]
By Lemmas \ref{shifts2} and \ref{shifts01}, $\Phi_{A+\eta}(f) = \Phi_A(f_\eta)$.  Moreover $\Psi_{A+\eta}^*(f) = \Psi_A^*(f_\eta)$ by a very routine argument, as follows.  For the primary functional calculus where $\Psi_A^*(f)$ is defined by integration on a vertical line, it is a consequence of Cauchy's theorem and an instance of the fact that the integral does not depend on the choice of the vertical line.   The extension to more general functions is a triviality (see \cite[Proposition 1.2.7]{HaaseB}). Now let $\eta\to0+$.  By Lemma \ref{shifts01}(1), $\Phi_{A}(f_\eta) \to \Phi_A(f)$ in the operator norm.  Moreover $(f_\eta)$ is a uniformly bounded family converging pointwise to $f$ on $\RR_{-\ep}$, and the family $\{\Psi_A^*(f_\eta) : \eta\in(0,\ep/2)\}$ is uniformly bounded in $\B(X)$.  By the Convergence Lemma for the half-plane calculus \cite[Theorem 3.1]{BHM}, $\Psi_A^*(f_\eta) \to \Psi_A^*(f)$ in the strong operator topology.   Hence $\Psi_A^*(f) = \Phi_A(f)$, as required.

Next, we assume that (C2) holds and $f$ is polynomially bounded in $\RR_{-\ep}$.  Then  we apply the case above with $f$ replaced by $r_1^n f$ for sufficiently large $n$, and the result follows.   In the more general case, $f$ would be replaced by $hf$ for some regulariser $h$.
\end{proof}

\begin{rems}
\noindent 1.  An alternative proof\footnote{kindly pointed out to us by a referee} of case (C2) proceeds as follows.  Let $f$ satisfy (C2), and let $g = r_1^2 f$.  Then $|g(z)| \le C(1+|z|^2)^{-1}$ and $|g''(z)| \le C(1+|z|^2)^{-1}$ for $z \in \overline{\C}_+$.  Then $\mathcal{F}^{-1}g^b \in L^1(\R)$ and has support in $\R_+$, so $g \in \LT$.    By \cite[Theorem 8.20]{ISEM21}, $\Psi^*_A(g)$ coincides with the Hille-Phillips calculus of $g$ and hence with $\Phi_A(g)$.   Then $\Psi_A^*(f) = \Phi_A(f)$.

\noindent 2.
Proposition \ref{compat} suggests an alternative way to define the $\Bes$-calculus.  For $f \in \Bes$, one could define $f(A+\ep)$ by the half-plane functional calculus, and then show that $f(A+\ep)$ satisfies (\ref{fcdef}).  This implies that $f(A+\ep) \in L(X)$ and $\|f(A+\ep)\| \le C\|f\|_\Bes$ for $f \in \Bes$.  Then one can conclude from Lemma \ref{shifts01} that
\begin{equation*} \label{altdef}
 \lim_{\ep\to0+} f(A+\ep)
\end{equation*}
exists in the operator norm, and this could become the definition of $f(A)$.  This process reverses some of the steps that we have taken.
\end{rems}

\subsection{A Convergence Lemma} \label{conver}

The Convergence Lemma in the holomorphic functional calculus of sectorial operators \cite[Proposition 5.1.4]{HaaseB} is a result of a Tauberian character.  Assume, for simplicity,  that $A \in \Sect(\pi/2-)$ has dense range, and $A$ admits bounded $H^\infty(\mathbb C_+)$-calculus (so $f(A) \in L(X)$ for all $f \in H^\infty(\C_+)$).  Then pointwise convergence of $(f_n)_{n \ge 1}\subset H^{\infty}(\mathbb C_+)$ to $f$ on $\mathbb C_+$, together with $\sup_{n \ge 1}\|f_n\|_{\infty}<\infty$, implies strong convergence of $f_n(A)$ to $f(A)$.  Moreover, $f_n(A) \to f(A)$ in the operator norm if $f_n \to f$ in $H^\infty(\mathbb C_+)$.  The Convergence Lemma allows one to replace convergence of $(f_n)_{n \ge 1}$ in $H^\infty(\mathbb C_+)$ by convergence of $(f_n)_{n \ge 1}$ in a weaker topology at the price of getting only strong convergence of $(f_n(A))_{n \ge 1}$. However, such convergence often suffices in various applications.  See \cite{HaaseB} and \cite{KuWe} for fuller discussions of that, and also for variants of the Convergence Lemma for other types of operators, including the situation when the $H^\infty(\mathbb C_+)$-calculus for $A$ is unbounded.
Several instances of applications of the Convergence Lemma for sectorial or half-plane operators can be found in Section \ref{compat2} of this paper.

Since the $\Bes$-calculus is compatible with the sectorial and half-plane calculi the Convergence Lemmas for those calculi can be applied to the $\Bes$-calculus, when the assumptions of Proposition \ref{sector_comp}, Proposition \ref{compat3} or Proposition \ref{compat} hold uniformly for the approximating sequence.  In particular all the functions must be holomorphic on the same open sector or half-plane.

Corollary \ref{concor} below is a Convergence Lemma which is specific to the $\Bes$-calculus.  It applies to operators whose spectrum may be as large as $\overline \C_+$ and functions which may not extend beyond $\overline\C_+$, and pointwise convergence of $(f_n)_{n\ge1}$ is assumed only on $\overline\C_+$.  This is compensated by adding an assumption on the behaviour of $(f_n')_{n \ge 1}$ near the imaginary axis.   We shall deduce the Convergence Lemma from the following result about convergence in operator norm.\footnote{A helpful suggestion from a referee led us to this result.}

\begin{thm}\label{conlem}
Let $A$ be a densely defined
operator on a Banach space $X$ such that $\sigma(A)\subset \overline{\C}_{+}$ and \eqref{8.1} holds.
Let $(f_n)_{n\ge 1}\subset \mathcal{B}$ be such that
$
\sup_{n\ge 1}\,\|f_n\|_{\mathcal{B}}<\infty.
$
Assume that for every $z\in \C_{+}$ there exists
\begin{equation}\label{Conv21}
f_0(z):=\lim_{n\to\infty}\,f_n(z) \in \C,
\end{equation}
and for every $r>0$ one has
\begin{equation}\label{gz}
\lim_{\delta\to 0+}\,\int_0^\delta \sup_{|\beta|\le r}\,|f_n'(\alpha+i\beta)|\,d\alpha=0,
\end{equation}
uniformly in $n$. Let $g \in {\rm Hol}\,(\mathbb C_+)$ be such that $g'\in H^1(\mathbb C_+)$ and $g(\infty)=0$,
and let
$
g_n(z):=f_n(z)g(z), n\ge 0.
$
Then $f_0\in \mathcal{B}$, $g_n\in \mathcal{B}_0, n \ge 0,$ and
\[
\lim_{n\to\infty}\,\|g_n(A)-g_0(A)\|=0.
\]
\end{thm}

\begin{proof}    By Proposition \ref{balg}(1),  $f_0\in \mathcal{B}$.   Replacing $f_n$ by $f_n-f_0,$ we may assume that $f_0=0$.   By Proposition \ref{prim},  $g \in \mathcal B_0$, and then $g_n\in \mathcal{B}_0, n \ge 0$.

Let $x\in X$, $x^{*}\in X^{*}$ with $\|x\|=\|x^{*}\|=1$, and $n\ge 1$.
Writing $z=\alpha+i\beta$ and $dS(z)$ for area measure on $\C_{+}$, we obtain from the definition \eqref{fcdef} of $g_n(A)$ that
\begin{align*}
\langle g_n(A)x,x^{*}\rangle&=-\frac{2}{\pi}\int_{\C_{+}}\alpha \langle (\overline{z}+A)^{-2}x,x^{*} \rangle f_n'(z)g(z)\,dS(z)\\
&\phantom{XX} -\frac{2}{\pi}\int_{\C_{+}}\alpha \langle (\overline{z}+A)^{-2}x,x^{*} \rangle f_n(z)g'(z)\,dS(z)\\
&=I_n+J_n.
\end{align*}
Recall that if \eqref{8.1} holds, then $-A$ generates a $C_0$-semigroup $(e^{-tA})_{t \ge 0}$ in $X$ such that  $K_A=\sup_{t \ge 0}\|e^{-tA}\|<\infty$.
Using \eqref{8.1},  \eqref{8.2}, \eqref{gammaa}, and  a Hille-Yosida estimate for $-A$,
for fixed $\delta\in (0,1)$,  we have
\begin{align*}
|J_n| &\le \frac{2\sup_{n \ge 1}\|f_n\|_\infty}{\pi}\left(\int_0^\delta+\int_{1/\delta}^\infty\right)
\alpha \int_{-\infty}^\infty|\langle (\overline{z}+A)^{-2}x,x^{*}\rangle|\,|g'(z)|\,dS(z)\\
& \phantom{XX} +\frac{2}{\pi}\int_{\delta}^{1/\delta}\int_{\R}\alpha |\langle (\overline{z}+A)^{-2}x,x^{*}\rangle|\,|f_n(z)||g'(z)|\,dS(z)\\
&\le  \gamma_A\sup_{n \ge 1}\|f_n\|_\infty \left(\int_0^\delta+\int_{1/\delta}^\infty\right)
\sup_{\beta\in \R}\,|g'(\alpha+i\beta)|\,d\alpha\\
& \phantom{XX} +\frac{2K_A}{\pi}\int_{\delta}^{1/\delta}{\alpha}^{-1}\int_{\R} |f_n(\alpha+i\beta)||g'(\alpha+i\beta)| \, d\beta\, d\alpha.
\end{align*}
Since $g \in \mathcal B$, the first two integrals in brackets above tend to 0 as $n\to\infty$.
Since $f_n(z)\to 0$ as $n\to\infty$ for any $z\in \C_{+}$, and $g' \in H^1(\mathbb C_+)$,  the third integral above  tends also to 0 as $n\to\infty$, by the  dominated convergence theorem.
So $J_n\to 0$ as $n\to \infty$, uniformly for $x$ and $x^*$ with $\|x\|=\|x^{*}\|=1$.

For $r>0$ let $\mathbb D_r=\{z\in \C_{+}:\,|z|<r\}$. For $I_n$, we consider separately
the integrals over $\C_{+}\setminus \mathbb D_r$ and over $\mathbb D_r$, where $r$ is to be chosen.
We consider first
\[
K_1(r,n):=-\frac{2}{\pi}\int_{\C_{+}\setminus \mathbb D_r}\alpha \langle (\overline{z}+A)^{-2}x,x^{*} \rangle f_n'(z)g(z)\,dS(z).
\]
Using \eqref{8.2} again,
\[
|K_1(r,n)|\le \gamma_A \sup_{|z|\ge r}\,|g(z)| \sup_{n \ge 1} \|f_n\|_{\mathcal B},
\]
where the right-hand side goes to zero as $r \to \infty$ by the assumption that $g(\infty)=0$ and  Proposition \ref{prim}.

Let $\ep>0$. We may take $r>0$ so large that $$|K_1(r,n)|\le \ep$$
for all $n\in \N$ and all $x\in X$, $x^{*}\in X^{*}$
with $\|x\|=\|x^{*}\|=1$. Now consider
\[
K_2(r,n):=-\frac{2}{\pi}\int_{\mathbb D_r}\alpha \langle (\overline{z}+A)^{-2}x,x^{*}\rangle f_n'(z)g(z)\,dS(z).
\]
We have
\begin{align*}
|K_2(r,n)| &\le \frac{2 \|g\|_\infty}{\pi}\int_0^r \sup_{|\b|\le r}\,|f_n'(\a+i\b)| \,\alpha \int_{-r}^r |\langle(\overline{z}+A)^{-2}x,x^{*}\rangle|\,d\beta\, d\alpha\\
&\le \gamma_ A \|g\|_\infty \int_0^r \sup_{|\beta|\le r}\,|f_n'(\alpha+i\beta)|\, d\alpha\\
&= \gamma_A \|g\|_\infty \left(\int_0^\delta \sup_{|\beta|\le r}\,|f_n'(\alpha+i\beta)|\, d\alpha+
\int_\delta^r \sup_{|\beta|\le r}\,|f_n'(\alpha+i\beta)|\, d\alpha\right).
\end{align*}
By the assumption \eqref{gz}, there exists $\delta>0$ such that the first integral in the last line
of the display is less than $\ep$ for all $n\in \N$. By Vitali's theorem, $(f_n')_{n\ge 1}$ converges uniformly to zero on compact subsets of $\C_{+}$.
Thus, for fixed $r$ and $\delta$, the second integral converges to zero as $n\to\infty$. Let $N=N(r,\delta)$ be such that the second integral is less than $\ep$ for all $n\ge N$.
Then, summing up the estimates for $K_1(r,n)$ and $K_2(r,n),$ one has
\[
|I_n|\le \ep(1 + 2 \gamma_A \|g\|_\infty)
\]
for all $n\ge N$, and the estimate is uniform for $x$ and $x^*$ with  $\|x\|=\|x^{*}\|=1$.
It follows that $\|g_n(A)\|\to 0$ as $n\to\infty$, as required.
\end{proof}

\begin{cor} \label{concor}
Let $A$, $f_n$ and $f_0$ be as in Theorem \ref{conlem}.   Then,  for every $x\in X$,
\[
\lim_{n\to\infty}\,\|f_n(A)x-f_0(A)x\|=0.
\]
\end{cor}

\begin{proof}
We assume that $f_0=0$.  Let $g(z)=(1+z)^{-1}$.  By Theorem \ref{conlem}, $\lim_{n\to\infty} \|f_n(A)(1+A)^{-1}\| = 0$.  It follows that $\lim_{n \to \infty} f_n(A)x = 0$, for each $x \in D(A)$.
By the boundedness of the $\mathcal B$-calculus we have $\sup_{n\ge 1} \|f_n(A)\| <\infty$.    Since $D(A)$
is dense in $X$, we have $\lim_{n\to \infty} f_n(A) = 0$ strongly.
\end{proof}

\begin{rems}  1.  The assumption (\ref{gz}) in Theorem \ref{conlem} implies that $f_n(i\beta)=\lim_{\a\to0+}\,f_n(\a+i\beta)$ uniformly in $n$, and hence (\ref{Conv21}) holds also for $z\in i\R$.   On the other hand, the assumptions do not imply that $f_0(\infty) = \lim_{n\to\infty} f_n(\infty)$, as Example \ref{cvex}a) below shows.

\noindent 2.  When $g(z) = (1+z)^{-1}$, the estimation of $J_n$ in the proof of Theorem \ref{conlem} can be simplified.   Applying \eqref{boves} with $g = g_{x,x^*}$ and $h(z) = f_n(z)(1+z)^{-2}$, and using $|1+z| \ge 1+\alpha$ and \eqref{gammaa}, we have
\begin{align*}
|J_n| &\le \frac{2}{\pi}\int_{\mathbb C_{+}} \frac{\alpha}{(1+\alpha)^{3/2}}
 \left|\langle(\overline{z}+A)^{-2}x, x^*\rangle\right| \left|\frac{f_n(z)}{(1+z)^{1/2}}\right| \,dS(z) \\
&\le 2 \gamma_A \int_0^\infty (1+\alpha)^{{-3/2}}\sup_{\Re w = \a} \left| \frac{f_n(w)}{(1+w)^{1/2}}\right|\,d\a.
\end{align*}
The dominated convergence theorem can be applied to this integral.
\end{rems}

Now we illustrate Corollary \ref{concor} for two natural choices of $(f_n)_{n \ge 1}$.

\begin{exas} \label{cvex}  Let $f \in \Bes$ and let $A$ be densely defined and satisfy \eqref{8.1}.

\noindent a) Consider $f_n(z) = f(z/n) \in \Bes,\, n \ge 1$.
Then $\|f_n\|_\mathcal B=\|f\|_\mathcal B$ (see Lemma \ref{shifts01}), and $f_n(z) \to f(0)$ as $n \to \infty$, for each $z \in \overline{\mathbb C}_+$.
Note that
\begin{multline*}
\lim_{\delta \to 0+} \int_{0}^{\delta}\sup_{|\beta|\le r}|f'_n(\alpha +i\beta)|\,d\alpha
=\lim_{\delta \to 0+}\int_{0}^{\delta}\frac{1}{n}\sup_{|\beta|\le r}|f'(\alpha/n +i\beta/n)|\,d\alpha\\
\le \lim_{\delta \to 0+}\int_{0}^{\delta} \sup_{\beta\in \mathbb R}|f'(\alpha +i\beta)|\,d\alpha
=0,
\end{multline*}
by the monotone convergence theorem.
Thus the conditions of Theorem \ref{conlem} hold, and $f_n(A) \to f(0)I$ strongly as $n \to \infty$.
When $f(z)=e^{-z}$ this conclusion agrees with the fact that $(e^{-tA})_{t \ge 0}$ is strongly continuous
at zero. Note that this example also shows that one cannot in general replace strong convergence in Corollary \ref{concor} by convergence in operator norm.  Indeed, $A$ is bounded if $e^{-A/n}\to I$ in $L(X)$.

\noindent b) Now let $f_n(z) = f(nz) \in \Bes, n \ge 1$. Then as above,
$\|f_n\|_\mathcal B=\|f\|_\Bes$, and $f_n(z) \to f(\infty)$ as $n \to \infty$, for each $z \in \C_+$.  Observe that
\begin{align*}
 \int_{0}^{\delta}\sup_{|\beta|\le r}|f'_n(\alpha +i\beta)|\,d\alpha
\ge \int_{0}^{\delta} n |f'(n\alpha)|\,d\alpha
=\int_{0}^{\delta n} |f'(\alpha)|\,d\alpha
= \int_{0}^{1} |f'(\alpha)|\,d\alpha
\end{align*}
if $\delta=1/n$. Thus the condition (\ref{gz}) of Theorem \ref{conlem} is violated.  When $f(z)=e^{-z}$, then $f_n(z)$ does not converge to $f(\infty)=0$ for $z \in i\R$ and $(f_n(A))_{n \ge 1}$ does not necessarily converge strongly to zero.
\end{exas}

\subsection{Spectral inclusion and mapping} \label{spinc}

With a few exceptions, given a semigroup generator $-A$, the spectral ``mapping'' theorem
for a functional calculus $\Upsilon_A$ usually takes the form of the spectral inclusion $f(\sigma(A))\subset \sigma(\Upsilon_A(f))$.
In general, equality fails dramatically, even for the very natural functions  $e^{-tz}$
and for rather simple operators $A$; see \cite[Section IV.3]{EN}, for example.
While one may expect only the spectral inclusion as above,
 the equality $f(\sigma(A))\cup \{f(\infty)\} = \sigma(\Upsilon_A(f))\cup \{f(\infty)\}$ sometimes holds if $A$ inherits some properties of bounded operators such as sharp resolvent estimates.
The statement below shows that the $\mathcal B$-calculus
possesses the standard spectral features.

\begin{thm}  \label{SIT}
Let $A$ be a densely defined operator on a Banach space $X$ such that $\sigma(A) \subset \overline{\C}_+$ and {\rm(\ref{8.1})} holds.  Let $f \in\Bes$ and $\l \in \C$.
\begin{enumerate} [\rm1.]
\item  If $x \in D(A)$ and $Ax = \l x$, then $f(A)x = f(\l)x$.
\item  If $x^* \in D(A^*)$ and $A^*x^* = \l x^*$, then $f(A)^*x^* = f(\l)x^*$.
\item  If $\l \in \sigma(A)$ then $f(\l) \in \sigma(f(A))$.
\item  If $A \in \Sect(\pi/2-)$, then $\sigma(f(A)) \cup \{f(\infty)\} = f(\sigma(A)) \cup \{f(\infty)\}$.
\end{enumerate}
\end{thm}

\begin{proof}
1.  From (\ref{fcdef}) and Proposition \ref{BHP2}, for all $x^* \in X^*$, we have
\begin{align*}
\langle f(A)x,x^* \rangle &= f(\infty) \langle x,x^* \rangle - \frac{2}{\pi} \int_0^\infty \int_\R \frac{\a \langle x,x^* \rangle}{(\a-i\b +\l)^2}f'(\alpha+i\beta) \,d\b\, d\a \\
&= f(\l) \langle x,x^* \rangle.
\end{align*}

\noindent 2.  This is similar to (1).

\noindent 3.   We shall follow the Banach algebra method used in \cite[Section 16.5]{HP} and \cite[Section 2.2]{Dav80}.   We may assume without loss of generality that $f(\infty)=0$.  Let $\mathcal A$ be the bicommutant of $\{(\l-A)^{-1} : \l \in \rho(A)\}$ in $\B(X)$, so the spectrum of $f(A)$ in $\mathcal{A}$ coincides with the spectrum in $\B(X)$.

First, assume that the resolvent of $A$ is bounded on the left half-plane.  By the resolvent identity,  $\|(z+A)^{-2}(1+A)^{-2}\| \le C(1+|z-1|)^{-2}$ for $z \in \C_+$.  Then we may obtain the following formula in which the integral is absolutely convergent in the operator norm:
\begin{equation} \label{farr}
f(A)(1+A)^{-2} = - \frac{2}{\pi} \int_0^\infty \int_\R \a (\a-i\b+A)^{-2} (1+A)^{-2} f'(\a+i\b) \,d\b \,d\a.
\end{equation}
Applying (\ref{fcdef}) with $x$ replaced by $(1+A)^{-2}x$ gives a weak form of (\ref{farr}), and then this strong form follows.

Let $\l \in \sigma(A)$.  Then $(1+\l)^{-1} \in \sigma((1+A)^{-1})$, so there exists a character $\gamma$ of $\mathcal{A}$ such that $\gamma((1+A)^{-1})= (1+\l)^{-1}$.  It follows from the resolvent identity that $\gamma((z+A)^{-1}) = (z+\l)^{-1}$ whenever $-z \in \rho(A)$.  Applying $\gamma$ to (\ref{farr}) and using Proposition \ref{BHP2}, we obtain
\begin{multline*}
\gamma(f(A)) (1+\l)^{-2} \\
= - \frac{2}{\pi} \int_0^\infty \int_\R \a (\a-i\b+\l)^{-2} (1+\l)^{-2} f'(\a+i\b) \,d\b \, d\a = f(\l) (1+\l)^{-2}.  
\end{multline*}
Thus $f(\l) = \gamma(f(A)) \in \sigma(f(A))$.

For the general case, we proceed as follows.  If $\ep>0$, $\l+\ep \in \sigma(A+\ep)$.  Applying the case above for $A+\ep$ (noting that $\mathcal A$ and $\gamma$ do not change), $f(\l+\ep) = \gamma(f(A+\ep))$.  Now $\|f(A+\ep) - f(A)\| \to 0$ as $\ep\to0+$, by Lemma \ref{shifts01} and Lemma \ref{shifts2}.   So $f(\l) = \gamma(f(A)) \in \sigma(f(A))$.

\noindent 4.  Let $A \in \operatorname{Sect}(\pi/2-)$, let $\mathcal A$ be as above, let $\gamma$ be any character of $\mathcal{A}$, and let $f \in \Bes$ with $f(\infty)=0$.  Applying $\gamma$ to (\ref{V32}) gives
\[
\gamma(f(A)) = - \frac{2}{\pi} \int_0^\infty \int_\R \a \gamma\left((\a-i\b+A)^{-1}\right)^2 f'(\a+i\b) \,d\b\, d\a.
\]
If $\gamma((1+A)^{-1})=0$, then $\gamma((z+A)^{-1}) = 0$ for all $z \in \C_+$, and then $\gamma(f(A)) = 0 = f(\infty)$.  Otherwise, there exists $\mu \in \sigma(A)$ such that $\gamma((z+A)^{-1}) = (z+\mu)^{-1}$, and then we obtain
\[
\gamma(f(A)) = f(\mu) \in f(\sigma(A)).  \qedhere
\]
\end{proof}

\begin{rem}
It may be possible to show that approximate $\l$-eigenvectors for $A$ are approximate $f(\l)$-eigenvectors for $f(A)$.  This would give a more direct proof of Theorem \ref{SIT}(3) and provide additional insight into the fine structure of $\sigma (f(A))$.  However this approach is not straightforward and we leave it as a topic for further research.
\end{rem}

\section{Applications of the $\Bes$-calculus for operator norm-estimates}  \label{apps}

In this section we apply the $\Bes$-norm estimates obtained in Section \ref{subss} to the $\Bes$-calculus defined in Section \ref{b-fc}, and we recover (and sometimes improve) various results in semigroup theory.  Recall the definitions of $\gamma_A$, $K_A$ and $M_A$ in \eqref{gammaa}, \eqref{ka} and \eqref{R11}, and the inequalities in \eqref{gaka} and \eqref{V30}.   We generally give the estimates in the form which arises from using the norm $\|\cdot\|_\Bes$, but we note that in Corollaries \ref{Vthm}--\ref{mbounded}, the relevant functions belong to $\Bes_0$, so the estimates could be slightly improved by using \eqref{const}.

\subsection{Functions in spectral subspaces}

The following result was given in \cite[Theorem 1.1]{V1} with a proof using the Littlewood-Paley decomposition and with a less precise norm-estimate.   The proof is immediate from Theorem \ref{besc}, Lemma \ref{L1} and Lemma \ref{EsRR}.

\begin{cor} \label{Vthm}
Let $A \in \operatorname{Sect}(\pi/2-)$, and let $f \in H^\infty[\ep,\sigma]$, where $0 < \ep < \sigma < \infty$.  Then
\[
\|f(A)\| \le 4(2\pi+3\log2)M_A\left(\log M_A + 1 \right) \left(1 + 4 \log\left(1 + \frac{\sigma}{\ep}\right) \right) \|f\|_\infty.
\]
\end{cor}

The following corollary covers some ideas which are included in \cite[Theorem 5.5.12, etc]{White} using Peller's method,  and in \cite{Haase} using transference methods.  The proof is immediate from  Theorem \ref{besc}, Lemma \ref{L1} and  (\ref{gaka}).

\begin{cor} \label{Hthm}
Let $-A$ be the generator of a bounded $C_0$-semigroup on a Hilbert space, and let $f \in H^\infty[\ep,\sigma]$, where $0 < \ep < \sigma < \infty$.  Then
\[
\|f(A)\| \le 2 K_A^2 \left(1 + 2 \log\left(1 + \frac{2\sigma}{\ep}\right) \right) \|f\|_\infty.
\]
\end{cor}

\subsection{Holomorphic extensions to the left}
Here we show how results of Haase and Rozendaal \cite[Theorem 1.1]{HRoz} (and a preliminary result of Zwart \cite{Zwart}) for bounded semigroups on Hilbert space are corollaries of results in Section \ref{subss}.

Corollary \ref{CorH1} coincides with \cite[Theorem 1.1(a)]{HRoz}, and a similar result for bounded holomorphic semigroups on Banach spaces was given by Schwenninger \cite{Sch1}.   The result is an immediate consequence of Lemma \ref{H2}(\ref{exp}), because $H^\infty[\tau,\infty) \cap H^\infty_\omega = e_{\tau} H^\infty_\omega$, by (\ref{expdef}).

\begin{cor}\label{CorH1}  Let
$-A$ be the generator of a bounded $C_0$-semigroup on a Hilbert space, let $g\in H^\infty_\omega$, where $\omega>0$, and let $\tau>0$.  Then
\[
\|g(A)e^{-\tau A}\|\le 2 K_A^2 \left( 2+ \frac{1}{2}\log\left(1+\frac{1}{\omega\tau}\right)\right)\|g\|_{H^\infty_\omega}.
\]
\end{cor}

Corollary \ref{CorH2} below coincides with \cite[Theorem 1.1(b)]{HRoz}.  For $g \in H^\infty_\omega$, $g(A)$ is a closed operator defined by the half-plane calculus.

\begin{cor} \label{CorH2}
Let $-A$ be the generator of a bounded $C_0$-semigroup on a Hilbert space.
Let $\omega>0$, $\a>0$, $\lambda \in \C_+$, and let $g \in H^\infty_\omega$.  Then $g(A)(\l+A)^{-\a} \in L(X)$, and
\[
\|g(A)(\lambda+A)^{-\alpha}\| \le \left(4+\frac{1}{\alpha}\right)\frac{1}{m^\alpha} K_A^2 \|g\|_{H_\omega^\infty}
\]
for all $g \in H^\infty_\omega$, where $m:=\min(\omega,{\rm Re}\,\lambda)$.
\end{cor}

\begin{proof}
For fixed $\alpha>0$ and $\lambda\in \C_{+},$ let
\[
f(z)=(\lambda+z)^{-\alpha},\qquad \varphi(x)=({\rm Re}\,\lambda+x)^{-\alpha}, \qquad z \in \mathbb C_+, \quad x>0.
\]
We  have
\[
\|f\|_{\mathcal{B}}\le \frac{2}{({\rm Re}\,\lambda)^\alpha}.
\]
Note that $f(A) = (\lambda + A)^{-\a}$, by compatibility of the $\Bes$-calculus with the HP-calculus, and $(fg)(A) = g(A)f(A)$ by \cite[Theorem 1.3.2c)]{HaaseB} for the half-plane calculus.
Then, applying Lemma 3.2(1),
we obtain
\begin{align*}
\|gf\|_{\mathcal{B}}
&\le \|g\|_{H_\omega^\infty}\left(\frac{2}{({\rm Re}\,\lambda)^\alpha}+\frac{1}{2}
\int_0^\infty \frac{dx}{(\omega+x)({\rm Re}\,\lambda+x)^\alpha}\right)\\
&\le \|g\|_{H_\omega^\infty}\left(\frac{2}{({\rm Re}\,\lambda)^\alpha}+\frac{1}{2}
\int_0^\infty \frac{dx}{(m+x)^{1+\alpha}}\right)\\
&\le  \|g\|_{H_\omega^\infty}\left(2+\frac{1}{2\alpha}\right)\frac{1}{m^\alpha}.
\end{align*}

So,
\[
\|g(A)(\lambda+A)^{-\alpha}\|\le \left(4+\frac{1}{\alpha}\right)\frac{1}{m^\alpha}K_A^2\|g\|_{H_\omega^\infty}.
\]
\end{proof}

Corollary \ref{mbounded} was originally obtained in \cite[Theorem 7.1]{BHM} and then reproved in \cite[Theorem 1.1(c), Corollary 4.4]{HRoz}.   It follows directly from Lemma \ref{deriv}.

\begin{cor} \label{mbounded}
Let $-A$ be the generator of a bounded $C_0$-semigroup on a Hilbert space, and let $f \in H^\infty_\omega$, where $\omega>0$.   Then
\[
\|f'(A)\| \le \frac{3 K_A^2}{\omega} \|f\|_{H^\infty_\omega}.
\]
\end{cor}

Similar results to Corollaries \ref{CorH1}, \ref{CorH2} and \ref{mbounded} could be deduced in the same way for $A \in \operatorname{Sect}(\pi/2-)$.  However those results can be proved directly by defining $f(A)$, $f(A)(\l+A)^{-\a}$ and $f'(A)$ respectively by absolutely convergent integrals as used in the functional calculus for invertible sectorial operators \cite[Section 2.5.1]{HaaseB}.   The relevant functions $f(z)$, $f(z)(\l+z)^{-\a}$ and $f'(z)$ respectively all decay at least at a polynomial rate along rays  with arguments in $(-\pi/2,\pi/2)$ and the contour can pass to the left of $0$.  Then straightforward estimates produce results of this type.

\subsection{Exponentially stable semigroups}
Corollaries \ref{CorH1}, \ref{CorH2} and \ref{mbounded} can all be adapted to the case when $-A$ generates an exponentially bounded $C_0$-semigroup on a Hilbert space, so $\|e^{-tA}\| \le Me^{-\omega t}, \, t\ge0$.  In this situation, one may apply the corollaries above with $A$ replaced by $A - \omega$ and $f \in H^\infty(\C_+)$ replaced by $f(\cdot+\omega) \in H^\infty_\omega$.   For example, the conclusion of this version of Corollary \ref{CorH2} becomes that $\|f(A)(\lambda+A)^{-\a}\| \le CM^2\|f\|_\infty$ for all $f \in H^\infty(\C_+)$, where the function $f(A)$ may be defined by the half-plane calculus.

Instead of giving full details, we now give another result which we formulate for exponentially stable semigroups on Hilbert space and we give a proof which uses this technique in order to apply Lemma \ref{HZL1}.  It extends a result of Schwenninger and Zwart \cite{SZ} who considered the case when $|f(is)| \le (\log(|s|+e))^{-\alpha}$ for some $\alpha>1$.

\begin{cor}\label{CH}
Let $-A$ be the generator of an exponentially stable $C_0$-semigroup on a Hilbert space $X$, so that
\[
\|e^{-tA}\|\le Me^{-\omega t},\qquad t\ge 0,
\]
for some $M,\omega>0$.  Let  $f\in H^\infty(\mathbb C_{+})$, let
\[
h(t) = \operatorname{ess\,sup}_{|s|\ge t} |f(is)|,
\]
and assume that $h$ satisfies the assumption {\rm(\ref{HZLass})} of Lemma \ref{HZL1}.  Then
\[
\|f(A)\|\le 6 M^2  \int_0^\infty \frac{h(t)}{\omega + t} \, dt.
\]
\end{cor}

\begin{proof}
Let $g(z) = f(z+\omega)$.   By an elementary property of the half-plane calculus, $f(A) = g(A-\omega)$.  By Lemma \ref{HZL1}, $g \in \Bes$.  By compatibility of the two calculi (Proposition \ref{compat}) and Lemma \ref{HZL1}, we have
\[
\|f(A)\|=\|g(A-\omega)\|\le
6 M^2 \int_0^\infty \frac{h(t)}{\omega + t} \,dt.
\]
Here we have used that $\|e^{-t(A-\omega)}\|\le M$ and $\lim_{|s|\to\infty}f(is)=0$ , so $f(\infty)=0$.  We then applied (\ref{const}) with $\gamma_{A-\omega} = 2 M^2$.
\end{proof}

\subsection{Inverse generator problem} \label{igp}

Let $-A$ be the generator of a bounded $C_0$-semigroup, and assume that $A$ has dense range.  Then $A$ is injective, and we may consider the operator $A^{-1}$ whose domain is the range of $A$.  The longstanding inverse generator problem asks whether $-A^{-1}$ also generates a $C_0$-semigroup.  The problem was raised by de Laubenfels \cite{deL}, and he pointed out the simple positive solution in the case of a bounded holomorphic $C_0$-semigroup on a Banach space.  Then $A$ is sectorial of angle $\theta \in [0,\pi/2)$ and $A^{-1}$ is sectorial of the same angle.  A negative answer to the problem was given in \cite{Zw05} (and, implicitly, already in \cite{Komatsu}).   Moreover, essentially any growth of a semigroup $(e^{-tA})_{t \ge 0}$ rules out a positive answer, and the answer is also negative for bounded semigroups on $L^p$-spaces when $p \neq 2$; see \cite{GTZ}, or \cite{Fackler} for a somewhat simpler counterexample.

The answer to the problem for bounded semigroups on Hilbert space remains unknown. The question arises in control theory (see \cite{Zw05} for a discussion), but it is also very natural from the viewpoint of general theory of functional calculus.  It is known that if the answer is always positive, then the $C_0$-semigroup $(e^{-tA^{-1}})_{t \ge 0}$ is bounded. Thus long-time estimates of $(e^{-tA^{-1}})_{t \ge 0}$ when $A^{-1}$ does generate a $C_0$-semigroup are of value.    The case when $(e^{-tA})_{t \ge 0}$ is exponentially stable is of particular interest, since then an integral representation  for $(e^{-tA^{-1}})_{t \ge 0}$ in terms of  $(e^{-tA})_{t \ge 0}$ is available (see for example the recent survey article \cite[Corollary 3.5]{Sasha}). The relevant estimates become simpler (and sometimes sharper), and one might hope for a positive solution at least in that restricted setting.  A fuller and more detailed discussion of the problem can be found in \cite{Sasha}.

The problem is essentially the same as the question whether the operators $\exp(-tA^{-1})$ (defined by half-plane functional calculus) are bounded.  So we consider formally operators of the form $f_t(A)$ where $f_t(z) = \exp(-t/z)$ for $t>0$.  As noted in Section \ref{expinv}, these functions are not uniformly continuous on $\C_+$, and hence they are not in $\Bes$.   We can derive \cite[Theorem 2.2]{Zwart07}, the main result of \cite{Zwart07}, from the $\Bes$-calculus, up to multiplication by an absolute constant.

\begin{cor}\label{CH2}
Let $-A$ be the generator of an exponentially stable $C_0$-semigroup on a Hilbert space $X$, so there are $M, \omega>0$ such that
\[
\|e^{-tA}\|\le M e^{-\omega t},\qquad  t \ge 0,
\]
and there exists $A^{-1} \in L(X)$.  Then
\[
\big\|e^{-tA^{-1}}\big\|\le \begin{cases} 2M^2(2-e^{-t/\omega)}), \quad &t\in (0,\omega],\\
 2M^2(2-e^{-1}+e^{-1}\log (t/\omega)),\quad &t>\omega.
\end{cases}
\]
\end{cor}

\begin{proof}
We may assume that $\omega=1$ by replacing $A$ by $\omega^{-1}A - I$ and $t$ by $\omega^{-1}t$.  Then we consider the function $\tilde f_t(z) = \exp(-t/(z+1))$ as in Lemma \ref{HZL11}.   Then $\tilde f_t = e^{-tr_1} \in \Bes$, so boundedness of the $\Bes$-calculus gives
\[
\tilde f_t(A-I) = \sum_{n=0}^\infty \frac{(-t)^n}{n!} (r_1(A-I))^n
= \sum_{n=0}^\infty \frac {(-tA^{-1})^{n}}{n!} = e^{-tA^{-1}}.
\]
The result follows from Lemma \ref{HZL11}.
\end{proof}

It is possible to improve Corollary \ref{CH2} by replacing the assumption that the semigroup is exponentially stable by the weaker assumptions that the semigroup is bounded and $A$ is invertible.

\begin{cor} \label {HZCex}
Let $-A$ be the generator of a bounded $C_0$-semigroup on
Hilbert space $X$, and assume that $A$ has a bounded inverse.  Then
\[
\|e^{-tA^{-1}}\|\le C_A(1 + \log(1+t)), \qquad t \ge 0,
\]
where $C_A$ is a constant depending only on $A$.
\end{cor}

\begin{proof}
Let $h(z) = z(1+z)^{-1}$ and $g_t(z) = e^{-t/z}$, for $t>0$.  Then  $h^2 g_t$ is the function $f_t$ as in Lemma \ref{HZLex}, so $f_t \in \Bes$ and $\|f_t\|_\Bes \le C(1 + \log(1+t))$ for some constant $C$.   For $\ep>0$, let $f_{t,\ep}(z) = f_t(z+\ep)$ and $g_{t,\ep}(z) = g_t(z+\ep)$.  Then $g_{t,\ep} \in \Bes$ and $g_{t,\ep}(A) = e^{-t(A+\ep)^{-1}}$ which is given by a power series as in the proof of Corollary \ref{CH2}. Using Lemma \ref{shifts2}, we have
\[
f_{t,\ep}(A) = f_t(A+\ep) = h(A+\ep)^2 g_{t,\ep}(A) = (A+\ep)^2(I+A+\ep)^{-2} e^{-t(A+\ep)^{-1}}.
\]
Letting $\ep\to0+$, using Lemma \ref{shifts01}(1) and noting that $A^{-1}$ is the limit in the operator norm of $(A+\ep)^{-1}$, we obtain
\[
f_t(A) = (A(I + A)^{-1})^2 e^{-tA^{-1}}.
\]
Hence
\[
\big\|e^{-tA^{-1}}\big\| = \big\|\left(I+A^{-1}\right)^2 f_{t}(A) \big\| \le  2 C K_A^2  \left\|(I+A^{-1})^2\right\| (1 + \log(1+t)).  \qedhere
\]
\end{proof}

\subsection{Cayley transforms}  \label{Cayley2}
If $-A$ is the generator of a bounded $C_0$-semigroup on a Banach space $X$
we let $V(A)$ be the Cayley transform $(A-I)(A+I)^{-1}$ of $A$.  Norm-estimates for powers of $V(A)$
are of substantial value in numerical analysis, for example in stability analysis of the Crank-Nicolson scheme,
see \cite{Hersh}, \cite{Thomee}, \cite{Crouzeix} and \cite{Palencia}.   The first bound  $\|V(A)\|=O(n^{3/2})$ was proved in \cite{Butzer}, where applications to ergodic theory were treated. In the setting of generators of bounded semigroups on Banach spaces, the optimal estimate $\|V(A)^n\|=O(n^{1/2})$  was obtained in \cite{Thomee} by means of the HP-functional calculus.  This settled  a conjecture by Hersh and Kato \cite{Hersh} who gave a slightly worse bound.  Since $V(A) = \chi^n(A)$, where $\chi(z) = (z-1)(z+1)^{-1}$, the optimality of this bound implies that the HP-norm of $\chi^n$ grows like $n^{1/2}$ (see \cite[Section 9]{Sasha}).

Research on asymptotics of $V(A)$ continued in a number of subsequent works, see the survey \cite{Sasha} and the references therein.  By means of the holomorphic functional calculus, it was proved  in \cite{Crouzeix} and \cite{Palencia} (see also \cite{Azizov}), that $V(A)$ is power bounded if $A$ generates a bounded holomorphic semigroup on $X$.

The corresponding question for bounded semigroups on Hilbert spaces has proved to be be more demanding.
While it is evident that $V(A)$ is a contraction (so power bounded) if $A$ is the generator of a contraction $C_0$-semigroup on a Hilbert space $X$, it is still unknown  whether $V(A)$ is power bounded whenever $A$ generates a bounded $C_0$-semigroup on $X$.  The  logarithmic estimate $\|V(A)^n\|=O(\log n)$  proved in \cite{Go04} remains the best so far in that case (see also \cite{GoZwBe11}).
The inverse generator problem and the problem of power boundedness of $V(A)$ are strongly related (essentially equivalent) to each other.  In particular, if $A$ and $A^{-1}$ are both generators of bounded $C_0$-semigroups on a Hilbert space, then $V(A)$ is power bounded (see \cite{Azizov}, \cite{Guo} and \cite{Go04}).   We refer the reader to \cite{Sasha} for more details.

The proof that $\|V(A)^n\|=O(\log n)$ in \cite{Go04} was based on intricate estimates of Laguerre polynomials.
Here we obtain the logarithmic estimate as a direct and elementary application of the $\mathcal B$-calculus and Lemma \ref{G}.

\begin{cor}\label{G1}
Let $-A$ be the generator of a bounded  $C_0$-semigroup on
Hilbert space $X$.  Then
\[
\|V(A)^n\|\le 2 K_A^2\left( 3+2\log(2n)\right), \qquad n \in \mathbb N.
\]
\end{cor}

\subsection{Bernstein functions}

Sharp resolvent estimates for Bernstein functions are valuable
in the Bochner-Phillips theory of subordination  of $C_0$-semigroups,
and in probability theory.
In the 1980s,  Kishimoto and Robinson  asked whether subordination
preserves holomorphy of semigroups, or in other words
whether Bernstein functions transform $\Sect(\theta)$ into itself, for  $\theta \in [0,\pi/2)$.
The positive answer in full generality was first obtained in \cite{GT15}, and then the result was reproved in \cite{BGT17} within a wider framework of functional calculus for Nevanlinna-Pick functions.
The $\mathcal B$-calculus offers a new streamlined and transparent proof of the result in the case of injective operators, as we show below.

If $A \in \Sect(\pi/2-)$ and $f$ is a Bernstein function, then $f(A)$ can be defined either by the sectorial calculus or by the Bochner--Phillips calculus, without ambiguity \cite[Proposition 3.6]{GT15}.

\begin{cor}
Let $X$ be a Banach space, $A \in \Sect(\omega)$ for some $\omega \in [0,\pi/2)$, and $A$ be injective.  Let $f$ be a Bernstein function.   Then $f(A) \in \Sect(\omega)$.
\end{cor}

\begin{proof}
It is known that $-f(A)$ generates a bounded $C_0$-semigroup \cite[Theorem 13.6]{Schill}, so $f(A) \in \Sect(\pi/2)$.  We need to improve the angle.

We shall apply Proposition \ref{BB} with $\a \in (2\omega/\pi, 1)$, $\beta=1/\a$, and $g(z) = f(z^\a)^{1/\a}$ for $|\arg z|< \pi/(2\a)$.  Then $A^{1/\a}$ is sectorial of angle $\omega/\a \in (\omega, \pi/2)$.  Take $\l\in \Sigma_\theta$ where $\theta \in (0,\pi/2)$, and let $h(z) = (\l+g(z))^{-1}$ for $z \in \C_+$.  By Proposition \ref{BB} and boundedness of the $\Bes$-calculus in Theorem \ref{besc},
\[
\text{$h(A^{1/\a}) \in L(X)$ and $\|h(A^{1/\a})\| \le C_\theta/|\lambda|$.}
\]
Here $C_\theta$ depends on $\alpha$, $\theta$ and on $A$, but not on $f$, and not on $\l$ subject to $\l \in \Sigma_\theta$. By compatibility of the functional calculi (Proposition \ref{sector_comp}), $h(A^{1/\a})$ in the $\Bes$-calculus is the same operator as in the sectorial calculus.

Since $g(z^{1/\a}) = f(z)^{1/\a}$ for $z \in \Sigma_\omega$, the Composition Rule for sectorial operators \cite[Theorem 2.4.2]{HaaseB} tells us that $g(A^{1/\a}) = f(A)^{1/\a}$, and then by \cite[Theorem 1.3.2 f)]{HaaseB}
\[
h(A^{1/\a}) = (\l + g(A^{1/\alpha}))^{-1} = (\l + f(A)^{1/\alpha})^{-1}.
\]
The estimates above for $\|h(A^{1/\a})\|$ establish that $f(A)^{1/\a} \in \Sect(\pi/2)$.   Then $f(A) \in \Sect(\pi\a/2)$ whenever $\a \in (2\omega/\pi,1)$.  This implies that $f(A) \in \Sect(\omega)$.
\end{proof}

\section{Appendix: relation to Besov spaces} \label{appx}

The main purpose of this section is to explain the connection between the spaces $\Bes$ (and $\Bov$) (more precisely, $\Bq$ and $\Bov_0$) with (homogeneous, analytic) Besov spaces of functions on $\R$.  In Sections \ref{besov}-\ref{apps}, we have given a largely self-contained account of the construction of the $\Bes$-calculus and we have obtained all the relevant estimates, without any need for the theory of Besov spaces.  We used Arveson's theory of spectral subspaces in the proof of Proposition \ref{densehinf} which played an important role in the construction, and we shall use the result again to show that $\Bq$ can be identified with a Besov space by passing to the boundary functions (Proposition \ref{bid}).  Conversely, Proposition \ref{densehinf} follows very easily from Proposition \ref{bid}, as noted in \cite{V1}  (slightly incorrectly).  More generally, the relations of the spaces in this paper to Besov spaces are instructive and potentially crucial for further work, so we set them out here.

Apparently (as remarked in \cite[p.120]{BCD}) there is no consensus surrounding the definition of homogeneous Besov spaces, let alone the definition of homogeneous analytic Besov spaces.  The definitions given below reflect the aims of the paper, and they may differ slightly from other sources. Our aim is to study the  holomorphic functions from $\mathcal B$ (and $\mathcal E$) in terms of their boundary values, and to relate the boundary values to the Besov spaces defined on $\mathbb R$.  The theory of inhomogeneous Besov spaces is much better developed, but unfortunately it seems to be unsuitable for the study of spaces of holomorphic functions, since the spectrum of their elements does not split easily near zero.

Let
\[
\mathcal{S}_{\infty}(\mathbb R):= \left\{\varphi \in \mathcal{S} (\mathbb R): \int_{\mathbb R} t^n \varphi(t)\, dt=0, \,\, n \in \mathbb N \cup \{0\}\right\},
\]
with the topology induced from $\mathcal{S}(\mathbb R)$, and let $\mathcal{S}'_{\infty}(\mathbb R)$ be the topological dual of $\mathcal{S}_{\infty}(\mathbb R)$.  Consider a system of functions $(\psi_n)_{n \in \mathbb Z}\subset \mathcal S(\mathbb R)$ such that
\begin{itemize}
\item $\supp\psi_n \subset \left\{t\in\R: a2^{n-1}\le |t|\le b 2^{n}\right\}, \quad n \in \mathbb Z$, for some $a, b >0$,
\item $\sum_{n\in \mathbb Z} \psi_n(t)=1, \quad t \in \mathbb R \setminus \{0\}$,  and
\item $\sup_{t \in \mathbb R} \sup_{n \in \mathbb Z} 2^{kn} |\psi^{(k)}_n(t)|<\infty, \quad k\ge0$.
\end{itemize}
Let $$ \varphi_n =\mathcal{F}\psi_n, \qquad n \in \mathbb Z$$.
For $f \in S'_{\infty}(\mathbb R)$, set
$$ f_n = f*\varphi_n, \qquad n \in \mathbb Z.$$
For $1 \le p, q \le \infty$ and $s \ge 0$, define the \emph{homogeneous} Besov space $B^s_{p,q}(\mathbb R)$
as the set of $f \in S'_{\infty}(\mathbb R)$ such that
\begin{equation}\label{dyadnorm}
\|f\|^{{\rm dyad}}_{B^s_{p, q}(\R)}:=\left\|\left(2^{ns}\|f_n\|_{L^p(\mathbb R)}\right)_{n \in \mathbb Z}\right\|_{l^q(\mathbb Z)} <\infty.
\end{equation}
It is crucial that the definition of  $B^s_{p,q}(\mathbb R)$ does not depend of the choice of $(\psi_n)_{n \in \mathbb Z}$, up to equivalence of norms, and it is a Banach space for all $1 \le p, q \le \infty$ and $s \ge 0$.
Moreover,
$$
\mathcal{S}_{\infty}(\mathbb R) \subset B^s_{p,q}(\mathbb R) \subset \mathcal{S}'_{\infty}(\mathbb R)
$$
and, for every $f \in B^s_{p,q}(\mathbb R)$,
$$
f=\sum_{n \in \mathbb Z} {f_n}
$$
in $S'_{\infty}(\mathbb R)$.

It is traditional to write $\dot{B}^s_{p,q}(\mathbb R)$ instead of $B^s_{p,q}(\mathbb R)$ to distinguish homogeneous and inhomogeneous Besov classes.   Since we will not deal with inhomogeneous Besov spaces, we omit the dot.  See  \cite{BCD}, \cite{Graf}, \cite{Peetre}, \cite{Sawano} or \cite{TriebelF}  for discussions of other properties of homogeneous Besov spaces.

\begin{remark}
One can show that $\mathcal{S}'_{\infty}(\mathbb R)$ is topologically isomorphic to $\mathcal{S}'(\mathbb R)/\mathbb{P}$,  where $\mathbb{P}$ stands for the space of polynomials on $\mathbb R$.  For $p,q \ge 1$, the space $B^s_{p,q}(\mathbb R)$ can also be introduced by defining a norm on the equivalence class $[f]\in S'(\mathbb R)/\mathbb{P}$ as  $\|f + \mathbb{P}\|:=\|f\|^{{\rm dyad}}_{B^s_{p, q}(\R)}$.  So long as we are interested only in the ``smoothness index'' $s=0$ as in this paper,  we could use an alternative isomorphic pair, $\mathcal{S}_{0}(\mathbb R)/\C$ and $\mathcal{S}'_{0}(\mathbb R)$, where $\mathcal{S}'_{0}(\mathbb R)$ is the topological dual of
\[
\mathcal{S}_0(\mathbb R):=\left\{\varphi \in \mathcal S(\mathbb R): \int_{\mathbb R}\varphi (t)\, dt=0\right\}.
\]
See \cite[Proposition 8]{Bourdaud} for details of that.   However, we prefer to avoid dealing with quotient classes.
\end{remark}

A particularly convenient way to choose the functions $\psi_n$ is as follows.   Let $\psi \in C^{\infty}(\mathbb R)$ be such that
\begin{equation*}
\psi \ge 0, \qquad {\rm supp} \, \psi \subset [1/2, 2], \qquad \psi (x)=1-\psi(x/2), \,\, x \in [1,2].
\end{equation*}
Define
\[
 \psi_n(x) = \psi(2^{-n}|x|), \quad  \psi_n^+(x) = \psi(2^{-n}x), \quad \varphi_n^+ = \mathcal{F}\psi_n^+.
\]
Then $(\psi_n)_{n\in\Z}$ satisfy the required properties.  From now onwards, we assume that $\psi$ and $(\psi_n)_{n\in\Z}$ have been chosen in this way.

Define the (homogeneous) \emph{analytic Besov space} as
\[
B^{s+}_{p,q}(\mathbb R) = \left\{f \in B^{s}_{p,q}(\mathbb R) :\supp \fti f \subset \R_+\right\},
\]
and note that if $f \in B^{s+}_{p,q}(\mathbb R)$ then $f*\varphi_n=f*\varphi^+_n, \, n \in \mathbb Z$.  Since $B^{s+}_{p,q}(\mathbb R)$ is a closed subspace of $B^{s}_{p,q}(\mathbb R)$, it is a Banach space.
See also  \cite[Section 2]{APeller} and \cite[Section 3.2]{Ricci} for  related constructions.
Note that the Riesz projection, given as a Fourier multiplier with the symbol $1_{[0,\infty)}$, is bounded on ${B}^s_{p,q}(\mathbb R)$ for all $1\le p,q \le \infty$ and $s\ge0$; see \cite[Corollary 6.7.2]{Graf} or \cite{Peetre}.

Now we consider the analytic Besov space $\mathcal B_{\rm dyad}:=B^{0+}_{\infty, 1}(\mathbb R)$ equipped with the corresponding dyadic norm
$$
\|f\|_{\mathcal B_{\rm dyad}}:=\sum_{n \in \mathbb Z}\|f *\varphi^+_n\|_\infty<\infty.
$$
It is straightforward to show that $\mathcal B_{\rm dyad}\subset \operatorname{BUC} (\mathbb R)$,
and for every  $f \in \mathcal B_{\rm dyad}$,
\begin{equation*}
f=\sum_{n \in \mathbb Z} f *\varphi^+_n
\end{equation*}
in $\operatorname{BUC}(\mathbb R)$.
We will show that  $\mathcal B_{\rm dyad}$ can be identified with $\Bq$
by showing that $\mathcal B_{\rm dyad}$ arises as the space of boundary values of
functions in $\Bq$.

For $g \in \mathcal B_{\rm dyad}$ write its dyadic decomposition $g=\sum_{n \in \mathbb Z} g*\varphi^+_n$, and define the (holomorphic) extension mapping $\mathcal P$:
\begin{equation}\label{defp}
\mathcal P(g)(t+is):=(P_t *g)(s) =  \sum_{n \in \mathbb Z} P(t)(g*\varphi^+_n)(s), \qquad t \ge 0, \,\, s \in \mathbb R,
\end{equation}
where $P_t$ and $P(t)$, are the Poisson kernel and semigroup for the right half-plane $\C_+$, as in Lemma \ref{ho}.  The mapping is well-defined, and $\mathcal P (g) \in H^\infty (\mathbb C_+)$ since  $g \in H^\infty(\R)$.

\begin{prop}   \label{bid}
The linear operator $\mathcal P$ given by \eqref{defp} is an isomorphism of $\mathcal B_{{\rm dyad}}$ onto $\mathcal B_0$.   Moreover, $\mathcal B_{\rm dyad}=\{f^b: f \in \mathcal B_0\}$,
and for every $f \in \mathcal B_0$ one has
\begin{equation}\label{replaplace}
f = \mathcal{L}(\mathcal F^{-1} f^b)
\end{equation}
in the sense of distributions.
\end{prop}

\begin{proof}
Note that for $g \in \mathcal B_{\rm dyad}$,
\begin{equation}\label{equality}
\|\mathcal P(g)\|_{\Bq}=\int_{0}^\infty \|(\mathcal Pg)'(t+i\cdot)\|_{\infty}\, dt =
\int_{0}^\infty \left \|\frac{d}{dt} P(t)g\right \|_{\infty}\, dt,
\end{equation}
and by \cite[Corollary 3, p.285]{Triebel82} (see also \cite[Sections 5.2.3 and 2.12.2]{TriebelF}) there exists $C >0$ such that for every  $g \in \mathcal B_{\rm dyad}$,
\begin{equation} \label{eqnm}
C^{-1}\|g\|_{\mathcal B_{\text{dyad}}} \le \int_{0}^\infty \left \|\frac{d}{dt} (P_t*g)\right \|_{\infty}\, dt \le C \|g\|_{\mathcal B_{\text{dyad}}}.
\end{equation}
Thus $\mathcal  P$ yields an isomorphic embedding of $\mathcal B_{\rm dyad}$ into $\mathcal B_0$.

If $f \in \ssp$, then $f^b * \varphi_n^+ = 0$ for all except finitely many $n$, so $f_b \in \Bes_{\rm dyad}$ and $f = \mathcal{P}f^b$.  Since $\ssp$ is dense in $\mathcal B_0$ by Proposition \ref{densehinf},
we infer that $\mathcal{P}$ maps $\Bes_{\rm dyad}$ onto $\Bq$.   The inverse map is given by $f \mapsto f^b$, by well-known properties of the Poisson kernels (see Lemma \ref{ho}).

The property \eqref{replaplace} is standard, and it follows from the representation theory of holomorphic functions of slow growth as distributional Laplace transforms; see \cite[Section 9 and Section 12.2, Corollary 4]{Vladimirov}  or \cite[Theorem 1 and Corollary]{BeWo65}.
\end{proof}

\begin{remark}  The right-hand inequality in (\ref{eqnm}) may be deduced from Lemma \ref{L1}, since $\mathcal{P}(g*\varphi_n^+) \in H^\infty[2^{n-1},2^{n+1}]$.   Using  also the maximum principle, we obtain that, for every $g \in \mathcal B_{{\rm dyad}}$,
\[
\|\mathcal{P} (g * \varphi_n^+)\|_\Bq \le C \|g*\varphi_n^+\|_\infty,
\]
so
\[
\|\mathcal{P}g\|_\Bq = \Big\|\sum_{n\in\Z} \mathcal{P}(g*\varphi_n^+) \Big\|_\Bq \le C \sum_{n\in\Z} \|g*\varphi_n^+\|_\infty = C \|g\|_{\Bes_{\rm dyad}}.
\]
This argument may also be used for other similar spaces.
\end{remark}

\begin{remark}
The description of Besov spaces $B^s_{p,q}(\mathbb R)$ in terms of the boundary behaviour of the Poisson semigroup goes back to \cite{Taib} (for $s >0$); see also \cite{TriebelF} and \cite{Peetre}.
For the study of general Besov spaces in terms of spaces of holomorphic functions in $\mathbb C_+$,
it seems more natural to use the half-plane extension of Besov spaces by the Fourier-Laplace transform, as in \cite{Ricci} (for $s<0$). In this approach, one formally  defines the holomorphic extension $\mathcal I f$ of a Besov function $f=\sum_{n \in \mathbb Z}f*\varphi^+_n$ as
\begin{equation*}
(\mathcal I f)(z):= \mathcal L\big(\fti{ \sum_{n\in\Z} f_n}\big)(z)=
2\pi \sum_{n \in \mathbb Z}\mathcal L(\fti f \fti \varphi^+_n)(z), \quad z \in \mathbb C_+.
\end{equation*}
In particular, using the mapping $\mathcal I$, certain holomorphic Besov spaces in a quite advanced setting were identified in \cite{Ricci} with Bergman spaces of holomorphic functions in a half-plane.
We prefer to use the Poisson semigroup since it fits better to this paper, and requires minimal developments of the theory of Besov spaces.   An approach to Besov spaces on $\R$ as boundary values of spaces of harmonic functions on $\C_+$  can be found in \cite{Bui1} (where the spaces are called Lipschitz spaces).
\end{remark}

\begin{remark}
For constructions of functional calculi it is crucial to deal with Banach algebras of functions and to include the constant functions, so we have worked with the space $\mathcal B$ and the norm $\|\cdot\|_{\mathcal B}=\|\cdot\|_{\mathcal B_0}+\|\cdot\|_{\infty}$.  (Recall from Proposition \ref{besprop} that $\|\cdot\|_{\mathcal B_0}\ge \|\cdot\|_{\infty}$ on $\mathcal B_0$.)  The Banach algebra $(\mathcal B, \|\cdot\|_{\mathcal B})$ can be considered as a modified Besov space, but it can hardly be identified with the conventional homogeneous Besov spaces.  Thus the standard facts from the theory of Besov spaces (duality, isomorphisms, etc) cannot a priori be applied to $(\mathcal B, \|\cdot\|_{\mathcal B})$.  The algebra $\Bes$ has appeared, for example, in \cite{V1}, and similar algebras were considered in \cite{Moussai}.
\end{remark}

Now we briefly discuss the homogeneous class $\mathcal E_{{\rm dyad}}:=B^{0+}_{1,\infty}(\mathbb R)$
with the corresponding norm
$$
\|f\|_{\mathcal E_{\rm dyad}}=\sup_{n \in \mathbb Z} \|f *\varphi^+_n\|_1, \qquad f \in \mathcal E_{{\rm dyad}},
$$
and its connections to the space $\Bov$ of holomorphic functions $g$  on $\mathbb C_+$ satisfying (\ref{defe0}).

Recall from Section \ref{dual} that
$\mathcal E$ is a Banach space with the norm in (\ref{defe})
and the subspace  $\mathcal E_0$ of functions $g \in\mathcal E$ with $g(\infty)=0$ is also a Banach space.
The spaces $\Bov$ and $\Bes$ are related via the (partial) duality given in (\ref{dualdef}).
Consequently, as discussed in Section \ref{nondual}, $\mathcal E$  and $\mathcal B$ are contractively embedded into $\mathcal B_0^*$ and $\mathcal E_0^*$, respectively, but in each case the range is not norm-dense (Propositions \ref{nondens1} and \ref{nondens2}).  Nevertheless, it is natural to consider further the nature of $\mathcal E$ and any  relations to $\mathcal B$, in particular by looking at Besov classes that can be associated to $\mathcal E$.

Parts of the proof of Proposition \ref{bid} can be applied in the context of $\Bov$, taking account of the estimate in Proposition \ref{space_e}(3) and using the same results or arguments from \cite{Triebel82}, \cite{Vladimirov} and \cite{BeWo65}.   In this way one can see that $\mathcal E_{\rm dyad}$ is isomorphically embedded into $\mathcal E$ as a closed subspace,
every $f$ in the range of the embedding there exists a distributional boundary value $f^b$,
and $f$ is the Fourier-Laplace transform of $f^b$.  However, the shift semigroup $(T_{\mathcal E}(a))_{a \ge 0}$ is not strongly continuous on $\mathcal E$ (Lemma 2.18), so a result for $\Bov$ similar to Proposition \ref{densehinf} cannot be expected, and consequently it is not clear that the embedding is surjective.

One may also consider the subspace $\mathcal E_{0, {\rm dyad}}$ of $\mathcal E_{{\rm dyad}}$
defined by
\begin{equation*}
\lim_{|n| \to \infty} \|f * \varphi^+_n\|_1=0, \qquad f \in \mathcal E_{{\rm dyad}},
\end{equation*}
and the subspace $\tilde {\mathcal E}$ of $\mathcal E$ given by
\begin{equation*}
\lim_{x\to0+} x \int_\R |g'(x+iy)|\, dy = \lim_{x\to\infty} x \int_\R |g'(x+iy)|\, dy =0.
\end{equation*}
Then $\mathcal E_{0, {\rm dyad}}$ is embedded isomorphically in $\tilde {\mathcal E}$, as shown in statement (2) of \cite[p.266]{V1}.  It seems plausible that $\tilde {\mathcal E}$ is isomorphic to  $\mathcal E_{0, {\rm dyad}}$, and $\mathcal E_0$ is isomorphic to $\mathcal E_{{\rm dyad}}$ in the sense of Proposition \ref{bid}, and also that $\mathcal E_{0, {\rm dyad}}^*$ can be identified in a natural way with $\mathcal B_{\rm dyad}$ (see \cite[Theorem 10.7]{Lenz}, or \cite[p.335-337]{Ricci} where several arguments are true for a larger class of Besov spaces than claimed).  None of these identifications is needed for the results of this paper, and
we have not worked out proofs (or corresponding counterexamples).

\begin{remark}
A different approach to Proposition \ref{bid} (and also to the very definition of holomorphic Besov spaces)
 was proposed in \cite{V1}, but some of the arguments given in the Appendix of that paper (e.g., p.266, lines 16-18) appear to be incomplete.  It is also stated there (p.265, last line) that $\mathcal E_{0, {\rm dyad}}^*$ coincides with $\mathcal B_{\rm dyad}$, but no explanation is given.
\end{remark}

\end{document}